\documentclass{article}

\usepackage[T1]{fontenc}
\usepackage[ansinew]{inputenc}
\makeatletter
\DeclareRobustCommand*{\bfseries}{
	\not@math@alphabet\bfseries\mathbf
	\fontseries\bfdefault\selectfont
	\boldmath
}
\makeatother

\usepackage{relsize}
\usepackage{amsmath}
\usepackage{amssymb}
\usepackage{amsthm}

\usepackage{amsfonts}
\usepackage[colorlinks]{hyperref}

\hypersetup{
	linkcolor={red!70!black},
	citecolor={blue!50!black},
	urlcolor={blue!80!black}
}
\usepackage{float}
\usepackage[all]{hypcap}
\usepackage{paralist}
\usepackage{color}
\usepackage{mathtools}

\usepackage{pgf,tikz}
\usepackage{tikz-3dplot}
\usepackage[all,cmtip]{xy}
\usepackage{xparse}
\usepackage{units}
\usepackage{stackrel}
\usepackage{fancyvrb}
\fvset{fontsize=\normalsize}
\usepackage{multicol}

\usepackage{rotating}

\usepackage{graphicx}
\setkeys{Gin}{width=\linewidth,totalheight=\textheight,keepaspectratio}

\usepackage{booktabs}

\usepackage{xypic}
\usepackage{cleveref}
\usepackage{imakeidx}

\makeindex

\usepackage{multicol}
\usepackage{mathrsfs}

\usepackage{silence}

\usepackage[a4paper,bindingoffset=0.2in,
left=1in,right=1in,top=1in,bottom=1in,
footskip=.25in]{geometry}

\usepackage{tikz-cd}
\usepackage{tkz-euclide}

\usetikzlibrary{arrows}
\usetikzlibrary[patterns]
\usetikzlibrary{calc}
\usetikzlibrary{knots}
\usetikzlibrary{shapes.misc}

\usepackage{tikzit}

\tikzstyle{black dot}=[fill=black, draw=black, shape=circle, minimum size=3pt, inner sep=0pt]
\tikzstyle{black dot small}=[fill=black, draw=black, shape=circle, minimum size=3pt, inner sep=0pt]
\tikzstyle{big white circle}=[fill=white, draw=black, shape=circle, minimum width=0.75cm]
\tikzstyle{white dot big}=[fill=white, draw=black, shape=circle, inner sep=1pt]
\tikzstyle{white dot}=[fill=white, draw=black, shape=circle, minimum size=3pt, inner sep=0pt]
\tikzstyle{flat box}=[fill=white, draw=black, shape=rectangle, minimum width=2.5cm, minimum height=0.5cm]
\tikzstyle{square}=[fill=white, draw=black, shape=rectangle]
\tikzstyle{flat box 2}=[fill=white, draw=black, shape=rectangle, minimum height=0.5cm, minimum width=1.0cm]
\tikzstyle{triangle}=[fill=white, draw=black, shape=regular polygon, regular polygon sides=3, inner sep=0pt, minimum size=2pt]
\tikzstyle{up-triangle}=[fill=white, draw=black, shape=regular polygon, regular polygon sides=3, shape border rotate=180, inner sep=0pt, minimum size=3pt]

\tikzstyle{thick}=[-, line width=1.5pt]
\tikzstyle{mid arrow}=[-, postaction={on each segment={mid arrow}}]
\tikzstyle{end arrow}=[->, >=latex]
\tikzstyle{red mid arrow}=[-, draw={rgb,255: red,214; green,42; blue,51}, postaction={on each segment={mid arrow}}, line width=1pt]
\tikzstyle{blue}=[-, draw={rgb,255: red,23; green,37; blue,167}, line width=1pt, dashed]
\tikzstyle{blue mid arrow}=[-, draw={rgb,255: red,23; green,37; blue,167}, postaction={on each segment={mid arrow}}, line width=1pt]
\tikzstyle{semicircle line}=[-, postaction={on each segment={red semicircle}}]
\tikzstyle{black dashed}=[-, draw=black, dashed]
\tikzstyle{dashed black thick}=[-, draw=black, dashed, line width=1.2pt]
\tikzstyle{red}=[-, draw=red]

\usetikzlibrary{decorations.pathreplacing,decorations.markings}
\tikzset{
	on each segment/.style={
		decorate,
		decoration={
			show path construction,
			moveto code={},
			lineto code={
				\path [#1]
				(\tikzinputsegmentfirst) -- (\tikzinputsegmentlast);
			},
			curveto code={
				\path [#1] (\tikzinputsegmentfirst)
				.. controls
				(\tikzinputsegmentsupporta) and (\tikzinputsegmentsupportb)
				..
				(\tikzinputsegmentlast);
			},
			closepath code={
				\path [#1]
				(\tikzinputsegmentfirst) -- (\tikzinputsegmentlast); 
			},
		},
	},
	mid arrow/.style={postaction={decorate,decoration={
				markings,
				mark=at position .5 with {\arrow[#1]{stealth}}
	}}},
	red semicircle/.style={postaction={decorate,decoration={
				markings,
				mark=at position .65 with {
					\arrow[#1]{Circle[left,fill=red,length=6pt,width=6pt]}
				}
	}}},
}

\newtheoremstyle{important-thm}
{3pt}
{3pt}
{\slshape}
{}
{\bfseries}
{.}
{.5em}
{}

\theoremstyle{important-thm}

\newtheorem{thm}{Theorem}[subsection]
\newtheorem{theorem}[thm]{Theorem}
\newtheorem{lemma}[thm]{Lemma}
\newtheorem{lem}[thm]{Lemma}
\newtheorem{prop}[thm]{Proposition}
\newtheorem{proposition}[thm]{Proposition}
\newtheorem{cor}[thm]{Corollary}
\newtheorem{corollary}[thm]{Corollary}

\theoremstyle{definition}

\newtheorem{rmk}[thm]{Remark}

\newtheorem{defn}[thm]{Definition}
\newtheorem{rem}[thm]{Remark}

\newtheorem{ntt}[thm]{Notation}
\newtheorem{const}[thm]{Construction}

\newenvironment{myproof}[1][Proof]{\begin{proof}[\textsc{#1}]}{\end{proof}}

\DeclarePairedDelimiterX\setc[2]{\{}{\}}{\,#1 \;\delimsize\vert\; #2\,}

\newcommand{\SVect}{\operatorname{SVect}}

\renewcommand{\mod}{\operatorname{mod}}

\renewcommand{\dim}{\operatorname{dim}}

\newcommand{\R}{\mathbb{R}}

\newcommand{\Bord}[1]{{\mathcal{B}\hspace{-.5pt}ord^{\hspace{1pt} #1,\mathrm{oc}}}}
\newcommand{\Bordori}{{\mathcal{B}\hspace{-.5pt}ord^{\hspace{1pt} \mathrm{oc}}}}

\newcommand{\Kn}[2]{{\mathcal{K}\hspace{-.5pt}n\mathcal{F}\hspace{-.5pt}rob^{\hspace{1pt} #1}\hspace{-1.5pt}\left( #2 \right)}}
\newcommand\Cb            {\mathbb{C}}
\newcommand\Ib            {\mathbb{I}}

\newcommand\Rb            {\mathbb{R}}
\newcommand\Zb            {\mathbb{Z}}

\newcommand\Cc            {\mathcal{C}}

\newcommand\Sc            {\mathcal{S}}
\newcommand\Sb            {\mathbb{S}^1}
\newcommand\Zc            {\mathcal{Z}}

\newcommand{\Cl}{C\hspace*{-1pt}\ell}

\usetikzlibrary{matrix}
\usetikzlibrary{arrows}
\usetikzlibrary[patterns]
\usetikzlibrary{decorations.pathreplacing}
\usetikzlibrary{decorations.pathmorphing}
\usetikzlibrary{decorations.text}
\usetikzlibrary{decorations.markings}

\numberwithin{equation}{section}

\definecolor{Blue}  {rgb} {0.282352,0.239215,0.803921}
\definecolor{Green} {rgb} {0.133333,0.545098,0.133333}
\definecolor{Red}   {rgb} {0.803921,0.000000,0.000000}
\definecolor{Violet}{rgb} {0.580392,0.000000,0.827450}

\hypersetup{pdfborder = {0 0 0}}

\newcommand{\GL}{\mathrm{GL}}

\newcommand{\id}{\mathrm{id}}

\newcommand{\Fun}{\operatorname{Fun}}

\newcommand{\Set}{\operatorname{Set}}

\usepackage{graphicx}
\setkeys{Gin}{width=\linewidth,totalheight=\textheight,keepaspectratio}

\title{}
\author{}

\usepackage{booktabs}

\usepackage{units}

\usepackage{fancyvrb}
\fvset{fontsize=\normalsize}

\newcommand{\dpar}{\partial_{\mathrm{par}}}
\newcommand{\dfree}{\partial_{\mathrm{free}}}
\newcommand{\din}{\partial_{\mathrm{in}}}
\newcommand{\dout}{\partial_{\mathrm{out}}}
\newcommand{\Bin}{B_{\mathrm{in}}}
\newcommand{\Bout}{B_{\mathrm{out}}}
\def\on{\operatorname}

\hypersetup{
	pdfborder = {0 0 0},
   pdftitle={TFT on open-closed r-spin surfaces - W.H. Stern and L. Szegedy},
   pdfstartview=FitH,      
   bookmarksdepth=section, 
   bookmarksopen=false,     
   bookmarksnumbered=true 
}

\title{Topological field theories on open-closed $r$-spin surfaces}
\author{Walker H. Stern and L\'or\'ant Szegedy}

\begin{document}
	
\thispagestyle{empty}
\def\thefootnote{\fnsymbol{footnote}}
\begin{flushright}
ZMP-HH/20-11\\
Hamburger Beitr\"age zur Mathematik 2020-08
\end{flushright}
\vskip 3em
\begin{center}\LARGE
	Topological field theories on open-closed $r$-spin surfaces
\end{center}

\vskip 2em
\begin{center}
{\large 
Walker H. Stern~${}^{a}$~~and~~L\'or\'ant Szegedy~${}^{b}$~\footnote{Emails: {\tt ws7jx@virginia.edu}~,~{\tt lorant.szegedy@univie.ac.at}}}
\\

\vskip 2em
${}^{a}$ Department of Mathematics, University of Virginia\\ 
141 Cabell Dr. Charlottesville, VA 22903 \\[1em]

${}^{b}$ Institute of Science and Technology Austria\\
Am Campus 1, 3400 Klosterneuburg, Austria
~\footnote{Present address:
Faculty of Physics, University of Vienna,
Boltzmanngasse 5, 1090 Vienna, Austria
}

\end{center}

\vskip 2em

\begin{abstract}
	In this article, we establish a connection between two models for $r$-spin structures on surfaces: the marked PLCW decompositions of Novak and Runkel-Szegedy, and the structured graphs of Dyckerhoff-Kapranov. We use these models to describe $r$-spin structures on open-closed bordisms, leading to a generators-and-relations characterization of the 2-dimensional open-closed $r$-spin bordism category. This results in a classification of 2-dimensional open closed field theories in terms of algebraic structures we term ``knowledgeable $\Lambda_r$-Frobenius algebras''. We additionally extend the state sum construction of closed $r$-spin TFTs from a $\Lambda_r$-Frobenius algebra $A$ with invertible window element of Novak and Runkel-Szegedy to the open-closed case. The corresponding knowledgeable $\Lambda_r$-Frobenius algebra is $A$ together with the $\mathbb{Z}/r$-graded center of $A$.
\end{abstract}
{\small
	\tableofcontents
}
	
	\section{Introduction}

The classification of 2-dimensional oriented topological field theories in terms of commutative Frobenius algebras is familiar to anyone who has even briefly encountered the subject \cite{Dijkgraaf:1989phd,Abrams:1996ty,Kock:2004fa}. Pictures of circles, and bordisms between them, have become a standard way of providing intuition for the more general framework. Only slightly less well-known are the classifications of the open sector (whose objects are intervals, and whose bordisms are surfaces with corners), and the classification of open-closed topological field theories in terms of `knowledgeable Frobenius algebras' in \cite{Lazaroiu:2001oc,Moore:2006db,Lauda:2008oc}. 

In the recent literature, there has been increasing interest in topological field theories with additional structure, which may be more relevant to physically interesting situations. These variants on topological field theory can be loosely grouped into two main classes: (1) remembering more information about the topological structure of the bordisms via higher-categorical data (\textsl{extended topological field theory}), or (2) adding addition topological or geometric structure to the bordisms in the source category. 

This paper concerns itself with the latter approach, continuing a program that can be viewed as beginning with the work of Novak \cite{Novak:2015phd}. In any dimension $n$, one can consider the double cover $\pi:\on{Spin}(n)\to \on{SO}(n)$, and equip manifolds with a reduction of structure group along $\pi$. One can then consider a bordism category in which the $n$-dimensional bordisms are equipped with such a reduction of structure group --- called a \textsl{spin structure}. 
Novak and Runkel's initial work on the subject in \cite{Novak:2015NR} provided a combinatorial model for spin structures on surfaces in terms of additional data on a triangulation, and used this to provide a state-sum construction of 2-dimensional spin topological field theories, with an eye towards the potential utility of such a construction in the study of conformal field theories.

However, in this 2-dimensional case, there is a potential for a wider array of `spin structures.' The 2-fold covering $\on{Spin}(2)\to \on{SO}(2)$ is not a universal cover, and, in fact, there is an $r$-fold cover $\on{Spin}_r(2)\to \on{SO}(2)$ for any $0<r<\infty$. We will follow the convention of \cite{Runkel:2018:rs} in writing $r=0$ for the case of the universal cover.  This allows one to consider general \textsl{(2-dimensional) $r$-spin topological field theories} in analogy to the above case
	(the case $r=0$ referring to framed TFTs).
These have already been the subject of study. In \cite{Runkel:2018:rs}, Novak's combinatorial model of $r$-spin structures on surfaces was expanded to ``polygonal'' cell decompositions with additional data. This was then used to provide a state-sum construction for \textsl{closed} $r$-spin topological field theories. In separate work, \cite{Dyckerhoff:2015csg} provide a model for $r$-spin structures on surfaces in terms of additional data on embedded graphs. This model was then used in \cite{Stern:2016stft} to provide a classification of \textsl{open} $r$-spin topological field theories. 

The purpose of the current paper is two-fold. Firstly, we establish the connection between the two combinatorial models for $r$-spin surfaces presented in \cite{Runkel:2018:rs} and \cite{Dyckerhoff:2015csg}, and provide a `dictionary' for translating between them. Secondly, we use this connection to extend the classification of open $r$-spin topological field theories for \cite{Stern:2016stft} to a classification of \textsl{open-closed} $r$-spin topological field theories.

\subsection*{Combinatorial models}

The key components necessary for our classification are a pair of combinatorial models for $r$-spin surfaces. The first model, that of \cite{Runkel:2018:rs}, takes a polygonal cell decomposition of an oriented surface $\Sigma$ ---  a PLCW decomposition (cf. Section \ref{subsec:PLCW} and \cite{Kirillov:2012pl}) --- and equips it with the addition structure of a \textsl{marking}, which is comprised of:
\begin{itemize}
	\item an orientation of each edge of the PLCW decomposition;
	\item a choice of a marked edge associated to each 2-cell of the PLCW decomposition; and 
	\item a $\Zb/r$-valued label for every edge of the PLCW decomposition. 
\end{itemize}
To obtain the corresponding $r$-spin structure, one first equips each 2-cell with the trivial $r$-spin structure (the only such, up to isomorphism). The $\Zb/r$-valued labels are then viewed as specifying transition functions between the two 2-cells joined at a given edge. The edge orientations, together with the orientation on $\Sigma$, give a canonical direction in which to read the transition function.  An additional admissibility criterion at the vertices then ensures that the resulting $r$-spin structure can be extended over the $0$-cells of the PLCW decomposition. 

The second combinatorial model is that of \cite{Dyckerhoff:2015csg}. Given a surface $S$ with a finite set of punctures $M$, one specifies a graph $\Gamma$ embedded in $S\setminus M$, meeting every boundary component, and inducing a homotopy equivalence $|\Gamma|\simeq S\setminus M$. An $r$-spin structure on $S\setminus M$ is then specified by providing the data of a functor $I(\Gamma)\to \Lambda_r$ from the \textsl{incidence category} of $\Gamma$ to the $r$-cyclic category. This amounts to specifying the data of:
\begin{itemize}
	\item An $\Zb/((n+1),r)$-torsor lying over every vertex with $n+1$ incident half-edges;
	\item a $\Zb/2r$-torsor lying over every edge; and
	\item morphisms relating the torsors associated to a vertex $v$ and an edge $e$ if one of the half-edges comprising $e$ is incident to $v$. 
\end{itemize}
These data must satisfy additional compatibilities packaged in the functoriality, as well as compatibility with the orientation of the surface. The $r$-spin structure is obtained by considering the torsor over a vertex $v$ as a `set of sheets' of the trivial $r$-spin structure over a disk containing $v$, and then using the morphisms to define transition functions. 

Both of these models, and the relation between them, are presented in detail in Section \ref{sec:combmod}.  Schematically, the following diagram gives a rough idea of the correspondence. 

\begin{center}
	\begin{tikzpicture}
	\path (0,0) node {\large  PLCW};
	\path (6,0) node {\large  Graph};
	
		\path (0,-0.5) node (a1) {PLCW decomposition};
		\path (6,-0.5) node (b1) {embedded graph};
		\draw[<->] (a1) to node[above]{dual graph} (b1);
		
		\path (0,-1.5) node (a2) {marked edge}; 
		\path (6,-1.5) node (b2) {vertex torsor trivialization};
		\draw[<->] (a2) to  (b2);
		
		\path (0,-2.5) node (a3) {edge orientation};
		\path (6,-2.5) node (b3) {edge torsor trivialization};
		\draw[<->] (a3) to  (b3);
		
		\path (0,-3.5) node (a4) {edge label}; 
		\path (6,-3.5) node (b4) {edge morphisms};
		\draw[<->] (a4) to  (b4);
	\end{tikzpicture}
\end{center}

\subsection*{Classification}

Once we have established our combinatorial models and their relation to one another, we proceed to a classification of open-closed $r$-spin topological field theories. We do this by providing generators and relations for the $r$-spin bordism category --- the relation between the two models allows us to import the classification in the open sector from \cite{Stern:2016stft} into the PLCW formalism. To this end we make heavy use of the oriented classification of \cite{Lauda:2008oc}. In particular, the generators of the open-closed $r$-spin bordism category $\Bord{r}$ are given by $r$-spin structures on the generators of the open-closed oriented bordism category $\Bordori$.

The generators of $\Bord{r}$ form what we call a \textsl{knowledgeable $\Lambda_r$-Frobenius algebra} in $\Bord{r}$. Loosely speaking, a knowledgeable $\Lambda_r$-Frobenius algebra in a symmetric monoidal category $\mathcal{S}$ consists of two compatible structures: 
\begin{enumerate}
	\item A \textsl{$\Lambda_r$-Frobenius algebra} as defined in \cite{Dyckerhoff:2015csg}, i.e.\ a Frobenius algebra $A$ in $\mathcal{S}$ whose Nakayama automorphism $N_A$ satisfies $N_A^r=\on{id}_A$. 
	\item A \textsl{closed $\Lambda_r$-Frobenius algebra}, which consists of 
	\begin{itemize}
		\item a collection $\{C_x\}_{x\in \Zb/r}$ of objects in $\mathcal{S}$ equipped with $\Zb/r$-actions and 
		\item morphisms 
		\begin{align*}
		\mu_{x,y}&:C_x\otimes C_y\to C_{x+y-1}& \eta_1&:\Ib\to C_1\\
		\Delta_{x,y}&:C_{x+y+1}\to C_x\otimes C_y& \varepsilon_{-1}&:C_{-1}\to \Ib
		\ .
		\end{align*}
		which intertwine the $\Zb/r$-actions, (here, the monoidal unit $\mathbb{I}$ is considered to carry the trivial action)
	\end{itemize}
	subject to conditions similar to those governing 
		commutative
	Frobenius algebras.
\end{enumerate}
These two compatible structures must then be related by a pair of morphisms, which satisfy analogues of the knowledge, duality, and Cardy conditions of \cite{Lauda:2008oc}. 

Once we have established that a chosen set of bordisms in $\Bord{r}$ form a knowledgeable $\Lambda_r$-Frobenius algebra in $\Bord{r}$, the classification can be proved in two steps. We first show that any bordism in $\Bord{r}$ can be written as a composition of generators, using the decomposition on the underlying oriented bordism given in \cite{Lauda:2008oc}. We then show that, if two $r$-spin bordisms are isomorphic, their decompositions into generators can be related by the conditions governing knowledgeable $\Lambda_r$-Frobenius algebras. 

After proving this classification, we briefly discuss extending the state-sum construction provided in \cite{Runkel:2018:rs} to the full open-closed category. Just as in in the closed case, the input for the state-sum construction is a $\Lambda_r$-Frobenius algebra $A$ with invertible window element. In this case, the corresponding knowledgeable $\Lambda_r$-Frobenius algebra is formed by $A$ together with the $\Zb/r$-graded center of $A$, as defined in \cite{Runkel:2018:rs}. 

\subsection*{Structure of the paper}

In Section \ref{sec:Bordcats}, we define the open-closed bordism categories, and fix notation and conventions for $r$-spin surfaces. Section \ref{sec:combmod} provides an exposition of the two main combinatorial models --- marked PLCW decompositions and $\Lambda_r$-structured graphs --- giving the relation between the models, and  the relation of the former to $r$-spin surfaces. In Section \ref{sec:FrobAlg}, we provide a detailed exposition of knowledgeable $\Lambda_r$-Frobenius algebras and their connection to the $\Zb/r$-graded centers of $\Lambda_r$-Frobenius algebras with invertible window element. Section \ref{sec:TFTclass} contains the main results of the paper: the classification of open-closed $r$-spin topological field theories, and the state-sum construction of the same. For ease of reading, we provide an appendix, which contains background on the $r$-cyclic category $\Lambda_r$, as well as a computational proof omitted in the main text. 

\subsection*{Acknowledgments} 

We would like to thank Tobias Dyckerhoff, who pointed out the similarity between our previous papers, and suggested we work together. We are also grateful to Ingo Runkel and Nils Carqueville for their helpful comments and suggestions 
	and for Ehud Meir for pointing our a mistake in an earlier version of this paper. 
Both authors were supported, at various times during the preparation of this paper, by the Max Planck Institute for Mathematics in Bonn and the University of Hamburg. Additionally, we would like to thank the RTG 1670 of the University of Hamburg, with which 
	they
were both affiliated. W.H.S. additionally acknowledges the support of the NSF Research Training Group at the
University of Virginia (grant number DMS-1839968) during the revision of the article.

	\section{The open-closed \texorpdfstring{$r$}{r}-spin bordism category}\label{sec:Bordcats}
\subsection{Smooth open-closed bordisms}\label{sec:ocbord}
In the following we recall some notions following \cite{Lauda:2008oc}.
      A \textsl{smooth $1$-dimensional manifold with perimeter} is a 1-dimensional manifold diffeomorphic to the disjoint union of closed intervals.
A \textsl{smooth $2$-dimensional manifold with corners} is a $2$-dimensional manifold such that every point has a neighborhood homeomorphic to 
an open subset of $\Rb_{\ge0}^2$ and such that the transition functions are restrictions to $\Rb_{\ge0}^2$
of diffeomorphisms of open subsets of $\Rb^2$.
In the following we will assume that all manifolds are smooth.

Let $\Sigma$ be a $2$-dimensional manifold with corners and let $p\in\Sigma$. 
Let $c(p)\in\Zb_{\ge0}$ denote the number of zero coefficients in local coordinates $\varphi(p)\in\Rb_{\ge0}^2$, 
which is independent of the choice of coordinates.
A \textsl{connected perimeter of $\Sigma$} is the closure of the component $\setc*{p\in\Sigma}{c(p)=1}$.
A \textsl{perimeter of $\Sigma$} is a disjoint union of pairwise disjoint connected perimeters.
A \textsl{$2$-dimensional manifold with perimeter} is a $2$-dimensional manifold with corners $\Sigma$ 
such that every $p\in\Sigma$ is contained in $c(p)$ different connected perimeters.

A \textsl{$2$-dimensional $\langle2\rangle$-manifold} is a $2$-dimensional manifold with perimeter 
$\Sigma$ with a fixed pair of perimeters $(\dpar\Sigma,\dfree\Sigma)$ of $\Sigma$
such that $\dpar\Sigma\cup\dfree\Sigma=\partial\Sigma$ and
such that $\dpar\Sigma\cap\dfree\Sigma$ is a perimeter of both $\dpar\Sigma$ and of $\dfree\Sigma$.
A \textsl{diffeomorphism of $2$-dimensional $\langle2\rangle$-manifolds} $f:\Sigma\to\Sigma'$
is a diffeomorphism of $2$-dimensional manifolds with corners 
such that $f(\dpar\Sigma)=\dpar\Sigma'$ and $f(\dfree\Sigma)=\dfree\Sigma'$.
In the following we will only write diffeomorphism for short.

By a \textsl{surface} we shall mean an oriented 2-dimensional $\langle2\rangle$-manifold 
and will use the notation $\Sigma=\left( \Sigma,\dpar\Sigma,\dfree\Sigma \right)$. 
Let $\Sigma$ be a compact surface. Then $\dpar\Sigma$ is diffeomorphic to a disjoint union of 
intervals $I=[-1,1]$ and of circles $\Sb$. In the following we define collar neighborhoods of these.

An \textsl{open collar} $U^o$ is an open neighborhood of $I=\left[ -1,1 \right]$ in $I\times\R$.
An \textsl{ingoing} (resp.\ \textsl{outgoing}) \textsl{open collar} 
$U^o_{\mathrm{in}}$ (resp.\ $U^o_{\mathrm{out}}$) is the intersection of a open collar with
the set $I\times\R_{\le0}$ (resp.\ $I\times\R_{\ge0}$).
A \textsl{closed collar} $U^c$ is an open neighborhood of $\Sb$ in $\Cb^{\times}$.
An \textsl{ingoing} (resp.\ \textsl{outgoing}) \textsl{closed collar} 
$U^c_{\mathrm{in}}$ (resp.\ $U^c_{\mathrm{out}}$) is the intersection of a closed collar with
the set $\setc*{z\in\Cb^{\times}}{|z|\ge 1}$ (resp.\ $\setc*{z\in\Cb^{\times}}{|z|\le 1}$).
Note that these collars are surfaces in the above sense with
$\dpar U^o_{\mathrm{in}}=\dpar U^o_{\mathrm{out}}=I\times\left\{ 0 \right\}\simeq I$ and
$\dpar U^c_{\mathrm{in}}=\dpar U^c_{\mathrm{out}}=\Sb$.

A \textsl{boundary parametrization} of a compact surface $\Sigma$ is:
\begin{enumerate}
	\item 
	A disjoint decomposition 
	$\din\Sigma \sqcup \dout\Sigma= \dpar\Sigma$ of the parametrized boundary into ingoing and outgoing boundary with connected components:
	$\Bin=\pi_0(\din\Sigma)$ and $\Bout=\pi_0(\dout\Sigma)$. 
	Each of these sets are allowed to be empty.
	\item
	A collection of ingoing open or closed collars $U_b$, $b \in B_\mathrm{in}$, 
	and outgoing open or closed collars  $V_c$, $c \in B_\mathrm{out}$, 
	together with a pair of orientation preserving embeddings
	\begin{align}
	\phi_\mathrm{in}:\bigsqcup_{b\in B_\mathrm{in}}U_b\rightarrow\Sigma\leftarrow
	\bigsqcup_{c\in B_\mathrm{out}}V_c:\phi_\mathrm{out} \ .
	\label{eq:bdrparam}
	\end{align}
	We require that for each $b$, the restriction $\phi_\mathrm{in}|_{U_b}$ maps $I$ 
	(resp.\ $\mathbb{S}^1$) diffeomorphically to the connected component $b$ of $\Bin$,
	and analogously for $\phi_\mathrm{out}|_{V_c}$.
\end{enumerate}

One would then define open-closed bordisms between open-closed objects,
diffeomorphism and glueing thereof and finally the category of open-closed bordisms \cite[Sec.\,3]{Lauda:2008oc}.
We will instead turn directly to define the category of open-closed $r$-spin bordisms, 
which in the particular case of $r=1$ will agree with the category of open-closed bordisms.

\subsection{\texorpdfstring{$r$}{r}-spin bordisms}\label{sec:rspinbord}

In the first part of this section we briefly recall the notion of $r$-spin structures and morphisms thereof. For further details we refer to \cite{Novak:2015phd,Runkel:2018:rs}.
Then we extend the notion of $r$-spin bordisms from closed to open-closed, meaning that we allow open parametrized boundary components.
We fix $r\in\Zb_{\ge0}$ and write $GL_2^+(\Rb)$ for the positive determinant $2{\times}2$ real matrices.
Let 
\[
p_{GL}^r:\widetilde{GL}_2^r\to GL_2^+(\Rb)
\]
be the $r$-fold cover for $r>0$ and the universal cover for $r=0$.

We write $F_{\Sigma}\to\Sigma$ for the oriented frame bundle of a surface $\Sigma$, which is a $GL_2^+(\Rb)$ principal bundle.
An \textsl{$r$-spin structure on $\Sigma$} is a $\widetilde{GL}_2^r$ principal bundle $P\to\Sigma$ together with
a bundle map $p:P\to F_{\Sigma}$ intertwining the $\widetilde{GL}_2^r$ and $GL_2^+(\Rb)$ actions.
Note that $p:P\to F_{\Sigma}$ is a $\Zb/r$ principal bundle.
An \textsl{$r$-spin surface} is a surface together with an $r$-spin structure.
We will often abbreviate $(P,p,\Sigma)$ by $\Sigma$.
A \textsl{morphism of $r$-spin surfaces} $\tilde{f}:(P,p,\Sigma)\to(P',p',\Sigma')$ is a bundle map $\tilde{f}:P\to P'$ with underlying map of surfaces $f$
such that the diagram
\begin{equation}
\begin{tikzcd}
P \rar{\tilde{f}} \dar[swap]{p} & P  \dar{p'} \\
F_{\Sigma} \rar{df_*}  & F_{\Sigma'} 
\end{tikzcd}\label{eq:def:morphism-of-r-spin-surfaces}
\end{equation}
commutes, where $df_*$ is the induced map from the derivative $df$ of $f$.
We call $\tilde{f}$ a \textsl{morphism of $r$-spin structures} if $f=\id_{\Sigma}$.

\begin{rmk}
	It is worth noting here that we here make use of $GL_2^+(\Rb)$ and its covers, rather than $SO(2)$ and its covers, discussed in the introduction. There is a correspondence between these two notions (c.f.\ e.g.\ \cite[Sec.\,I.5]{Dyckerhoff:2015csg}), but working with $SO(2)$ involves the additional choice of a metric on the surface. 
\end{rmk}

\medskip

In the following we will need the notion of $r$-spin collars and for this we introduce some more notation.
Consider $I\times\Rb$, which is contractible, hence every $r$-spin structure on it is isomorphic to the trivial one.
Now consider $\Cb^{\times}=\Rb^{2}\setminus\left\{ 0 \right\}$.

\begin{lemma}[{\cite[Sec.\,3.4]{Novak:2015phd}}]
	\label{lem:r-spin-R2}
	There is a bijection of sets
	\begin{align}
	\begin{aligned}
	\{\text{isomorphism classes of $r$-spin structures on $\Cb^{\times}$}\}
	\simeq \Zb/r\ .
	\end{aligned}
	\label{eq:r-spin-R2}
	\end{align}
\end{lemma}

Let us fix an $r$-spin structure $\Cb^{\kappa}$ on $\Cb^{\times}$,
whose isomorphism class corresponds to $\kappa\in\Zb/{r}$ via \eqref{eq:r-spin-R2}.
An \textsl{$r$-spin boundary parametrization} of an $r$-spin surface $\Sigma$ is:
\begin{enumerate}
	\item A boundary parametrization of the underlying surface $\Sigma$ as in Section~\ref{sec:ocbord}; 
	we use the same notation as in \eqref{eq:bdrparam}.
	\label{part:bdryparam}
	\item A pair of maps 
	\begin{align}
		x^\mathrm{in}:B_\mathrm{in}&\to\Zb/r\cup \left\{ * \right\}        & \text{and}&&x^\mathrm{out}:B_\mathrm{out}&\to\Zb/r\cup \left\{ * \right\}\ ,
	\label{eq:collarmaps}\\
	b&\mapsto x^\mathrm{in}_b&           &&          c&\mapsto x^\mathrm{out}_c\nonumber
	\end{align}
	such that $x^\mathrm{in}_b,x^\mathrm{out}_c\in\Zb/r$ for $b$, $c$ closed boundary components and
	$x^\mathrm{in}_b=*=x^\mathrm{out}_c$ for $b$, $c$ open boundary components.
	We will also use the shorthand notation $x=(x^\mathrm{in},x^\mathrm{out})$.
	\label{part:maps}
	\item A pair of morphisms of $r$-spin surfaces
	which parametrize the in- and outgoing boundary components by collars with $r$-spin structure,	
	\begin{align}
	\varphi_\mathrm{in}:\bigsqcup_{b\in B_\mathrm{in}}
	U_b^{x^\mathrm{in}_b}
	\rightarrow\Sigma\leftarrow
	\bigsqcup_{c\in B_\mathrm{out}}
	V_c^{x^\mathrm{out}_c}
	:\varphi_\mathrm{out} \ .
	\label{eq:rsbdrparam}
	\end{align} 
	Here, if $b$ is a closed boundary component, 
	$U_b^{x^\mathrm{in}_b}$ is the restriction of $\Cb^{x^\mathrm{in}_b}$ to the ingoing closed collar $U_b$. 
	If $b$ is an open boundary component, then $x^\mathrm{in}_b=*$ and $U_b^{*}$ is the restriction of the trivial $r$-spin structure to $U_b$.
	We define $V_c^{x^\mathrm{out}_c}$ analogously.
	The maps of surfaces underlying $\varphi_{\mathrm{in}/\mathrm{out}}$ 
	are required to be the maps $\phi_{\mathrm{in}/\mathrm{out}}$ in \eqref{eq:bdrparam} from 
	Part~\ref{part:bdryparam}.
\end{enumerate}
Note that the maps $x^\mathrm{in},x^\mathrm{out}$ in Part~\ref{part:maps} are not extra data.
They 
are uniquely determined by the restriction of the $r$-spin structure on $\Sigma$ 
to the collar neighborhoods of the boundary components and Lemma~\ref{lem:r-spin-R2}.

A \textsl{diffeomorphisms between $r$-spin surfaces with parametrized boundary}
$\tilde{f}:\Sigma\to\Sigma'$
is a diffeomorphism of $r$-spin surfaces which respects germs of the boundary parametrization.
By this we mean the following. Let $f_*:\pi_0(\Sigma)\to\pi_0(\Sigma')$ be the map induced by $f$.
Then for every $b\in B_\mathrm{in}$ there exists a collar $C_b$ contained in
both $U_b$ and $U'_{f_*(b)}$ such that 
\begin{align}
\begin{aligned}
	\tilde{f}\circ\varphi_\mathrm{in}|_{C^{x^\mathrm{in}_b}_{b}}=\varphi_\mathrm{in}'|_{C^{x^\mathrm{in}_{f_*(b)}}_{f_*(b)}}\ ,
\end{aligned}
\label{eq:morphism-of-r-spin-surf-bdry-param}
\end{align}
where $C^{x^\mathrm{in}_b}_{b}$ is defined as $U_b^{x^\mathrm{in}_b}$,
and the similar condition for outgoing boundary components holds.

\medskip

An \textsl{$r$-spin object} is a pair $(X,\rho)$ 
consisting of a finite set $X$ and a map 
$\rho : X \to \Zb/r \cup \left\{ * \right\}$, $x \mapsto \rho_x$. 
Below we construct a category with objects $r$-spin objects,
and whose morphisms we define now.

\begin{defn}
	Let $(X,\rho)$ and $(Y,\sigma)$ be two $r$-spin objects.
	An \textsl{open-closed $r$-spin bordism from $(X,\rho)$ to $(Y,\sigma)$} is a compact
	$r$-spin surface $\Sigma$ with boundary parametrization  as in \eqref{eq:rsbdrparam} together with bijections $\beta_\mathrm{in} : X \to B_\mathrm{in}$ and $\beta_\mathrm{out} : Y \to B_\mathrm{out}$ such that
	\begin{align} \label{eq:boundary-label-compatible}
	\begin{aligned}
	\raisebox{.4\totalheight}{
		\xymatrix{
			X \ar[rr]^{\beta_\mathrm{in}} \ar[dr]_\rho && 
			B_\mathrm{in} \ar[dl]^{x^\mathrm{in}} \\
			& \Zb/r \cup \left\{ * \right\}
	}}
	\quad \text{and} \quad
	\raisebox{.4\totalheight}{
		\xymatrix{
			Y \ar[rr]^{\beta_\mathrm{out}} \ar[dr]_\sigma && 
		B_\mathrm{out} \ar[dl]^{x^\mathrm{out}} \\
			& \Zb/r \cup \left\{ * \right\}
	}}
	\end{aligned}
	\end{align}
	commute.
	We will abbreviate an open-closed $r$-spin bordism $\Sigma$ from $(X,\rho)$ to $(Y,\sigma)$ as $\Sigma : \rho \to \sigma$. 
\end{defn}

Given open-closed $r$-spin bordisms $\Sigma : (X,\rho) \to (Y,\sigma)$ and $\Xi : (Y,\sigma) \to (Z,\tau)$, 
we define the \textsl{glued open-closed $r$-spin bordism} $\Xi \circ \Sigma : \rho \to \tau$
as follows. 
For every $y \in Y$, we glue the boundary component $\beta^\Sigma_\mathrm{out}(y)\in B_\mathrm{out}^{\Sigma}$ to the boundary component  $\beta^\Xi_\mathrm{in}(y)\in B_\mathrm{in}^{\Xi}$
using the $r$-spin boundary parametrizations $\varphi^\Sigma_\mathrm{out}$ and $\varphi^\Xi_\mathrm{in}$. 
The diagrams in~\eqref{eq:boundary-label-compatible} make sure that we glue together the restrictions of the same $r$-spin structure
on $\Cb^\times$ or on $I\times\Rb$.

Two open-closed $r$-spin bordisms between the same $r$-spin objects,
$\Sigma,\Sigma' : (X,\rho) \to (Y,\sigma)$ are called \textsl{equivalent} 
if there is a diffeomorphism $\tilde{f}: \Sigma \to \Sigma'$ of $r$-spin surfaces with boundary parametrization such that
\begin{align}
\begin{aligned}
\xymatrix @R=.6pc{
	& B_\mathrm{in}^{\Sigma} \ar[dd]^{f_*} \\
	X \ar[ru]^{\beta_{\mathrm{in}}} \ar[rd]_{\beta'_{\mathrm{in}}} \\
	& B_\mathrm{in}^{\Sigma'} 
}
\end{aligned}
\qquad \text{and} \qquad
\begin{aligned}
\xymatrix @R=.6pc{
	B_\mathrm{out}^{\Sigma} \ar[dd]_{f_*} \\
	& Y \ar[lu]_{\beta_{\mathrm{out}}} \ar[ld]^{\beta'_{\mathrm{out}}} \\
	B_\mathrm{out}^{\Sigma'}
}
\end{aligned}
\end{align}
commute. 
The composition of equivalence classes $[\Xi]:\sigma\to\tau$ and $[\Sigma]:\rho\to\sigma$ is $[\Xi\circ\Sigma]:\rho\to\tau$,
which is independent of the choice of representatives.

\begin{defn}
	The \textsl{category of open-closed $r$-spin bordisms} $\Bord{r}$ has $r$-spin objects
	as objects and equivalence classes of open-closed $r$-spin bordisms as morphisms.
\end{defn}

$\Bord{r}$ is a symmetric monoidal category with tensor product given by disjoint union.
The symmetric structure is given by cylinders with different boundary parametrizations.

	\section{Combinatorial models for \texorpdfstring{$r$}{r}-spin surfaces}\label{sec:combmod}

In this section we will briefly develop the relevant extensions of the models for $r$-spin surfaces appearing in 
\cite{Runkel:2018:rs}, \cite{Dyckerhoff:2015csg}, and \cite{Stern:2016stft}, as well as the relation between them. Throughout, we will maintain the convention established above, namely that a \textsl{surface} is an oriented 2-dimensional $\langle 2\rangle$-manifold. 

Given that the model of \cite{Dyckerhoff:2015csg,Stern:2016stft} is designed for use with punctures, rather than free boundary circles, we first briefly establish the relations between these two constructions.  
	Throughout, we will assume $0\le r<\infty$; only in Section~\ref{sec:str-graph} we exclude the case $r=0$ of the universal cover for notational convenience.

\begin{defn}
	Given a surface with boundary parametrization, we will call a component of the free boundary which does not meet the parameterized boundary a \textsl{free circle}.
	We call a surface with boundary parametrization \index{surface!whole}\textsl{whole} if it has no free circles. 
	
	A \textsl{punctured}\index{surface!punctured} surface $(S,M)$ is a surface $S$ together with a finite set $M\subset S$ of \textsl{punctures} such that $M\cap \partial S=\emptyset$. An \textsl{$r$-spin structure} on a punctured surface $(S,M)$ is an $r$-spin structure on $S\setminus M$. 
\end{defn}

The basic idea now is that two operations --- \textsl{drilling} and \textsl{patching} --- establish a bijection between equivalence classes of $r$-spin structures on different surfaces. We will briefly describe the constructions, and state the necessary facts about them. The relevant notions are treated in detail in \cite[Sec.\,2.3.1]{Stern:2019phd}. 

A \textsl{drilling} of a punctured surface $(S,M)$ is constructed by choosing an open ball around each puncture, such that the closures of the open balls are pairwise disjoint and disjoint from $\partial S$, and removing these open balls from the surface. This creates a new surface $\Sigma$ which has one free circle for every puncture $m\in M$. 

In the other direction, a \textsl{patching} of a surface $\Sigma$ with boundary parametrization is constructed by gluing an oriented disk with a puncture at its center into each free circle. It is worth noting here that $\Sigma$ is diffeomorphic to a drilling of a patching of $\Sigma$, and every whole surface $(S,M)$ is diffeomorphic to a patching of a drilling of $(S,M)$. 

\begin{ntt}
	We denote by $\on{rSpin}(S,M,\phi)_{x}$ the set of $r$-spin structures on $(S,M)$ with boundary parametrization $\phi$ and boundary labels $x=(x^\mathrm{in},x^\mathrm{out})$. We denote by $\on{rSpin}(\Sigma,\phi)_{x}$ the set of $r$-spin structures on $\Sigma$ with boundary parametrization covering $\phi$ and boundary labels $x$. 
	
	We will use the notations 
	\[
		\on{rSpin}(S,M,\phi)_{x}/_{\on{diffeo}}
	\] 
	and 
	\[
		\on{rSpin}(S,M,\phi)_{x}/_{\on{iso}}
	\]
	to denote equivalence classes under $r$-spin diffeomorphisms preserving boundary parametrization and $r$-spin diffeomorphisms over the identity, respectively. 
\end{ntt}

\begin{prop}
	Let $(S,M,\phi)$ be a whole punctured surface with boundary parametrization $\phi$, and let $x=(x^\mathrm{in},x^\mathrm{out})$ be boundary labels for $(S,M,\phi)$. Let $(\Sigma, \phi)$ be a drilling of $(S,M,\phi)$. Then the inclusion $\Sigma\hookrightarrow S$ induces  bijections 
	\[
		\on{rSpin}(S,M,\phi)_{x}/_{\on{diffeo}}\cong \on{rSpin}(\Sigma,\phi)_{x}/_{\on{diffeo}}
	\]
	and 
	\[
		\on{rSpin}(S,M,\phi)_{x}/_{\on{iso}}\cong \on{rSpin}(\Sigma,\phi)_{x}/_{\on{iso}}
	\]
	via the restriction of $r$-spin structures to submanifolds. 
\end{prop}

\begin{proof}
	We prove the second statement first. Denote by $\rho_S: (S\setminus M) \to B\GL_2^+(\mathbb{R})$ the classifying map of the oriented frame bundle of $S\setminus M$, and similarly let $\rho_\Sigma: \Sigma \to B\GL_2^+(\mathbb{R})$ classify the oriented frame bundle of $\Sigma$. Then isomorphism classes of $r$-spin structures on $(S,M)$ are in bijection with homotopy classes of lifts 
	\[
	\begin{tikzcd}
	& B \widetilde{\GL}_2^r(\mathbb{R})\arrow[dd]\\
	S\setminus M\arrow[dr,"\rho_S"']\arrow[ur,dashed] & \\
	& B \GL_2^+(\mathbb{R}) 
	\end{tikzcd}
	\]
	and similarly for $\Sigma$. Since $\Sigma$ is a deformation retract of $S\setminus M$, this implies that there is a bijection between isomorphism classes of $r$-spin structures on $(S,M)$ and isomorphism classes of $r$-spin structures on $\Sigma$, induced by restriction. 
	
	For both $\Sigma$ and $S\setminus M$, the (germs of) boundary parameterizations are local data defined on collar neighborhoods $U_b\subset \Sigma \subset S\setminus M$. Thus, the bijection on isomorphism classes of $r$-spin structures descends to a bijection 
	\[
	\on{rSpin}(S,M,\phi)_{x}/_{\on{iso}}\cong \on{rSpin}(\Sigma,\phi)_{x}/_{\on{iso}}
	\]
	as desired.  
	
	The first statement follows from the second by applying \cite[Prop.\,2.1.14]{Stern:2019phd}. This tells us (1) that a boundary parameterization preserving $r$-spin diffeomorphism between elements of $\on{rSpin}(\Sigma,\phi)_x$ can be extended to a boundary parameterization preserving $r$-spin diffeomorphism between (representatives of) the corresponding elements of $\on{rSpin}(S,M,\phi)_x/_{\on{iso}}$; and (2) that a boundary parameterization preserving $r$-spin diffeomorphism between  elements of $\on{rSpin}(S,M,\phi)_x$ induces a boundary parameterization preserving $r$-spin diffeomorphism between the restrictions of these elements to $\Sigma$. 
\end{proof}

\begin{rmk}
	This bijection can be interpreted in terms of mapping class groups. If we consider a punctured surface $(S,M)$ with $m=|M|$, and $b$ boundary components, the mapping class groups fit into a short exact sequence
	\[
	\begin{tikzcd}
	0\arrow[r] & \mathbb{Z}^m\arrow[r] & \on{Mod}(\Sigma)\arrow[r] & \on{Mod}(S)\arrow[r] & 0 
	\end{tikzcd}
	\]
	where $\mathbb{Z}^m$ is generated by Dehn twists around the free circles. 
\end{rmk}

\subsection{Bordisms with marked PLCW decompositions}\label{subsec:PLCW}

This section will treat the first of our two combinatorial models for $r$-spin surfaces: the marked PLCW decompositions of \cite{Runkel:2018:rs}. This model applies for $0\leq r<\infty$. 

We begin by briefly recalling the PLCW decompositions of \cite{Kirillov:2012pl}. A slightly longer description than that given here is contained in 
\cite[Sec.\,2.2]{Runkel:2018:rs}.  

We consider the `linear closed ball' $B^n:=[-1,1]^n\subset \mathbb{R}^n$. An \textsl{$n$-cell} is the image of a piecewise linear (PL) map $\phi:B^n\to \mathbb{R}^m$ injective on the interior of $B^n$ (called a \textsl{characteristic map}). A \textsl{generalized cell decomposition} is a collection of cells in $\mathbb{R}^m$ with disjoint interiors, and such that the boundary of any cell is a union of cells. A \textsl{regular cell map} $f:K_\bullet\to L_\bullet$ is a piecewise linear map $f:\bigcup_{C\in K_\bullet}C\to \bigcup_{D\in L_\bullet} D $ such that, for every $C\in K_\bullet$ with characteristic map $\phi$, $f\circ \phi$ is the characteristic map of a cell in $L_\bullet$.

\begin{defn}
	A PLCW decomposition of dimension $n$ is a generalized cell decomposition $K_\bullet$ with at least one $n$-cell, and no cells of higher dimension, such that
	\begin{enumerate}
		\item the collection $(K_\bullet)^{<n}$ of all cells of $K_\bullet$ of dimension less than $n$ is a PLCW decomposition, and
		\item  for each cell $A\in K_n$, the induced map $\partial A\to (K_\bullet)^{<n}$ is a regular cell map.
	\end{enumerate}
Note that we can use smooth approximation to apply PLCW decompositions to smooth manifolds.
\end{defn}

For a punctured surface $(S,M)$ with boundary parametrization $\phi:\coprod U_i \to S$, we define the \textsl{core} $c(U_i)$ of each collar $U_i$ to be the subset $S^1\subset U_i$ if $U_i$ is a closed collar, and the subset $I\times\{0\}\subset U_i$ if $U_i$ is an open collar. By a \textsl{$\phi$-admissible PLCW decomposition} $(S,M)$, we mean a PLCW decomposition $K_\bullet$ of $S$ such that each puncture is a 0-cell of the decomposition, (i.e.\ $M\subset K_0$), and for each collar $U_i$, the restriction of $K_\bullet$ to the image of $c(U_i)\subset U_i\overset{\phi}{\to} S$ determines a PLCW decomposition of $c(U_i)$ with a single $1$-cell.

\begin{rmk}
	Note that our definition also implies that, for an open collar $U_i$, the endpoints of $c(U_i)$ are contained in $K_0$. 
\end{rmk}

\begin{defn}\label{defn:markedPLCW}
	Let $(S,M)$ be a punctured surface with boundary parametrization $\phi:\coprod U_i\to S$, and let $K_\bullet$ be a $\phi$-admissible PLCW decomposition of $S$. We denote by $\overline{K}_1\subset K_1$ the set of all  edges of $K_\bullet$  except boundary edges in the free boundary. 
		
	A \textsl{$\phi$-admissible marking of $K_\bullet$} consists of the following data: 
	\begin{itemize}
		\item An orientation of each edge in $\overline{K}_1$  (\textsl{edge orientations})
		\item A section $m:K_2\to E$ of the canonical map $E\to K_2$, where
		\[
		E:=\{(e,f)\mid e\in K_1, f\in K_2\text{ s.t.\ } e\text{ is an edge of }f\}
		\]
		(\textsl{edge markings}). 
		\item A map $s:\overline{K}_1\to \mathbb{Z}/r$, whose value on an edge $e$ we denote by $s_e$ (\textsl{edge indices}).
	\end{itemize} 
	If $S$ comes equipped with boundary labels $x^\mathrm{in}:B_{\on{in}}\to \mathbb{Z}/r$ and $x^\mathrm{out}: B_{\on{out}}\to \mathbb{Z}/r$, we additionally include the data of $x^\mathrm{in}_u$ or $x^\mathrm{out}_u$ assigned to the vertices of the corresponding boundary component as in \cite[Eq.\ 2.10]{Runkel:2018:rs}. 
\end{defn}

\begin{rmk}
	Note that Definition~\ref{defn:markedPLCW} can be applied to either a whole punctured surface $(S,M)$, or to a surface $\Sigma$ with parameterized boundary and no punctures.  
\end{rmk}

\begin{figure}[tb]
  \centering
  \begin{tikzpicture}
	\begin{pgfonlayer}{nodelayer}
		\node [style=none] (1) at (2, -2) {};
		\node [style=none] (2) at (3, 2) {};
		\node [style=none] (3) at (0, 4) {};
		\node [style=none] (4) at (-3, 2) {};
		\node [style=black dot] (5) at (-2, -2) {};
		\node [style=none] (6) at (-1, 0) {};
		\node [style=none] (7) at (-1, 1) {};
		\node [style=none] (8) at (0, 0) {};
		\node [style=none] (9) at (0.5, 0) {1};
		\node [style=none] (10) at (-1, 1.5) {2};
		\node [style=none] (11) at (-3, -2) {$v$};
	\end{pgfonlayer}
	\begin{pgfonlayer}{edgelayer}
		\draw (4.center) to (3.center);
		\draw (3.center) to (2.center);
		\draw (2.center) to (1.center);
		\draw (4.center) to (5);
		\draw [style=end arrow] (6.center) to (7.center);
		\draw [style=end arrow] (6.center) to (8.center);
		\draw [style=semicircle line] (5) to (1.center);
	\end{pgfonlayer}
\end{tikzpicture}
  \caption{Vertex $v$ in the counterclockwise direction with respect to the orientation of the surface. The orientation is represented by the ordered basis of the tangent space, drawn as 2 arrows marked 1 and 2.}
  \label{fig:ccw-vertex}
\end{figure}
Given a $\phi$-admissible marking of a $\phi$-admissible PLCW decomposition $K_\bullet$ of $(S,M)$, and a vertex $u$ in a boundary component, we set 
	\[
	\epsilon_u:=\begin{cases}
		+1 & u\in B_{\on{in}}\\
		-1 & u \in B_{\on{out}}
	\end{cases}
	\] 
For a vertex $v\in K_0$, we denote by $D_v$ the number of faces whose marked edge has $v$ as its boundary vertex in a counterclockwise direction (with respect to the orientation on $S$),
	see Figure~\ref{fig:ccw-vertex}. 
We denote $N^{\on{start}}_v$ the number of oriented edges starting at $v$, and by $N^{\on{end}}_v$ the number ending at $v$. We will write 
\[
N_v:=N^{\on{start}}_v+N^{\on{end}}_v
\]
We further denote by $\partial^{-1}(v)$ the set of edges whose boundary contains $v$. Note that we do not necessarily have $N_v = |\partial^{-1}(v)|$, as loops are allowed. For each $e\in \partial^{-1}(v)$ we denote 
\[
\hat{s}_e:=\begin{cases}
	-1 & e \text{ starts and ends at } v \text{ ($e$ is a loop)}\\
	s_e & e\text{ points out of } v\\
	-1-s_e & e\text{ points into } v.
\end{cases}
\] 

\begin{defn}
	\label{defn:AdmissiblePLCWSM}
	Let $(S,M)$ be a punctured surface with boundary parametrization $\phi:\coprod U_i \to S$. Given a $\phi$-admissible marking of a $\phi$-admissible PLCW decomposition $K_\bullet$ of $(S,M)$, let $v$ be a vertex in $K_0\setminus M$. If $v$ lies in the interior of $S$, we call  the marking \textsl{$v$-admissible} if 
	\begin{equation}\label{eq:intvertcondition}
	\sum_{e\in \partial^{-1}(v)} 	\hat{s}_e \equiv D_v -N_v+1 (\mod r). 
	\end{equation}
	If $v$ lies in the image of the core of a closed collar under $\phi$, we call the marking \textsl{$v$-admissible} if 
	\begin{equation}
		\sum_{e\in \partial^{-1}(v)} 	\hat{s}_e \equiv D_v -N_v+\epsilon_v\cdot x_v (\mod r)\ , 
	\end{equation}
	where $x_v$ is the value of $x=(x^\mathrm{in},x^\mathrm{out})$ on 
	the boundary component on which the vertex $v$ sits.
	We call the $\phi$-admissible marking on $K_\bullet$ \textsl{admissible} if it is $v$-admissible at every vertex fulfilling one of the two above descriptions. We denote the set of all admissible markings of  $K_\bullet$ by 
	\[
		\mathcal{M}^{K}_{x}(S,M).
	\]
	If $M=\emptyset$, we alter this notation to $\mathcal{M}^K_{x}(S)$.  
\end{defn}

As in \cite{Runkel:2018:rs}, admissible markings of PLCW decompositions on punctured surfaces correspond to $r$-spin structures. Before discussing this correspondence, however, we relate marked PLCW decompositions to the procedures of drilling and patching.

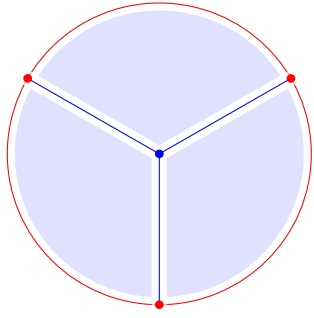
\begin{figure}[tb]
	\begin{center}
		\begin{tikzpicture}
		\path[fill=blue!40,opacity=0.3] (30:2)
		arc (30:150:2) -- (0,0)-- cycle;
		\path[fill=blue!40,opacity=0.3] (150:2)
		arc (150:270:2) -- (0,0)-- cycle;
		\path[fill=blue!40,opacity=0.3] (-90:2)
		arc (-90:30:2) -- (0,0)-- cycle;
		\draw[line width=2mm, draw=white] (0,0) circle (2);
		\draw[draw=red] (0,0) circle (2);
		\foreach \x/\lab in {30/a,150/b,270/c}{
			\path (\x:2) node (\lab) {};
			\draw[line width=2mm, draw=white] (0,0) to (\x:2);
			\draw[draw=blue] (0,0) to (\x:2);
			\draw[draw=red,fill=red] (\x:2) circle (0.05);
		};
		\path (0,0) node (ori) {}; 
		\draw[draw=blue,fill=blue] (0,0) circle (0.05); 
		\end{tikzpicture}
	\end{center}
	\caption{An extension (blue) of a PLCW decomposition of the circle (red) to the disk.}\label{fig:PLCWextension}
\end{figure}

\begin{lem}\label{lem:ExtPLCWfromDrill}
	Let $(S,M)$ be a whole surface with boundary parametrization $\phi: \coprod U_i\to S$. Let $\Sigma$ be a drilling of $(S,M)$, and let $K_\bullet$ be a $\phi$-admissible PLCW decomposition of $\Sigma$. Then there is a $\phi$-admissible PLCW decomposition $T_\bullet$ of $(S,M)$ which restricts to $K_\bullet$ on $\Sigma\subset S$. 
\end{lem}

\begin{proof}
	Let $U_m\subset S$ be the disk about $m$ whose removal defines $\Sigma$. The PLCW decomposition $K_\bullet$ determines a PLCW decomposition of $\partial U_m$ by restriction to a free boundary component. Consequently, the lemma reduces to showing that, for any PLCW decomposition $J_\bullet$ of $S^1=\partial D$, there is a PLCW decomposition $L_\bullet$ of $D$ extending $J_\bullet$ which has $\{0\}\in D$ as one of its $0$-cells.  This is straightforward, as we may take cells 
	\begin{align*}
	L_0&:= J_0\amalg \{0\},\\
	L_1&:= J_1\amalg J_0\times \{0\}\\
	L_2&:= J_1.  
	\end{align*}
	as shown in 
	Figure~\ref{fig:PLCWextension}.
\end{proof}

\begin{defn}[Fixed PLCW moves]\label{defn:fixmoves}
	Given a $\phi$-admissible PLCW decomposition $K_\bullet$ of punctured surface $(S,M)$, and an admissible marking $(m,o,s)$ of $K_\bullet$,\footnote{Here $m$ denotes the edge markings, $o$ the orientations, and $s$ the edge indices.} we call the following operations the \textsl{fixed PLCW moves}:
	\begin{enumerate}
		\item Reversing the orientation of an edge. This changes the edge labels as in 
			Figure~\ref{fig:PLCWfix} (1).
		\item Moving the marked edge of a face counterclockwise. This changes the edge labels as in 
			Figures~\ref{fig:PLCWfix} (2a) and (2b).
		\item Adding $\pm k$ to all edge labels, as in 
			Figure~\ref{fig:PLCWfix} (3). 
			We call this a \textsl{deck transformation}.
	\end{enumerate} 
	The fixed PLCW moves generate an equivalence relation 
	on $\mathcal{M}^{K}_{x}(S,M)$
	which we will denote by $\sim_{\on{fix}}$.
\end{defn}

\begin{figure}[tb]
	\begin{center}
		\begin{tikzpicture}
	\begin{pgfonlayer}{nodelayer}
		\node [style=none] (0) at (-10, 2) {};
		\node [style=none] (1) at (-7, 2) {};
		\node [style=none] (2) at (-11, 5) {};
		\node [style=none] (3) at (-6, 5) {};
		\node [style=none] (4) at (-8.5, 7) {};
		\node [style=none] (5) at (-2, 2) {};
		\node [style=none] (6) at (1, 2) {};
		\node [style=none] (7) at (-3, 5) {};
		\node [style=none] (8) at (2, 5) {};
		\node [style=none] (9) at (-0.5, 7) {};
		\node [style=none] (10) at (-8.5, 2.5) {$e$};
		\node [style=none] (11) at (-8.5, 1.5) {$s_e$};
		\node [style=none] (12) at (-0.5, 2.5) {$e$};
		\node [style=none] (13) at (-0.5, 1.5) {$-s_e-1$};
		\node [style=none] (14) at (-4.5, 3.5) {$\leftrightarrow$};
		\node [style=none] (15) at (6, 2) {};
		\node [style=none] (16) at (9, 2) {};
		\node [style=none] (17) at (5, 5) {};
		\node [style=none] (18) at (10, 5) {};
		\node [style=none] (19) at (7.5, 7) {};
		\node [style=none] (25) at (7.5, 2.5) {$e_1$};
		\node [style=none] (26) at (7.5, 1.5) {$s_{e_1}$};
		\node [style=none] (29) at (11.5, 3.5) {$\leftrightarrow$};
		\node [style=none] (30) at (-10, -5) {};
		\node [style=none] (31) at (-7, -5) {};
		\node [style=none] (32) at (-11, -2) {};
		\node [style=none] (33) at (-6, -2) {};
		\node [style=none] (34) at (-8.5, 0) {};
		\node [style=none] (35) at (-2, -5) {};
		\node [style=none] (36) at (1, -5) {};
		\node [style=none] (37) at (-3, -2) {};
		\node [style=none] (38) at (2, -2) {};
		\node [style=none] (39) at (-0.5, 0) {};
		\node [style=none] (40) at (-8.5, -4.5) {$e$};
		\node [style=none] (41) at (-8.5, -5.5) {$s_e$};
		\node [style=none] (42) at (-0.5, -4.5) {$e$};
		\node [style=none] (43) at (-0.5, -5.5) {$s_e-1$};
		\node [style=none] (44) at (-4.5, -3.5) {$\leftrightarrow$};
		\node [style=none] (45) at (6, -5) {};
		\node [style=none] (46) at (9, -5) {};
		\node [style=none] (47) at (5, -2) {};
		\node [style=none] (48) at (10, -2) {};
		\node [style=none] (49) at (7.5, 0) {};
		\node [style=none] (50) at (14, -5) {};
		\node [style=none] (51) at (17, -5) {};
		\node [style=none] (52) at (13, -2) {};
		\node [style=none] (53) at (18, -2) {};
		\node [style=none] (54) at (15.5, 0) {};
		\node [style=none] (55) at (7.5, -4.5) {$e$};
		\node [style=none] (56) at (7.5, -5.5) {$s_e$};
		\node [style=none] (57) at (15.5, -4.5) {$e$};
		\node [style=none] (58) at (15.5, -5.5) {$s_e+1$};
		\node [style=none] (59) at (11.5, -3.5) {$\leftrightarrow$};
		\node [style=none] (60) at (-11, 7) {$(1)$};
		\node [style=none] (61) at (5, 7) {$(3)$};
		\node [style=none] (62) at (-11, 0) {$(2a)$};
		\node [style=none] (63) at (5, 0) {$(2b)$};
		\node [style=none] (64) at (-8.5, -5) {};
		\node [style=none] (65) at (7.5, -5) {};
		\node [style=none] (66) at (8.5, 5.5) {$e_3$};
		\node [style=none] (67) at (9, 6.5) {$s_{e_3}$};
		\node [style=none] (68) at (6.5, 5.5) {$e_4$};
		\node [style=none] (69) at (6, 6.5) {$s_{e_4}$};
		\node [style=none, rotate=-72] (70) at (6, 3.5) {$e_5=e_3$};
		\node [style=none, rotate=-72] (71) at (5, 3.5) {$s_{e_5}=s_{e_3}$};
		\node [style=none, rotate=72] (72) at (9, 3.5) {$e_2$};
		\node [style=none, rotate=72] (73) at (10, 3.5) {$s_{e_2}$};
		\node [style=none] (74) at (14, 2) {};
		\node [style=none] (75) at (17, 2) {};
		\node [style=none] (76) at (13, 5) {};
		\node [style=none] (77) at (18, 5) {};
		\node [style=none] (78) at (15.5, 7) {};
		\node [style=none] (79) at (15.5, 2.5) {$e_1$};
		\node [style=none] (80) at (15.5, 1.5) {$s_{e_1}+k$};
		\node [style=none] (83) at (16.5, 5.5) {$e_3$};
		\node [style=none] (84) at (17, 6.5) {$s_{e_3}$};
		\node [style=none] (85) at (14.5, 5.5) {$e_4$};
		\node [style=none, rotate=36] (86) at (14, 6.5) {$s_{e_4}-k$};
		\node [style=none, rotate=-72] (87) at (14, 3.5) {$e_5=e_3$};
		\node [style=none, rotate=-72] (88) at (13, 3.5) {$s_{e_5}=s_{e_3}$};
		\node [style=none, rotate=72] (89) at (17, 3.5) {$e_2$};
		\node [style=none, rotate=72] (90) at (18, 3.5) {$s_{e_2}+k$};
	\end{pgfonlayer}
	\begin{pgfonlayer}{edgelayer}
		\draw (2.center) to (0.center);
		\draw (2.center) to (4.center);
		\draw (4.center) to (3.center);
		\draw (3.center) to (1.center);
		\draw [style=mid arrow] (0.center) to (1.center);
		\draw (7.center) to (5.center);
		\draw (7.center) to (9.center);
		\draw (9.center) to (8.center);
		\draw (8.center) to (6.center);
		\draw [style=mid arrow] (15.center) to (16.center);
		\draw (32.center) to (30.center);
		\draw (32.center) to (34.center);
		\draw (34.center) to (33.center);
		\draw (33.center) to (31.center);
		\draw (37.center) to (35.center);
		\draw (37.center) to (39.center);
		\draw (39.center) to (38.center);
		\draw (47.center) to (45.center);
		\draw (47.center) to (49.center);
		\draw (49.center) to (48.center);
		\draw (48.center) to (46.center);
		\draw (52.center) to (50.center);
		\draw (52.center) to (54.center);
		\draw (54.center) to (53.center);
		\draw [style=mid arrow] (35.center) to (36.center);
		\draw [style=mid arrow] (51.center) to (50.center);
		\draw [style=mid arrow] (65.center) to (45.center);
		\draw [style=mid arrow] (64.center) to (31.center);
		\draw [style=semicircle line] (30.center) to (64.center);
		\draw [style=semicircle line] (65.center) to (46.center);
		\draw [style=semicircle line] (36.center) to (38.center);
		\draw [style=semicircle line] (51.center) to (53.center);
		\draw [style=mid arrow] (16.center) to (18.center);
		\draw [style=mid arrow] (17.center) to (19.center);
		\draw [style=black dashed] (19.center) to (18.center);
		\draw [style=black dashed] (17.center) to (15.center);
		\draw [style=mid arrow] (74.center) to (75.center);
		\draw [style=mid arrow] (75.center) to (77.center);
		\draw [style=mid arrow] (76.center) to (78.center);
		\draw [style=black dashed] (78.center) to (77.center);
		\draw [style=black dashed] (76.center) to (74.center);
		\draw [style=mid arrow] (6.center) to (5.center);
	\end{pgfonlayer}
\end{tikzpicture}
	\end{center}
	\caption{The fixed PLCW moves.}\label{fig:PLCWfix}
\end{figure}
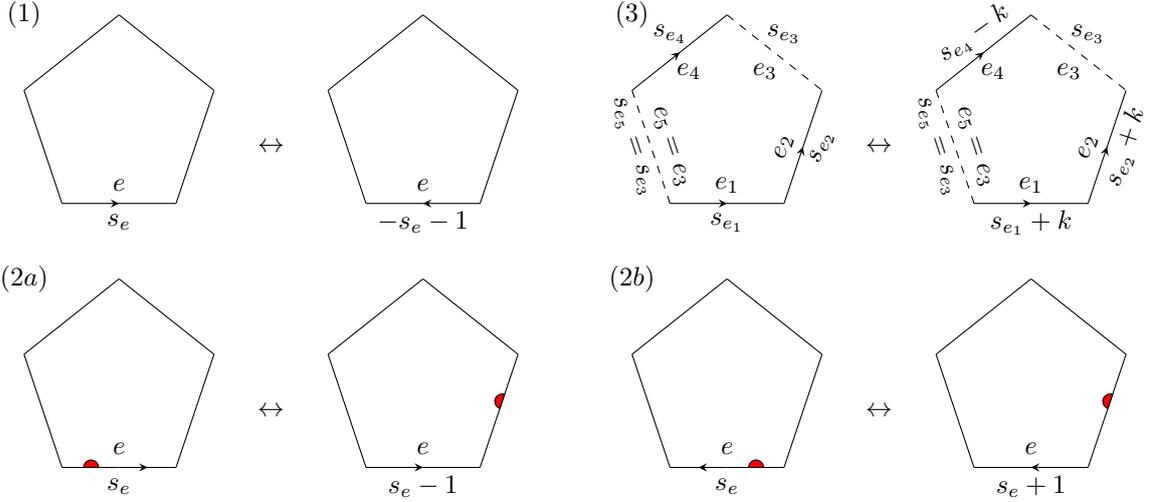

\begin{rmk}
	Note that, while the edge labels will determine transition functions for an $r$-spin bundle, they take values in a $\Zb/r$-torsor, rather than in $\widetilde{GL}_2^r(\Rb)$ itself. It is the relation between $r$-spin transition functions and $\Zb/r$-torsors which yields the relations in Figure~\ref{fig:PLCWfix}. See \cite{Novak:2015phd} for further details. 
\end{rmk}

\begin{prop}\label{prop:PLCWopenclosedisRSpin}
	Let $(S,M)$ be a punctured surface, $\phi$ a boundary parametrization of $(S,M)$, and $x=(x^\mathrm{in},x^\mathrm{out})$ boundary labels. Let $K_\bullet$ be a $\phi$-admissible PLCW decomposition of $(S,M)$. There is a bijection 
	\[
		\on{rSpin}(S,M,\phi)_{x}/_{\on{iso}}\cong \mathcal{M}^K_{x}(S,M)/_{\sim_{\on{fix}}}.
	\]
\end{prop}

\begin{proof}
	The proof is
		analogous to
	that of \cite[Thm.\,2.13]{Runkel:2018:rs}, presented in Appendix A.4 of \textsl{loc.\ cit.} While we do not replicate the 
		complete 
	proof here, we comment on the two differences to be observed. 

	The construction of \cite[Sec.\,4.8]{Novak:2015phd} gives an $r$-spin structure on $S\setminus K_0$ --- the surface minus the vertices of the PLCW decomposition.
	By requiring that for every inner vertex which is not a puncture the admissibility condition \eqref{eq:intvertcondition} holds, this $r$-spin structure extends to these vertices,
	and we obtain an $r$-spin structure on $S\setminus M$.
	Note that we do not require the $r$-spin structure to extend to puncures.

	Secondly, when dealing with free boundary components, the boundary edges of $K_\bullet$ lying outside the parameterized boundary carry no additional data (orientations 
	  or
      	indices), and the boundary vertices no compatibility conditions. However, since these edges abut only one 2-cell (on only one side), we do not need the data of transition functions on these edges. Each parameterized open boundary component $\psi: U\to S$ has edge labels on its core, which specify the transition function between the trivialized $r$-spin structure on $U$ and the trivialized $r$-spin structure on the 2-cell of $K_\bullet$ abutting the image of the core of $U$. 

	Just as in \cite[Sec.\,4.8]{Novak:2015phd}, the moves which generate the equivalence relation induce isomorphisms of $r$-spin structures.

	Conversely, obtaining a marking of the PLCW decomposition from an $r$-spin structure works via the same construction as in \cite[Sec.\,4.8]{Novak:2015phd}.
	We note here that we only record the marking for inner and boundary edges, for the same reason that we explained above.
\end{proof}

\subsubsection{Changing PLCW decompositions}	

\begin{defn}\label{defn:elementarymoves}
	Given a punctured surface $(S,M)$  with boundary parametrization $\phi$, boundary labels $x^\mathrm{in},x^\mathrm{out}$, a $\phi$-admissible PLCW decomposition $K_\bullet$, and an admissible marking $(m,o,s)$ of $K_\bullet$, the \textsl{elementary PLCW moves} are the operations 
	\begin{itemize}
		\item[4a.] adding or removing a vertex as in 
			Figure~\ref{fig:ElemMoves} (a), 
		\item[4b.] adding or removing an edge as in 
			Figure~\ref{fig:ElemMoves} (b),
	\end{itemize}
	subject to the conditions that the endpoints of parameterized boundary intervals and the punctures can be neither removed or added.
\end{defn}

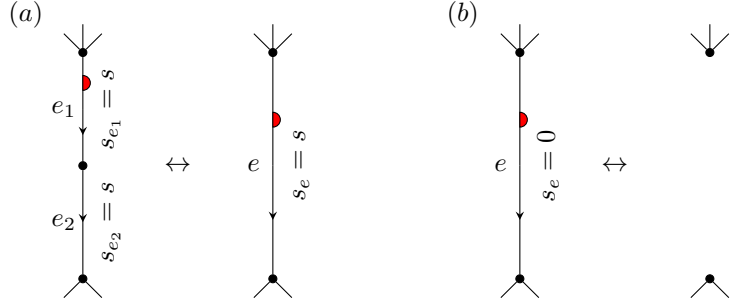
\begin{figure}[tb]
	\begin{center}
		\begin{tikzpicture}
	\begin{pgfonlayer}{nodelayer}
		\node [style=black dot small] (1) at (-8, -3) {};
		\node [style=black dot small] (2) at (-8, 0) {};
		\node [style=none] (5) at (-3, 0) {};
		\node [style=none] (6) at (-8, 1.5) {};
		\node [style=none] (7) at (-5.5, 0) {$\leftrightarrow$};
		\node [style=none] (8) at (-3.5, 0) {$e$};
		\node [style=none] (9) at (-8.5, 1.5) {$e_1$};
		\node [style=none] (10) at (-8.5, -1.5) {$e_2$};
		\node [style=none, rotate=90] (11) at (-2.25, 0) {$s_e=s$};
		\node [style=none, rotate=90] (12) at (-7.25, 1.5) {$s_{e_1}=s$};
		\node [style=none, rotate=90] (13) at (-7.25, -1.5) {$s_{e_2}=s$};
		\node [style=none] (14) at (-9.5, 4) {$(a)$};
		\node [style=none] (15) at (2, 4) {$(b)$};
		\node [style=black dot small] (16) at (3.5, 3) {};
		\node [style=none] (22) at (3.5, 0) {};
		\node [style=none] (23) at (6, 0) {$\leftrightarrow$};
		\node [style=none] (26) at (3, 0) {$e$};
		\node [style=none, rotate=90] (29) at (4.25, 0) {$s_{e}=0$};
		\node [style=none] (30) at (3, 3.5) {};
		\node [style=none] (31) at (3.5, 3.75) {};
		\node [style=none] (32) at (4, 3.5) {};
		\node [style=black dot small] (33) at (8.5, 3) {};
		\node [style=none] (34) at (8, 3.5) {};
		\node [style=none] (35) at (8.5, 3.75) {};
		\node [style=none] (36) at (9, 3.5) {};
		\node [style=black dot small] (37) at (-3, 3) {};
		\node [style=none] (38) at (-3.5, 3.5) {};
		\node [style=none] (39) at (-3, 3.75) {};
		\node [style=none] (40) at (-2.5, 3.5) {};
		\node [style=black dot small] (41) at (-8, 3) {};
		\node [style=none] (42) at (-8.5, 3.5) {};
		\node [style=none] (43) at (-8, 3.75) {};
		\node [style=none] (44) at (-7.5, 3.5) {};
		\node [style=none] (45) at (-8.5, -3.5) {};
		\node [style=none] (47) at (-7.5, -3.5) {};
		\node [style=black dot small] (48) at (-3, -3) {};
		\node [style=none] (49) at (-3.5, -3.5) {};
		\node [style=none] (50) at (-2.5, -3.5) {};
		\node [style=black dot small] (51) at (3.5, -3) {};
		\node [style=none] (53) at (3, -3.5) {};
		\node [style=none] (54) at (4, -3.5) {};
		\node [style=black dot small] (55) at (8.5, -3) {};
		\node [style=none] (56) at (8, -3.5) {};
		\node [style=none] (57) at (9, -3.5) {};
	\end{pgfonlayer}
	\begin{pgfonlayer}{edgelayer}
		\draw [style=mid arrow] (2) to (1);
		\draw [style=mid arrow] (6.center) to (2);
		\draw [style=semicircle line] (16) to (22.center);
		\draw (30.center) to (16);
		\draw (31.center) to (16);
		\draw (32.center) to (16);
		\draw (34.center) to (33);
		\draw (35.center) to (33);
		\draw (36.center) to (33);
		\draw (38.center) to (37);
		\draw (39.center) to (37);
		\draw (40.center) to (37);
		\draw (42.center) to (41);
		\draw (43.center) to (41);
		\draw (44.center) to (41);
		\draw [style=semicircle line] (37) to (5.center);
		\draw [style=semicircle line] (41) to (6.center);
		\draw (45.center) to (1);
		\draw (1) to (47.center);
		\draw (49.center) to (48);
		\draw (48) to (50.center);
		\draw (53.center) to (51);
		\draw (51) to (54.center);
		\draw (56.center) to (55);
		\draw (55) to (57.center);
		\draw [style=mid arrow] (5.center) to (48);
		\draw [style=mid arrow] (22.center) to (51);
	\end{pgfonlayer}
\end{tikzpicture}
	\end{center}
	\caption{The elementary moves on a marked PLCW decomposition.}\label{fig:ElemMoves}
\end{figure}

\begin{prop}
	The elementary PLCW moves induce isomorphisms of $r$-spin structures on $(S,M)$. 
\end{prop}

\begin{proof}
	This follows from the same reasoning as \cite[Prop.\,2.18]{Runkel:2018:rs}.
\end{proof}

\begin{prop}
	Let $(S,M)$ be a whole surface with boundary parametrization $\phi$ and boundary labels $x=(x^\mathrm{in},x^\mathrm{out})$, and let $\Sigma$ be a drilling of $(S,M)$. Fix a $\phi$-admissible PLCW decomposition $K_\bullet$ of $\Sigma$ and let $T_\bullet$ be the extension of $K_\bullet$ constructed in 
	Lemma~\ref{lem:ExtPLCWfromDrill}. 
	Then there is a bijection
	\[
		\mathcal{M}^{T}_{x}(S,M)/_{\sim_{\on{fix}}}\cong \mathcal{M}^{K}_{x}(\Sigma)/_{\sim_{\on{fix}}}.
	\]
\end{prop} 

\begin{proof}
	As in 
	Lemma~\ref{lem:ExtPLCWfromDrill} 
	it will suffice to work locally on a disk. We will denote by $J_\bullet$ the PLCW decomposition of the circle (free boundary component) obtained from drilling at $m\in M$. We will denote by $L_\bullet$ the PLCW decomposition of the disk, as in 
	Lemma~\ref{lem:ExtPLCWfromDrill}. 
	Note that $J_1\subset L_1\subset \overline{T}_1$, but $J_1\cap \overline{K}_1=\emptyset$. 
	
	Suppose given an admissible marking of $K_\bullet$. We choose edge orientations on $J_1$ to move counterclockwise around the boundary circle (with respect to the induced orientation from $S$) and orientations on $L_1\setminus J_1$ to point into the puncture $m$. For each element of $L_2$, we mark the side lying on the circle. We then set all edge indices for $J_1$ to be $0$.\footnote{It is worth pointing out that these choices of orientations, markings and edge indices are arbitrary, and merely made to give a convenient definition of a map.} These choices, together with 
	\eqref{eq:intvertcondition}, 
	then uniquely determine the edge labels for the edges of $L_1\setminus J_1$. Moreover, since there is no compatibility condition at the marked vertex $m$, these choices determine an admissible marking of $T_\bullet$ extending the marking of $K_\bullet$, and thus, a map of sets 
	\[
		\phi:\mathcal{M}^{K}_{x}(\Sigma)\to \mathcal{M}^{T}_{x}(S,M).
	\]
	
	In the other direction, given an admissible marking of $T_\bullet$ in $\mathcal{M}^{T}_{x}(S,M)$, we simply forget those edge orientations, edge indices, and edge labels not needed for a marking of $K_\bullet$, determining a map of sets 
	\[
		\psi: \mathcal{M}^{T}_{x}(S,M)\to \mathcal{M}^{K}_{x}(\Sigma).
	\]
	It follows from the construction that $\psi\circ \phi=\on{id}$. 
	
	Moreover, given $(m,o,s)\in \mathcal{M}^{T}_{x}(S,M)$, both $(m,o,s)$ and $\phi(\psi(m,o,s))$ restrict to the same marking in $\mathcal{M}^{K}_{x}(\Sigma)$, meaning that $(m,o,s)$ and $\phi(\psi(m,o,s))$ can only differ on $L_\bullet$. The edge orientations and markings can be made to agree via application of moves (1) and (2) from 
	Figure~\ref{fig:PLCWfix}, 
and deck transformations can be used to bring the edge indices for edges in $J_1$ to $0$. Since the compatibilities then determine the remaining edge indices in $L_1$, we thus have 
	\[
	(m,o,s)\sim_{\on{fix}}\phi(\psi(m,o,s)).
	\]
	
	Since the fixed PLCW moves are defined locally, both $\phi$ and $\psi$ descend to maps of quotients by $\sim_{\on{fix}}$, yielding the desired bijection. 
\end{proof}

\subsubsection{Gluing PLCW decompositions}

As in \cite{Novak:2015phd}, one can glue two marked PLCW decompositions along closed collars, and the induced $r$-spin structure on the glued surface is precisely that given by gluing the two $r$-spin structures on the unglued structures. Since the gluing is local, we will give the construction for gluing along a single boundary component. 

\begin{const}\label{const:gluing}
	Let $(S_1,M_1)$ and $(S_2,M_2)$ be two punctured surfaces equipped with open-closed boundary parametrizations $\phi_1$ and $\phi_2$, and boundary labels $x^\mathrm{in,1},x^\mathrm{out,1}$ and $x^\mathrm{in,2},x^\mathrm{out,2}$ respectively. Let $K_\bullet^1$ and $K_\bullet^2$ be $\phi^1$- and $\phi^2$-admissible PLCW decompositions with admissible markings $(m^1,o^1,s^1)$ and $(m^2,o^2,s^2)$ respectively.
	
	Denote by $\psi^1:U\to S_1$ the restriction of $\phi^1$ to a single outgoing boundary component, and by $\psi^2:V\to S_2$ the restriction of $\phi^2$ to a single incoming boundary component such that the boundary labels 
	$x^\mathrm{out,1}$ and $x^\mathrm{in,2}$ 
agree (note that this condition is vacuous if the boundary components are open). We denote by $S$ the gluing of $S_1$ and $S_2$ along $\psi_1$ and $\psi_2$, and we write $M:=M_1\amalg M_2\subset S$. Note $K_\bullet^1$ and $K_\bullet^2$ both restrict to the same PLCW decomposition $L_\bullet$ of $c(U)=c(V)$. 
	
	These data give rise to a new PLCW decomposition $K_\bullet$ of $\Sigma$ with cells: 
	\begin{align*}
	K_0 & = K^1_0\coprod_{L_0} K^2_0\\
	K_1 & = K^1_1 \coprod_{L_1} K^2_1\\
	K_2 & = K^1_2 \coprod K^2_2
	\end{align*}
	
	We suppose, without loss of generality, that the edge orientations induced on $L_\bullet$ by $o^1$ and $o^2$ are both those induced by the orientation of $\Sigma_1$ (the opposite of that induced by $\Sigma_2$). We then get edge orientations $o$ on $K_1$, and markings $m$ on $K_2$ induced by $o^1,o^2$ and $m^1,m^2$. 
	
	On edges not contained in $L_\bullet$, $s^1$ and $s^2$ induce edge indices $s$ on $L_\bullet$, which are admissible at all vertices not contained in $L_\bullet$.  Suppose $e\in L_1$ is an edge. We then define 
	\begin{align}
	s_e:= s^1_e+s_e^2+1
	\label{eq:glueing-rule-c}
	\end{align}
	yielding an admissible marking $(m,o,s)$ on $K_\bullet$. 
\end{const}

\begin{prop}
	With notation as in 
	Construction~\ref{const:gluing}, 
	there is an isomorphism of $r$-spin surfaces 
	\[
	\Sigma(m,o,s)\cong \Sigma_2(m_2,o_2,s_2)\circ 
		\Sigma_1(m_1,o_1,s_1).
	\]
\end{prop}

\begin{proof}
	This can be seen from, either the gluing procedure of \cite[Sec.\,4.5]{Novak:2015phd} or the holonomy formula in \cite[Prop.\,2.15]{Runkel:2018:rs}.
\end{proof}

\begin{rmk}
	We can apply an analogous gluing procedure, even if the restrictions of $K^1_\bullet$ and $K^2_\bullet$ to $c(U)=c(V)$ have more than one 1-cell, so long as both still restrict to the same PLCW decomposition $L_\bullet$.  We can apply this more general gluing procedure to obtain the following useful corollary. 
\end{rmk}

\begin{center}
	\begin{tikzpicture}[scale=0.5]
	\begin{pgfonlayer}{nodelayer}
	\node [style=none] (0) at (-5, 0) {};
	\node [style=none] (1) at (0, 5) {};
	\node [style=none] (2) at (5, 0) {};
	\node [style=none] (3) at (0, -5) {};
	\node [style=none] (4) at (-2, 1) {};
	\node [style=none] (5) at (0, 2) {};
	\node [style=none] (6) at (-1, -1) {};
	\node [style=none] (7) at (0.25, 0.5) {};
	\node [style=none] (8) at (2.25, 0.5) {};
	\node [style=none] (9) at (2.25, -1.75) {};
	\node [style=none] (10) at (0.25, -2.5) {};
	\node [style=none] (11) at (-2.25, -6.75) {};
	\node [style=none] (12) at (6.75, -4.5) {};
	\node [style=none] (13) at (6.75, 2.75) {};
	\node [style=none] (14) at (2, 6.75) {};
	\node [style=none] (15) at (-6.75, 4.25) {};
	\node [style=none] (16) at (-6.75, -3.75) {};
	\node [style=none] (17) at (7, 1) {};
	\end{pgfonlayer}
	\begin{pgfonlayer}{edgelayer}
	\draw [bend right=45] (0.center) to (3.center);
	\draw [bend right=45] (3.center) to (2.center);
	\draw [bend right=45] (1.center) to (0.center);
	\draw [bend left=45] (1.center) to (2.center);
	\draw (4.center) to (6.center);
	\draw (6.center) to (7.center);
	\draw (7.center) to (5.center);
	\draw (5.center) to (4.center);
	\draw (7.center) to (8.center);
	\draw (8.center) to (9.center);
	\draw (9.center) to (10.center);
	\draw (10.center) to (6.center);
	\draw (6.center) to (16.center);
	\draw (4.center) to (15.center);
	\draw (5.center) to (14.center);
	\draw (7.center) to (13.center);
	\draw (9.center) to (12.center);
	\draw (10.center) to (11.center);
	\draw (8.center) to (17.center);
	\end{pgfonlayer}
	\path[fill=black, opacity=0.2] (6.center)--(4.center)--(5.center)--(7.center)--(8.center)--(9.center)--(10.center)--cycle;
	\end{tikzpicture}
	\begin{tikzpicture}[scale=0.5]
	\draw (0,0) circle (5);
	\foreach \th in {10,30,80,140,200,260,330}{
		\draw (0,0) to (\th:6);
	};
	\path (-8,0) node {$\rightsquigarrow$}; 
	\end{tikzpicture}
\end{center}

\begin{cor}\label{cor:contract-edges-inside-disk}
	Let $\Sigma$ be a surface with boundary parametrization $\phi$, PLCW decomposition $K_\bullet$ and marking $(m,s,o)$. Let $\gamma:D^2\to \Sigma$ be an oriented embedding such that the restriction to the boundary $\gamma|_{S^1}:S^1\to \Sigma$ does not intersect $K_0$, only intersects the edges in $K_1$ transversely, and does not intersect any edge in $K_1$ more than once. Assume further that for any $\sigma\in K_2$, $\gamma(S^1)\cap\sigma$ consists of at most a single interval. 
	\begin{enumerate}
		\item We can replace $K_\bullet$ by a PLCW decomposition $H_\bullet$, which agrees with $K_\bullet$ outside the image of $\gamma$ and contains only one vertex in the interior of $\gamma(D^2)$.
		\item There is a marking of $H_\bullet$, which agrees with the marking of $K_\bullet$ outside of $\gamma(D^2)$, which induces an isomorphic $r$-spin structure on $\Sigma$. 
	\end{enumerate}  
\end{cor}

\begin{proof}
	Denote by $V$ the set of points of intersection of $\gamma|_{S^1}$ with edges in $K_1$, and order them $(v_1,v_2,\ldots,v_n)$ compatibly with the orientation of $S^1$. Assume without loss of generality that, for each point of intersection, the edge markings and orientations of $K_\bullet$ are such that, for all $1\leq i\leq n$, 
	\begin{itemize}
		\item the edge containing $v_i$ is oriented into $\gamma(D^2)$, 
		\item the edge $e$ containing $v_i$ is the marked edge of the 2-cell to the left of $e$ (with respect to the orientations of $e$ and $\Sigma$).
	\end{itemize}
	Using moves of type (a) from 
	Figure~\ref{fig:ElemMoves}, 
	we can then split the edges containing $v_i$ at the points $v_i$. Using moves of type (b) from 
	Figure~\ref{fig:ElemMoves}, 
	we can add $1$-cells $e_i$ to $e_{i+1}$ to the PLCW decomposition (with edge index 0 in all cases), so that we obtain a new (equivalent) marked PLCW decomposition which now restricts to a PLCW decomposition of $\gamma(S^1)$. 
	
	Cutting along $\gamma(S^1)$ displays $\Sigma$ as the gluing of two $r$-spin surfaces along a common closed boundary component with boundary label $0$. Since one of these is the disk, we may take a different marked PLCW decomposition of the disk representing the same $r$-spin structure. Choosing the radial PLCW decomposition from 
	Figure~\ref{fig:PLCWextension} 
	yields the desired result. 
\end{proof}

\subsection{Structured graphs and PLCW decompositions}
\label{sec:str-graph}

In what follows, we will make extensive use of the notion of structured graphs as defined in \cite{Dyckerhoff:2015csg}, and elaborated on in \cite{Stern:2016stft}. However, to more easily make connections with marked PLCW decompositions, our definitions will differ slightly from those found in the above sources. For this reason, we begin by reviewing the formalism of graphs structured over the $r$-cyclic category $\Lambda_r$. Where applicable, we will comment on the differences between the variant discussed here and those presented in \cite{Dyckerhoff:2015csg,Stern:2016stft}. For the definition of the $r$-cyclic category, as well as the conventions we use, see 
Appendix~\ref{app:Lambda-r}. 
In particular, the notations $\psi^{i,j}_k$ and $\theta^{i,j,k}_\ell$ are defined there. 

In the case where $r=0$ it is necessary to use the \textsl{paracyclic category} $\Lambda_\infty$ in place of $\Lambda_r$. While the model still works in this case, and bears effectively the same relation to marked PLCW decompositions, we will restrict our attentions in this section to $0<r<\infty$ so as to avoid the need for two conventions.

\begin{rmk}
	For the entirety of this section, we fix the following convention for marked PLCW decompositions. An incoming  boundary edge is always given the edge orientation opposite that induced by the surface. An outgoing boundary edge always has edge orientation induced by the surface. We can do this without loss of generality by applying fixed PLCW moves of type 1.  
\end{rmk}

\begin{defn}
	A \textsl{graph}\index{graph} $\Gamma$ consists of a finite set $H$ of half-edges, a finite set $V$ of vertices, an involution $\eta: H\to H$, and a map $s: H\to V$. 

	We will call the orbits of order $2$ under $\eta$ the \textsl{edges} of $\Gamma$, and denote the set of edges by $E$. We will call the orbits of order $1$ under $\eta$ the \textsl{external half-edges}. We will consistently abuse notation by identifying an external half-edge $\{h\}$ with the half-edge $h$. The set $H_v:=s^{-1}(v)$ for $v\in V$ consists of the 
		half-edges 
	\textsl{adjacent to $v$}, and $|H_v|$ is the \textsl{valency} of $v$. We additionally prohibit vertices of valency $0$. 
\end{defn}

\begin{const}
	Given a graph $\Gamma$, we construct the \textsl{incidence category $I(\Gamma)$ of $\Gamma$}\index{incidence category} as follows:
	\begin{enumerate}
		\item $I(\Gamma)$ has an object for each vertex or edge of $\Gamma$.
		\item If $e=\{h,h^\prime\}$ is an edge of $\Gamma$, with $s(h^\prime)=v$, then $I(\Gamma)$ has a morphism $v\to e$. 
	\end{enumerate}
	$I(\Gamma)$ additionally has formally defined identities at every object. Since there are no non-trivial pairs of composable morphisms, this uniquely defines a category. 
	We further define the \textsl{augmented incidence category} $A(\Gamma)$ to, in addition to the above, have
	\begin{enumerate}
		\setcounter{enumi}{2}
	\item an object for every external half-edge of $\Gamma$,
		\item for each external half-edge $h$ of $\Gamma$, a morphism $s(h)\to h$. 
	\end{enumerate}
	We construct a functor, the \textsl{incidence diagram} \index{incidence diagram}
	\[
	I_\Gamma: I(\Gamma)\to \Set 
	\]
	as follows. We send $v\mapsto H_v$ and $e=\{h,h^\prime\}\mapsto \{h,h^\prime\}$. The morphism $s(h)\to \{h,h^\prime\}$ is sent to the map $H_v\to \{h,h^\prime\}$ given by
	\begin{equation*}
	k  \mapsto \begin{cases}
	h^\prime & k=h\\
	h & \text{else}
	\end{cases}
	\end{equation*}
	
	Given a partition of the external half-edges of $\Gamma$ into two sets $\on{In}(\Gamma)$ and $\on{Out}(\Gamma)$, we can extend the incidence diagram to the augmented incidence category to obtain a functor 
	\[
	A_\Gamma: A(\Gamma)\to \Set 
	\]
	as follows. To each external half-edge $h$, we assign the set $\{0,1\}$, and to the morphism $s(h)\to h$, we assign the map $H_v\to \{0,1\}$, given by the map 
		\begin{equation*}
			k  \mapsto \begin{cases}
			0 & k=h\\
			1 & \text{else}
			\end{cases}
		\end{equation*}
		if $h\in \on{In}(\Gamma)$ and the map 
		\begin{equation*}
		k  \mapsto \begin{cases}
		1 & k=h\\
		0 & \text{else}
		\end{cases}
		\end{equation*}
		if $h\in \on{Out}(\Gamma)$.
	We will call the functor $A_\Gamma$ the \textsl{augmented incidence diagram} of $\Gamma$. 
\end{const}

\begin{ntt}
	We will refer to the realization of the nerve $|N(A(\Gamma))|$ as the \textsl{realization of $\Gamma$}, and denote it by $|\Gamma|$. 
\end{ntt}

\begin{defn}
	Let $\Gamma$ be a graph. A \textsl{$\Lambda_r$-structure on $\Gamma$} consists of a functor $\tilde{A}_\Gamma$ and a natural isomorphism $\mu$ making the diagram 
	\[
	\begin{tikzcd}
	& \Lambda_r\arrow[d, "\lambda_r"]\\
	A(\Gamma)\arrow[ur,"\tilde{A}_\Gamma"]\arrow[r,"A_\Gamma"'] & \Set 
	\end{tikzcd}
	\]
	commute up to $\mu$, subject to the condition that, for any external half-edge $h$ of $\Gamma$, the component $\mu_h$ of $\mu$ is the identity on the set $\{0,1\}$. 
		
	An \textsl{isomorphism of $\Lambda_r$-structures $(\tilde{A}_\Gamma,\mu^A)\to(\tilde{B}_\Gamma,\mu^B)$ on $\Gamma$} consists of a natural isomorphism $\theta: \tilde{A}_\Gamma \Rightarrow \tilde{B}_\Gamma$ such that the diagram 
	\[
	\begin{tikzcd}
	\lambda_r\circ\tilde{A}_\Gamma\arrow[rr,Rightarrow, "\lambda_r\circ\theta"]\arrow[dr,Rightarrow,"\mu^A"'] & & \lambda_r\circ\tilde{B}_\Gamma\arrow[dl,Rightarrow,"\mu^B"]\\
	& A_\Gamma  &
	\end{tikzcd}
	\]
	commutes, and such that the component $\theta_h$ at any external half-edge $h$ is the identity on $[1]_r$. 
\end{defn}

\begin{ntt}
	We will denote the set of $\Lambda_r$-structures on $\Gamma$ by $\Lambda_r(\Gamma)$, and will denote the equivalence classes under isomorphism of $\Lambda_r$-structures by 
	\[
	\Lambda_r(\Gamma)_{/\sim_{\on{iso}}}. 
	\]
\end{ntt}

\begin{rmk}
	The definition above of a $\Lambda_r$-structure on $\Gamma$ is effectively equivalent to that of an augmented $\Lambda_r$-structured graph from \cite{Dyckerhoff:2015csg} or \cite{Stern:2016stft}. The natural isomorphisms are introduced here to avoid the necessity of working with more general $\Lambda_r$-structured sets. The isomorphisms of $\Lambda_r$-structures defined above correspond to a special case of the more general morphisms of structured graphs introduced in \cite{Dyckerhoff:2015csg}. 
\end{rmk}

\begin{const}
	Given a whole surface $(S,M)$ with a PLCW decomposition $K_\bullet$ of $S$ with respect to $M$, we define the dual graph $\Gamma_K$ of $K_\bullet$ to have vertices $V:= K_2$, half-edges 
	\[
		H=\left\lbrace (a,b)\in K_2\times \overline{K}_1\mid b
		\text{ is an edge of } a\right\rbrace,
	\]
	valency map $s: H  \to V$ given by $(a,b)  \mapsto a$
	and involution $\eta:H\to H$ which sends $(a,b)$ to $(a^\prime,b)$, where $a^\prime$ is 
		\begin{enumerate}
			\item the other 2-cell abutting $b$, if $b$ does not lie in $\partial S$,\footnote{Note that this could be the same as $a$, depending on the PLCW decomposition.} 
			\item $a$ if $b$ lies in the boundary $\partial S$. 
		\end{enumerate}
 
	Note that we can identify $|\Gamma_K|$ with a subcomplex of the 1-skeleton of the barycentric subdivision of $K_\bullet$, and thereby obtain an embedding 
	\[
	\gamma_K: |\Gamma_K|\to S\setminus M 
	\]
	inducing a bijection between the set of parameterized boundary components and the set $\partial\Gamma_K$ of external half-edges of $\Gamma_K$. Moreover, if $M=K_0\setminus \partial S$, then $\gamma_K$ is a homotopy equivalence. 
\end{const}

\begin{rmk}
	Let $\gamma: |\Gamma| \to S\setminus M$
	be an embedding which is a homotopy equivalence, such that $\gamma$ defines a bijection between $\partial \Gamma$ and the set of parameterized boundary components of $S$. Then there is a PLCW decomposition $K_\bullet$ of $S$ with respect to $M$ and an isomorphism $\phi:\Gamma\cong \Gamma_K$ such that $\gamma_K\circ |\phi|$ is homotopic to $\gamma$.
\end{rmk}

\begin{rmk}
	An embedding of a graph $\Gamma$ into an oriented surface $S$ induces a cyclic order on the set of half-edges incident to any vertex or edge in $\Gamma$. Given such an embedding, we will consider only those $r$-cyclic structures on $\Gamma$ which induce precisely these cyclic orders. We will denote the set of such by $\Lambda_r(\Gamma)^S$. By a \textsl{$\Lambda_r$-structured graph embedded in $(S,M)$}, we will mean an isotopy class of embeddings satisfying this condition. 
\end{rmk}

We now aim to relate markings of a PLCW decomposition $K_\bullet$ of $(S,M)$ to $\Lambda_r$-structures on the dual graph of $K_\bullet$. The fundamental constructions in this regard are as follows.

\begin{const}\label{const:StructuredGraphFromPLCW}
	Suppose given an admissible marked PLCW decomposition $(K_\bullet, m,s,o)$ of a whole surface $(S,M)$. The dual graph $\Gamma$ comes equipped with an embedding 
	\[
	\gamma:|\Gamma|\to S\setminus M
	\]
	which is a homotopy equivalence.
	We can put a $\Lambda_r$-structure on $\Gamma$ as follows:
	\begin{enumerate}
		\item To each vertex $v$ of $\Gamma$, corresponding to an $(n+1)$-gon $\sigma_v\in K_2$, we assign  $[n]_r\in \Lambda_r$. Note that the edge marking together with the cyclic order on the set $E_v$ of edges of $\sigma_v$ defines a bijection 
		\[
		\mu_v: \{0,1,\ldots,n\}\to E_v.
		\]
		\item To each internal edge $e$ of $\Gamma$ corresponding to an internal edge $u_e\in K_1$, we assign $[1]_r\in \Lambda_r$. Note that the oriented unit normal to $u_e$ in $S$ defines a linear order on the set $U_e$ of 2-cells adjacent to $u_e$, and therefore an isomorphism 
		\[
		\mu_e:\{0,1\}\to U_e.
		\]
		\item To each external half-edge $e$ of $\Gamma$ corresponding to a boundary edge $u_e\in K_1$, we assign $[1]_r$. 
		\item For each morphism $v\to e$ in the incidence category $I(\Gamma)$, where $e$ is an internal edge, we assign a morphism as follows:
		\begin{enumerate}
			\item If the edge orientation of $u_e$ is opposite that induced by the orientation of $\sigma_v$, then we send $v\to e$ to the morphism 
			\[
			\psi_n^{k,n}\circ \tau^{-(n+1)s_e}:[n]_r\to [1]_r
			\]
			where $\sigma_v$ is an $(n+1)$-gon, $k=\mu_v^{-1}(u_e)$, and $s_e$ is the edge label of $u_e$. 
			\item if the edge orientation of $u_e$ is that induced by the orientation of $\sigma_v$, then we send $v\to e$ to the morphism 
			\[
			\psi_n^{k,0}: [n]_r\to [1]_r
			\]
			where $\sigma_v$ is an $(n+1)$-gon and $k=\mu_v^{-1}(u_e)$.
		\end{enumerate}
		\item For each morphism $v\to e$ in the incidence category $I(\Gamma)$, where $e$ is an external half-edge, we assign a morphism as follows: 
		\begin{enumerate}
			\item If the edge orientation of $u_e$ is opposite that induced by the orientation of $\sigma_v$ (i.e.\ if $e$ is incoming), then we send $v\to e$ to the morphism 
			\[
			\psi_n^{k,n}\circ \tau^{-(n+1)s_e}:[n]_r\to [1]_r
			\]
			where $\sigma_v$ is an $(n+1)$-gon, $k=\mu_v^{-1}(u_e)$, and $s_e$ is the edge label of $u_e$. 
			\item if the edge orientation of $u_e$ is that induced by the orientation of $\sigma_v$ (i.e.\ if $u$ is outgoing), then we send $v\to e$ to the morphism 
			\[
			\psi_n^{k,0}\circ \tau^{(n+1)(s_e+1)}:[n]_r\to [1]_r
			\]
			where $\sigma_v$ is an $(n+1)$-gon, $k=\mu_v^{-1}(u_e)$.
		\end{enumerate}		
	\end{enumerate}
	This defines a functor $F_K:I(\Gamma)^{\on{op}}\to \Lambda_r$. 
	Moreover, the isomorphisms $\mu_e$ and $\mu_v$ define a natural isomorphism $\mu$ from 
	\[
	I(\Gamma)^{\on{op}}\overset{F_K}{\longrightarrow} \Lambda_r\to \Set 
	\]
	to the incidence diagram $I_\Gamma:I(\Gamma)\to \Set$.
\end{const}

\begin{const}
	Let $\Gamma$ be a graph with a $\Lambda_r$-structure $(\tilde{A}_\Gamma, \mu)$. For any object $o\in A(\Gamma)$ which is not an external half-edge, an automorphism $\phi: \tilde{A}_\Gamma(o)\overset{\cong}{\to} \tilde{A}_\Gamma(o)$ defines a new $\Lambda_r$-structure $(\tilde{B}_\Gamma,\nu)$ by precomposing morphisms $\tilde{A}_\Gamma(f)$ into $o$ with $\phi$  (if $o$ is an edge) or postcomposing morphisms $\tilde{A}_\Gamma(f)$ out of $o$ with $\phi^{-1}$ (if $o$ is a vertex).

	We then have a canonical natural transformation $\eta(o,\phi):\tilde{A}_\Gamma\Rightarrow \tilde{B}_\Gamma$ with $o$-component $\phi$, inducing an isomorphism of $\Lambda_r$-structures on $\Gamma$. 
\end{const}

\begin{defn}
	We call isomorphisms of $\Lambda_r$-structures on $\Gamma$ of the form $\eta(o,\phi)$ for some $o$ and $\phi$ the \textsl{elementary isomorphisms}\index{elementary isomorphism}. If $o$ is an edge, we call $\eta(o,\phi)$ and \textsl{edge isomorphism} an if $o$ is a vertex we call $\eta(o,\phi)$ a \textsl{vertex isomorphism}. 	
\end{defn}

\begin{lem}\label{lem:factoringgraphisos}
	Every isomorphism of $\Lambda_r$-structures on a graph $\Gamma$ admits a factorization in terms of elementary isomorphisms. 
\end{lem}

\begin{proof}
	Given an isomorphism $\kappa: (\tilde{A}_\Gamma,\mu^A)\to(\tilde{B}_\Gamma,\mu^B)$, choose an order $\ell_1,\ell_2,\cdots \ell_k$ on the objects of $I(\Gamma)$ which are not external half-edges. It is immediate from the definitions that 
	\[
	\kappa=\eta(\ell_1,\kappa_{\ell_1})\circ\eta(\ell_2,\kappa_{\ell_2}) \circ \cdots \circ \eta(\ell_k,\kappa_{\ell_k}). \qedhere
	\] 
\end{proof}

\begin{prop}\label{prop:fixedmovesandIsos}
	Let $(S,M)$ be a punctured surface, $\phi:\coprod U_i\to S$ an open boundary parametrization, and $K_\bullet$ a $\phi$-admissible PLCW decomposition of $S$ with $M=K_0\setminus \partial S$. Then 
	Construction~\ref{const:StructuredGraphFromPLCW} 
defines a bijection 
	\[
	\Lambda_r(\Gamma_K)^S_{/\sim_{\on{iso}}}\cong \mathcal{M}^K(S,M)_{/\sim_{\on{fix}}}.
	\] 
\end{prop}

\begin{proof}
	We need only check that elementary isomorphisms are sent to fixed PLCW moves and vice versa. We defer the necessary computations to 
	Appendix~\ref{app:appendixproofs}.
\end{proof}

\begin{const}
	Given a graph $\Gamma$ and an edge $e=\{h,h^\prime\}$ of $\Gamma$ connecting two distinct vertices $v_1$ and $v_2$, we define the  $\Gamma_{/e}$ to be the graph obtained by contracting the edge $e$.
	
	This construction defines an obvious functor $C_e: A(\Gamma)\to A(\Gamma_{/e})$, and a natural transformation 
	\[
	\rho^e: A_{\Gamma_{/e}}\circ C_e \Rightarrow A_\Gamma. 
	\]
	Which is the identity on objects other than $v_1$, $v_2$, and $e$. In particular, $\rho^e$ determines (and is determined by) a diagram 
	\begin{equation}
		\begin{tikzcd}
			A_{\Gamma_{/e}}(v)\arrow[r, "\rho^e_{v_1}"]\arrow[d,"\rho^e_{v_2}"'] & A_\Gamma(v_1)\arrow[d]\\
			A_{\Gamma}(v_2)\arrow[r] & A_\Gamma(e)
		\end{tikzcd}
	\end{equation}
	
	In the presence of a $\Lambda_r$ structure $(\tilde{A}_\Gamma,\mu^A)$ on $\Gamma$, we will call a $\Lambda_r$-structure $(\tilde{A}_{\Gamma_{/e}},\mu^e)$ on $\Gamma_{/e}$ the \textsl{(structured) contraction of $(\tilde{A}_\Gamma,\mu^A)$} if 
	\begin{enumerate}
		\item $\rho_e$ lifts to a natural natural isomorphism $\tilde{\rho}^e: \tilde{A}_{\Gamma_{/e}}\circ C_e \Rightarrow \tilde{A}_\Gamma$ which is the identity away from $v_1$, $v_2$, and $e$.
		\item The induced diagram  
		\begin{equation}\label{diag:structcont}
			\begin{tikzcd}
				\tilde{A}_{\Gamma_{/e}}(v)\arrow[r, "\tilde{\rho}^e_{v_1}"]\arrow[d,"\tilde{\rho}^e_{v_2}"'] & \tilde{A}_\Gamma(v_1)\arrow[d]\\
				\tilde{A}_{\Gamma}(v_2)\arrow[r] & A_\Gamma(e)
			\end{tikzcd}
		\end{equation}
		is a pullback diagram
	\end{enumerate}
	Note that the structured contraction of $(\tilde{A}_\Gamma,\mu^A)$ always exists and is uniquely defined up to unique isomorphism.
	
	The induced map $|C_e|:|\Gamma|\to |\Gamma_{/e}|$ is a homotopy equivalence which is the identity away from the closed subset 
	\begin{center}
		\begin{tikzpicture} 
		\draw[fill=black] (0,0) circle (0.05) node (e){} node[above] {$e$};
		\draw[fill=black] (-1,0) circle (0.05)node (v1){} node[above] {$v_1$};
		\draw[fill=black] (1,0) circle (0.05) node (v2){} node[above] {$v_2$}; 
		\draw[->] (v1) to (e);
		\draw[->] (v2) to (e); 
		\draw[fill=black] (4,0) circle (0.05) node[above] {$v$};
		\path (2.5,0) node {$\mapsto$};
		\end{tikzpicture}
	\end{center}
	One can choose a homotopy inverse $W_e: |\Gamma_e|\to |\Gamma|$ which is the identity on all objects except for $e$, $v_1$, and $v_2$. Indeed, we can choose this homotopy such that $v\mapsto e$, and such that, for any edge or half-edge $u$ attached to $v_1$ (or $v_2$) in $\Gamma$, the path $v\to u$ in $|\Gamma_{/e}|$ is sent to the path given by $e\to v_1\to u$ in $|\Gamma|$. Consequently, given an embedding $\gamma:|\Gamma|\to S\setminus M$ which meets every boundary component a set of external half-edges  and is a homotopy equivalence, we obtain an embedding 
	\[
	\gamma_e:|\Gamma_{/e}|\to S
	\]
	which meets every boundary component a set of external half-edges  and is a homotopy equivalence. 
\end{const}

\begin{prop}\label{prop:internalcontractions}
	Let $(S,M)$ be a marked surface, $\phi:\coprod U_i\to S$ an open boundary parametrization, and $K_\bullet$ a $\phi$-admissible PLCW decomposition of $S$ with $M=K_0\setminus\partial S$. Let $\Gamma$ be the dual graph to $K_\bullet$,  let $e$ be an internal edge of $K_\bullet$, and let $\mathscr{M}$ an admissible marking of $K_\bullet$ with $s_e=0$ and with marking and orientation locally as in 
	Figure~\ref{fig:ElemMoves} (b). 
Let $T_\bullet$ be the PLCW decomposition and $\mathscr{M}^\prime$ the marking obtained from $K_\bullet$ and $\mathscr{M}$ via the move of 
Figure~\ref{fig:ElemMoves} (b). 
Then
	\begin{enumerate}
		\item\label{intcont:dualgraph} $\Gamma_{/e}$ can be identified with the dual graph of $T_\bullet$,
		\item\label{intcont:homotopy} $\gamma_e:|\Gamma_{/e}|\to S\setminus M$ is homotopic to the embedding induced by $T_\bullet$,
		\item\label{intcont:LambdarStruct} the $\Lambda_r$-structure induced on $\Gamma_{/e}$ by $\mathscr{M}^\prime$ is the structured contraction of the $\Lambda_r$-structure on $\Gamma$ induced by $\mathscr{M}$.
	\end{enumerate} 
\end{prop}

\begin{proof}
	Statements \ref{intcont:dualgraph} and \ref{intcont:homotopy} are immediate from the definitions of $|\Gamma_{/e}|$ and $W_e$. Statement \ref{intcont:LambdarStruct} is checked by direct computation in the $r$-cyclic category, analogous to those in the proof of 
	Proposition~\ref{prop:fixedmovesandIsos}. 
	We leave these computations as an exercise to the reader.
\end{proof} 

\begin{cor}\label{cor:equivalence of models}
	Equivalence classes of $\Lambda_r$-structured graphs embedded in $S\setminus M$ under contraction and isomorphism are in bijection with equivalence classes of marked PLCW decompositions $K_\bullet$ of with $K_0\setminus \partial S=M$ under fixed and elementary PLCW moves fixing $K_0$. 
\end{cor}

\begin{rmk}
	It is immediate from the definitions that the correspondence between $\Lambda_r$-structured graphs and marked PLCW decompositions respects gluing. The upshot of 
	Propositions~\ref{prop:fixedmovesandIsos} and \ref{prop:internalcontractions} 
is thus that the open sector of an open-closed $r$-spin topological field theory can be equivalently described in terms of marked PLCW decompositions or $\Lambda_r$-structured graphs. Since \cite{Stern:2016stft} already provides a classification in terms of the latter, these results allow us to import the relations among open generators into our description in terms of marked PLCW decompositions. 
\end{rmk}

	\section{Frobenius algebras}\label{sec:FrobAlg}

\subsection{Knowledgeable \texorpdfstring{$\Lambda_r$}{Lambda\_r}-Frobenius algebras}
Let $\Sc$ denote a strict symmetric monoidal category with tensor unit $\Ib$ and braiding $\sigma$.
Let $r\in\Zb_{\ge0}$ and consider objects $\{C_x\}_{x\in\Zb/r}$ in $\Sc$ 
together with morphisms
\begin{align}
	\mu_{x,y}&:C_x\otimes C_y\to C_{x+y-1}& \eta_1&:\Ib\to C_1
	\label{eq:C-morphisms-alg}\\
	\Delta_{x,y}&:C_{x+y+1}\to C_x\otimes C_y& \varepsilon_{-1}&:C_{-1}\to \Ib
	\label{eq:C-morphisms-coalg}\ .
\end{align}
We draw these morphisms as:
\begin{align}
	\begin{tikzpicture}
	\begin{pgfonlayer}{nodelayer}
		\node [style=square] (0) at (-0.5, 0) {\scriptsize{$x,y$}};
		\node [style=none] (1) at (-1, 0) {};
		\node [style=none] (2) at (0, 0) {};
		\node [style=none] (3) at (-0.5, 0) {};
		\node [style=none, label={above:$C_{x+y-1}$}] (4) at (-0.5, 1) {};
		\node [style=none, label={below:$C_x$}] (5) at (-1, -1) {};
		\node [style=none, label={below:$C_y$}] (6) at (0, -1) {};
		\node [style=none] (7) at (-2.5, 0) {$\mu_{x,y}=$};
		\node [style=none, label={above:$C_{x}$}] (8) at (7, 1) {};
		\node [style=none, label={above:$C_{y}$}] (9) at (8, 1) {};
		\node [style=none, label={below:$C_{x+y+1}$}] (10) at (7.5, -1) {};
		\node [style=square] (11) at (7.5, 0) {\scriptsize{$x,y$}};
		\node [style=none] (12) at (7, 0) {};
		\node [style=none] (13) at (8, 0) {};
		\node [style=none] (14) at (7.5, 0) {};
		\node [style=none] (15) at (5.5, 0) {$\Delta_{x,y}=$};
		\node [style=none, label={above:$C_{1}$}] (16) at (3.5, 1) {};
		\node [style=none] (17) at (2, 0) {$\eta_{1}=$};
		\node [style=white dot] (18) at (3.5, -1) {};
		\node [style=none, label={below:$C_{-1}$}] (19) at (11.5, -1) {};
		\node [style=none] (20) at (10, 0) {$\varepsilon_{-1}=$};
		\node [style=white dot] (21) at (11.5, 1) {};
	\end{pgfonlayer}
	\begin{pgfonlayer}{edgelayer}
		\draw (1.center) to (5.center);
		\draw (2.center) to (6.center);
		\draw (4.center) to (3.center);
		\draw (16.center) to (18);
		\draw (8.center) to (12.center);
		\draw (9.center) to (13.center);
		\draw (14.center) to (10.center);
		\draw (21) to (19.center);
	\end{pgfonlayer}
\end{tikzpicture}
	\label{eq:C-morphisms}
\end{align}
Note that, in our convention, string diagrams should be read as running from bottom to top. 
We furthermore define morphisms 
with $x\in\Zb/r$
\begin{align}
	\begin{tikzpicture}
	\begin{pgfonlayer}{nodelayer}
		\node [style=none] (0) at (-3.5, 0) {$N_x:=$};
		\node [style=square] (1) at (-0.5, -1) {\scriptsize{$x,-x$}};
		\node [style=none] (2) at (0, -1) {};
		\node [style=none] (3) at (-1, -1) {};
		\node [style=none] (4) at (-1, -0.5) {};
		\node [style=none] (5) at (0, -0.5) {};
		\node [style=white dot] (6) at (-0.5, -2) {};
		\node [style=square] (7) at (-0.5, 1) {\scriptsize{$x,-x$}};
		\node [style=none] (8) at (-1, 0.5) {};
		\node [style=none] (9) at (0, 0.5) {};
		\node [style=none] (10) at (0, 1) {};
		\node [style=none] (11) at (-1, 1) {};
		\node [style=none, label={below:$C_x$}] (12) at (-2, -2) {};
		\node [style=none, label={above:$C_{x}$}] (13) at (-2, 2) {};
		\node [style=white dot] (14) at (-0.5, 2) {};
		\node [style=none] (15) at (-2, -0.5) {};
		\node [style=none] (16) at (-2, 0.5) {};
		\node [style=none, label={below:$C_x$}] (17) at (7.5, -1) {};
		\node [style=none, label={above:$C_{x}$}] (18) at (7.5, 1) {};
		\node [style=none] (19) at (6, 0) {$N_{x}^k=$};
		\node [style=white dot big] (20) at (7.5, 0) {\scriptsize{$k$}};
		\node [style=none] (21) at (2.5, 0) {\text{, graphically}};
	\end{pgfonlayer}
	\begin{pgfonlayer}{edgelayer}
		\draw (1) to (6);
		\draw (14) to (7);
		\draw (13.center) to (16.center);
		\draw (15.center) to (12.center);
		\draw [in=90, out=-90] (16.center) to (4.center);
		\draw [in=90, out=-90] (8.center) to (15.center);
		\draw (9.center) to (5.center);
		\draw (18.center) to (20);
		\draw (20) to (17.center);
		\draw (10.center) to (9.center);
		\draw (11.center) to (8.center);
		\draw (5.center) to (2.center);
		\draw (4.center) to (3.center);
	\end{pgfonlayer}
\end{tikzpicture}\ .
	\label{eq:C-morphisms-Nx}
\end{align}

These morphisms are required to satisfy the following relations
for $x,y,z,w\in\Zb/r$
with $x+y-2=z+w$:
\begin{align}
	\begin{tikzpicture}
	\begin{pgfonlayer}{nodelayer}
		\node [style=square] (0) at (-8, 1) {\scriptsize{$x,y+z-1$}};
		\node [style=none] (1) at (-7.5, 1) {};
		\node [style=none] (2) at (-7.5, 1) {};
		\node [style=none, label={above:$C_{x+y+z-2}$}] (4) at (-8, 2) {};
		\node [style=none] (6) at (-7.5, -0.5) {};
		\node [style=none] (17) at (-10, 0) {$=$};
		\node [style=square] (27) at (-7.5, -0.5) {\scriptsize{$y,z$}};
		\node [style=none] (28) at (-8.5, 0) {};
		\node [style=none] (29) at (-8, -0.5) {};
		\node [style=none] (31) at (-8.5, 1) {};
		\node [style=none, label={below:$C_x$}] (32) at (-9, -1.5) {};
		\node [style=none, label={below:$C_y$}] (33) at (-8, -1.5) {};
		\node [style=none, label={below:$C_z$}] (34) at (-7, -1.5) {};
		\node [style=square] (35) at (-12, 1) {\scriptsize{$x+y-1,z$}};
		\node [style=none] (36) at (-12.5, 1) {};
		\node [style=none] (37) at (-11.5, 1) {};
		\node [style=none, label={above:$C_{x+y+z-2}$}] (39) at (-12, 2) {};
		\node [style=none] (40) at (-11.5, 0) {};
		\node [style=square] (41) at (-12.5, -0.5) {\scriptsize{$x,y$}};
		\node [style=none] (42) at (-13, -0.5) {};
		\node [style=none] (43) at (-12, -0.5) {};
		\node [style=none] (44) at (-12.5, -0.5) {};
		\node [style=none] (45) at (-12.5, 0) {};
		\node [style=none, label={below:$C_x$}] (46) at (-13, -1.5) {};
		\node [style=none, label={below:$C_y$}] (47) at (-12, -1.5) {};
		\node [style=none, label={below:$C_z$}] (48) at (-11, -1.5) {};
		\node [style=none] (49) at (-7, -0.5) {};
		\node [style=square] (50) at (2, 1) {\scriptsize{$x,1$}};
		\node [style=none] (51) at (1.5, 1) {};
		\node [style=none] (52) at (2.5, 1) {};
		\node [style=none, label={above:$C_{x}$}] (54) at (2, 2) {};
		\node [style=none, label={below:$C_x$}] (55) at (1.5, -1.5) {};
		\node [style=none] (56) at (2.5, 0) {};
		\node [style=none] (57) at (-3, 0) {(associativity)};
		\node [style=white dot] (58) at (2.5, -1.5) {};
		\node [style=square] (59) at (5, 1) {\scriptsize{$1,x$}};
		\node [style=none] (60) at (5.5, 1) {};
		\node [style=none] (61) at (4.5, 1) {};
		\node [style=none, label={above:$C_{x}$}] (63) at (5, 2) {};
		\node [style=none, label={below:$C_x$}] (64) at (5.5, -1.5) {};
		\node [style=none] (65) at (4.5, 0) {};
		\node [style=white dot] (66) at (4.5, -1.5) {};
		\node [style=none] (67) at (3.5, 0) {$=$};
		\node [style=none] (68) at (6.5, 0) {$=$};
		\node [style=none, label={above:$C_{x}$}] (69) at (7.5, 2) {};
		\node [style=none, label={below:$C_x$}] (70) at (7.5, -1.5) {};
		\node [style=none] (71) at (10, 0) {(unitality)};
	\end{pgfonlayer}
	\begin{pgfonlayer}{edgelayer}
		\draw (2.center) to (6.center);
		\draw [in=90, out=-90, looseness=1.25] (28.center) to (32.center);
		\draw (29.center) to (33.center);
		\draw (37.center) to (40.center);
		\draw (42.center) to (46.center);
		\draw (43.center) to (47.center);
		\draw (45.center) to (44.center);
		\draw [in=90, out=-90, looseness=1.25] (40.center) to (48.center);
		\draw (36.center) to (45.center);
		\draw (31.center) to (28.center);
		\draw (39.center) to (35);
		\draw (4.center) to (0);
		\draw (49.center) to (34.center);
		\draw (51.center) to (55.center);
		\draw (52.center) to (56.center);
		\draw (58) to (56.center);
		\draw (60.center) to (64.center);
		\draw (61.center) to (65.center);
		\draw (66) to (65.center);
		\draw (69.center) to (70.center);
		\draw (54.center) to (50);
		\draw (63.center) to (59);
	\end{pgfonlayer}
\end{tikzpicture}\ ,
	\label{eq:C-morphisms-relations-alg}\\
	\begin{tikzpicture}
	\begin{pgfonlayer}{nodelayer}
		\node [style=none, label={above:$C_{x}$}] (8) at (-13, 1.5) {};
		\node [style=none, label={above:$C_{y}$}] (9) at (-12, 1.5) {};
		\node [style=square] (11) at (-12.5, 0.5) {\scriptsize{$x,y$}};
		\node [style=none] (12) at (-13, 0.5) {};
		\node [style=none] (13) at (-12, 0.5) {};
		\node [style=none, label={above:$C_{y}$}] (72) at (-8, 1.5) {};
		\node [style=none, label={above:$C_{z}$}] (73) at (-7, 1.5) {};
		\node [style=square] (75) at (-7.5, 0.5) {\scriptsize{$y,z$}};
		\node [style=none] (76) at (-8, 0.5) {};
		\node [style=none] (77) at (-7, 0.5) {};
		\node [style=none, label={below:$C_{x+y+z+2}$}] (81) at (-12, -2) {};
		\node [style=square] (82) at (-12, -1) {\scriptsize{$x+y+1,z$}};
		\node [style=none] (83) at (-12.5, -1) {};
		\node [style=none] (84) at (-11.5, -1) {};
		\node [style=none, label={below:$C_{x+y+z+2}$}] (88) at (-8, -2) {};
		\node [style=square] (89) at (-8, -1) {\scriptsize{$x,y+z+1$}};
		\node [style=none] (90) at (-8.5, -1) {};
		\node [style=none] (91) at (-7.5, -1) {};
		\node [style=none, label={above:$C_{x}$}] (93) at (-9, 1.5) {};
		\node [style=none, label={above:$C_{z}$}] (94) at (-11, 1.5) {};
		\node [style=none] (95) at (-12.5, 0) {};
		\node [style=none] (96) at (-11.5, 0) {};
		\node [style=none] (97) at (-8.5, 0) {};
		\node [style=none] (98) at (-7.5, 0) {};
		\node [style=none] (99) at (-10, 0) {$=$};
		\node [style=none] (100) at (-3, 0) {(coassociativity)};
		\node [style=square] (103) at (2, -1) {\scriptsize{$x,-1$}};
		\node [style=none] (104) at (1.5, -1) {};
		\node [style=none] (105) at (2.5, -1) {};
		\node [style=square] (109) at (5, -1) {\scriptsize{$-1,x$}};
		\node [style=none] (110) at (4.5, -1) {};
		\node [style=none] (111) at (5.5, -1) {};
		\node [style=white dot] (113) at (2.5, 1.5) {};
		\node [style=white dot] (114) at (4.5, 1.5) {};
		\node [style=none, label={above:$C_{x}$}] (115) at (1.5, 1.5) {};
		\node [style=none, label={above:$C_{x}$}] (116) at (5.5, 1.5) {};
		\node [style=none] (117) at (6.5, 0) {$=$};
		\node [style=none, label={above:$C_{x}$}] (118) at (7.5, 1.5) {};
		\node [style=none, label={below:$C_x$}] (119) at (7.5, -2) {};
		\node [style=none] (120) at (10, 0) {(counitality)};
		\node [style=none, label={below:$C_x$}] (121) at (2, -2) {};
		\node [style=none, label={below:$C_x$}] (122) at (5, -2) {};
		\node [style=none] (123) at (3.5, 0) {$=$};
	\end{pgfonlayer}
	\begin{pgfonlayer}{edgelayer}
		\draw (8.center) to (12.center);
		\draw (9.center) to (13.center);
		\draw (72.center) to (76.center);
		\draw (73.center) to (77.center);
		\draw [in=90, out=-90, looseness=1.25] (93.center) to (97.center);
		\draw [in=90, out=-90, looseness=1.25] (94.center) to (96.center);
		\draw (95.center) to (11);
		\draw (95.center) to (83.center);
		\draw (96.center) to (84.center);
		\draw (82) to (81.center);
		\draw (75) to (98.center);
		\draw (98.center) to (91.center);
		\draw (97.center) to (90.center);
		\draw (89) to (88.center);
		\draw (118.center) to (119.center);
		\draw (115.center) to (104.center);
		\draw (113) to (105.center);
		\draw (114) to (110.center);
		\draw (116.center) to (111.center);
		\draw (109) to (122.center);
		\draw (103) to (121.center);
	\end{pgfonlayer}
\end{tikzpicture}\ ,
	\label{eq:C-morphisms-relations-coalg}
\end{align}
\begin{align}
	\begin{tikzpicture}
	\begin{pgfonlayer}{nodelayer}
		\node [style=none] (71) at (9.5, 0) {(Frobenius)};
		\node [style=none, label={above:$C_{z}$}] (124) at (-7, 2) {};
		\node [style=none, label={above:$C_{w}$}] (125) at (-6, 2) {};
		\node [style=square] (126) at (-6.5, 1) {\scriptsize{$z,w$}};
		\node [style=none] (127) at (-7, 1) {};
		\node [style=none] (128) at (-6, 1) {};
		\node [style=none] (129) at (-6.5, 0) {};
		\node [style=square] (130) at (-6.5, -1) {\scriptsize{$x,y$}};
		\node [style=none] (131) at (-7, -1) {};
		\node [style=none] (132) at (-6, -1) {};
		\node [style=none] (133) at (-6.5, -1) {};
		\node [style=none] (134) at (-6.5, 0) {};
		\node [style=none, label={below:$C_x$}] (135) at (-7, -2) {};
		\node [style=none, label={below:$C_y$}] (136) at (-6, -2) {};
		\node [style=none] (137) at (-5, 0) {$=$};
		\node [style=square] (138) at (-2.5, 1) {\scriptsize{$x,z-w+1$}};
		\node [style=none] (139) at (-3, 1) {};
		\node [style=none] (140) at (-2, 1) {};
		\node [style=none, label={above:$C_{z}$}] (141) at (-2.5, 2) {};
		\node [style=none] (143) at (-3, 0) {};
		\node [style=none, label={below:$C_{y}$}] (144) at (-0.5, -2) {};
		\node [style=square] (145) at (-0.5, -1) {\scriptsize{$y-w-1,w$}};
		\node [style=none] (146) at (-1, -1) {};
		\node [style=none] (147) at (0, -1) {};
		\node [style=none] (149) at (0, 0) {};
		\node [style=square] (150) at (5.5, 1) {\scriptsize{$w-y+1,y$}};
		\node [style=none] (151) at (5, 1) {};
		\node [style=none] (152) at (6, 1) {};
		\node [style=none, label={above:$C_{w}$}] (153) at (5.5, 2) {};
		\node [style=none] (154) at (6, 0) {};
		\node [style=none, label={below:$C_{x}$}] (156) at (3.5, -2) {};
		\node [style=square] (157) at (3.5, -1) {\scriptsize{$z,x-z-1$}};
		\node [style=none] (158) at (3, -1) {};
		\node [style=none] (159) at (4, -1) {};
		\node [style=none] (160) at (3, 0) {};
		\node [style=none, label={below:$C_x$}] (162) at (-3, -2) {};
		\node [style=none, label={above:$C_{w}$}] (163) at (0, 2) {};
		\node [style=none, label={below:$C_y$}] (164) at (6, -2) {};
		\node [style=none, label={above:$C_{z}$}] (165) at (3, 2) {};
		\node [style=none] (166) at (1.5, 0) {$=$};
	\end{pgfonlayer}
	\begin{pgfonlayer}{edgelayer}
		\draw (124.center) to (127.center);
		\draw (125.center) to (128.center);
		\draw (129.center) to (126);
		\draw (131.center) to (135.center);
		\draw (132.center) to (136.center);
		\draw (134.center) to (133.center);
		\draw (139.center) to (143.center);
		\draw (141.center) to (138);
		\draw (149.center) to (147.center);
		\draw (145) to (144.center);
		\draw (152.center) to (154.center);
		\draw (153.center) to (150);
		\draw (160.center) to (158.center);
		\draw (157) to (156.center);
		\draw (143.center) to (162.center);
		\draw (163.center) to (149.center);
		\draw (165.center) to (160.center);
		\draw (154.center) to (164.center);
		\draw [in=90, out=-90, looseness=1.25] (140.center) to (146.center);
		\draw [in=-90, out=90, looseness=1.25] (159.center) to (151.center);
	\end{pgfonlayer}
\end{tikzpicture}\ ,
	\label{eq:C-morphisms-relations-Frobenius}
\end{align}
\begin{align}
  \begin{tikzpicture}
	\begin{pgfonlayer}{nodelayer}
		\node [style=none, label={below:$C_x$}] (0) at (-1, -1) {};
		\node [style=none, label={above:$C_{x}$}] (1) at (-1, 1) {};
		\node [style=white dot big] (2) at (-1, 0) {\scriptsize{$r$}};
		\node [style=none] (4) at (6, 0) {(deck transformation)};
		\node [style=none, label={below:$C_x$}] (5) at (1, -1) {};
		\node [style=none, label={above:$C_{x}$}] (6) at (1, 1) {};
		\node [style=none] (8) at (0, 0) {$=$};
	\end{pgfonlayer}
	\begin{pgfonlayer}{edgelayer}
		\draw (1.center) to (2);
		\draw (2) to (0.center);
		\draw (6.center) to (5.center);
	\end{pgfonlayer}
\end{tikzpicture}\ ,
  \label{eq:C-morphism-relations-deck}
\end{align}

\begin{align}
	\begin{tikzpicture}
	\begin{pgfonlayer}{nodelayer}
		\node [style=none] (17) at (-3.5, 0) {$=$};
		\node [style=none, label={above:$C_{x+y-1}$}] (39) at (-5.5, 1.5) {};
		\node [style=square] (41) at (-5.5, 0.5) {\scriptsize{$x,y$}};
		\node [style=none] (42) at (-6, 0) {};
		\node [style=none] (43) at (-5, 0) {};
		\node [style=none] (44) at (-5.5, 0.5) {};
		\node [style=none, label={below:$C_x$}] (46) at (-5, -2) {};
		\node [style=none, label={below:$C_y$}] (47) at (-6, -2) {};
		\node [style=none] (57) at (7, 0) {(commutativity)};
		\node [style=none, label={above:$C_{x+y-1}$}] (58) at (-1.5, 1.5) {};
		\node [style=square] (59) at (-1.5, 0.5) {\scriptsize{$y,x$}};
		\node [style=none] (60) at (-2, 0.5) {};
		\node [style=none] (61) at (-1, 0.5) {};
		\node [style=none, label={below:$C_x$}] (63) at (-1, -2) {};
		\node [style=none, label={below:$C_y$}] (64) at (-2, -2) {};
		\node [style=white dot big] (65) at (-2, -1) {\scriptsize{$x-1$}};
		\node [style=none] (66) at (-6, 0.5) {};
		\node [style=none] (67) at (-5, 0.5) {};
		\node [style=none] (68) at (0.5, 0) {$=$};
		\node [style=none, label={above:$C_{x+y-1}$}] (69) at (2.5, 1.5) {};
		\node [style=square] (70) at (2.5, 0.5) {\scriptsize{$y,x$}};
		\node [style=none] (71) at (2, 0.5) {};
		\node [style=none] (72) at (3, 0.5) {};
		\node [style=none, label={below:$C_x$}] (73) at (3, -2) {};
		\node [style=none, label={below:$C_y$}] (74) at (2, -2) {};
		\node [style=white dot big] (75) at (3, -1) {\scriptsize{$1-y$}};
	\end{pgfonlayer}
	\begin{pgfonlayer}{edgelayer}
		\draw [in=90, out=-90, looseness=1.25] (42.center) to (46.center);
		\draw [in=90, out=-90, looseness=1.25] (43.center) to (47.center);
		\draw (39.center) to (44.center);
		\draw (66.center) to (42.center);
		\draw (67.center) to (43.center);
		\draw (60.center) to (65);
		\draw (65) to (64.center);
		\draw (61.center) to (63.center);
		\draw (58.center) to (59);
		\draw (69.center) to (70);
		\draw (72.center) to (75);
		\draw (75) to (73.center);
		\draw (71.center) to (74.center);
	\end{pgfonlayer}
\end{tikzpicture}\ ,
	\label{eq:C-morphisms-relations-commutativity}
\end{align}
\begin{align}
	\begin{tikzpicture}
	\begin{pgfonlayer}{nodelayer}
		\node [style=none] (57) at (4, 0) {(twist)};
		\node [style=white dot big] (91) at (-1.5, 0) {\scriptsize{$x$}};
		\node [style=none, label={below:$C_x$}] (92) at (-1.5, -2) {};
		\node [style=none, label={above:$C_{x}$}] (93) at (-1.5, 2) {};
		\node [style=none] (94) at (0, 0) {$=$};
		\node [style=none, label={below:$C_x$}] (96) at (1.5, -2) {};
		\node [style=none, label={above:$C_{x}$}] (97) at (1.5, 2) {};
	\end{pgfonlayer}
	\begin{pgfonlayer}{edgelayer}
		\draw (93.center) to (91);
		\draw (91) to (92.center);
		\draw (97.center) to (96.center);
	\end{pgfonlayer}
\end{tikzpicture}\ ,
	\label{eq:C-morphisms-relations-twist}
\end{align}
and finally
\begin{align}
	\begin{tikzpicture}
	\begin{pgfonlayer}{nodelayer}
		\node [style=none] (17) at (-0.5, 0) {$=$};
		\node [style=square] (41) at (-2.5, -1) {\scriptsize{$x,-x$}};
		\node [style=none] (57) at (8, 0) {(twist$'$)};
		\node [style=none] (76) at (-3, -1) {};
		\node [style=none] (77) at (-2, -1) {};
		\node [style=none] (78) at (-3, -0.5) {};
		\node [style=none] (79) at (-2, -0.5) {};
		\node [style=white dot] (80) at (-2.5, -2) {};
		\node [style=square] (81) at (-2.5, 1) {\scriptsize{$x,-x$}};
		\node [style=none] (82) at (-3, 0.5) {};
		\node [style=none] (83) at (-2, 0.5) {};
		\node [style=none] (84) at (-3, 1) {};
		\node [style=none] (85) at (-2, 1) {};
		\node [style=none, label={above:$C_{-1}$}] (87) at (-2.5, 2) {};
		\node [style=white dot big] (98) at (-3, 0) {\scriptsize{$k$}};
		\node [style=square] (99) at (3, -1) {\scriptsize{$x-k-1,-x+k+1$}};
		\node [style=none] (100) at (2.5, -1) {};
		\node [style=none] (101) at (3.5, -1) {};
		\node [style=none] (102) at (2.5, -0.5) {};
		\node [style=none] (103) at (3.5, -0.5) {};
		\node [style=white dot] (104) at (3, -2) {};
		\node [style=square] (105) at (3, 1) {\scriptsize{$x-k-1,-x+k+1$}};
		\node [style=none] (106) at (2.5, 0.5) {};
		\node [style=none] (107) at (3.5, 0.5) {};
		\node [style=none] (108) at (2.5, 1) {};
		\node [style=none] (109) at (3.5, 1) {};
		\node [style=none, label={above:$C_{-1}$}] (110) at (3, 2) {};
		\node [style=white dot big] (111) at (2.5, 0) {\scriptsize{$k$}};
	\end{pgfonlayer}
	\begin{pgfonlayer}{edgelayer}
		\draw (78.center) to (76.center);
		\draw (79.center) to (77.center);
		\draw (41) to (80);
		\draw (84.center) to (82.center);
		\draw (85.center) to (83.center);
		\draw (83.center) to (79.center);
		\draw (87.center) to (81);
		\draw (82.center) to (98);
		\draw (98) to (78.center);
		\draw (102.center) to (100.center);
		\draw (103.center) to (101.center);
		\draw (99) to (104);
		\draw (108.center) to (106.center);
		\draw (109.center) to (107.center);
		\draw (107.center) to (103.center);
		\draw (110.center) to (105);
		\draw (106.center) to (111);
		\draw (111) to (102.center);
	\end{pgfonlayer}
\end{tikzpicture}\ .
	\label{eq:C-morphisms-relations-SL2Z}
\end{align}
\begin{defn} \label{def:closed-Lambda-r-FA}
	We call the collection of objects $C=\left\{ C_x \right\}_{x\in\Zb/r}$
	in $\Sc$ with morphisms in \eqref{eq:C-morphisms-alg} and
	\eqref{eq:C-morphisms-coalg} satisfying 
	\eqref{eq:C-morphisms-relations-alg}--\eqref{eq:C-morphisms-relations-SL2Z}
	a \textsl{closed $\Lambda_r$-Frobenius algebra in $\Sc$}.
\end{defn}
\begin{rmk}
	We note that the Frobenius relation together with unitality and counitality imply associativity and coassociativity.
	\label{rem:FR-unit-counit}
\end{rmk}
The following proposition can be shown in a similar way as one shows that the
Nakayama automorphism of a Frobenius algebra is an algebra and coalgebra automorphism \cite{Fuchs:2008fa}.
\begin{prop}
For a closed $\Lambda_r$-Frobenius algebra the automorphisms $N_x:C_x\to C_x$ endow the objects with a $\Zb/r$-action
and its structure morphisms intertwine this action.
\end{prop}

\begin{prop}\label{prop:duals-cat-dimension}
	Let $C$ be a closed $\Lambda_r$-Frobenius algebra.
	\begin{enumerate}
		\item 
	The objects $C_x$ and $C_{-x}$ are dual to each other.
		\item If $\mathrm{gcd}(x,r)=\mathrm{gcd}(y,r)$ then 
	the categorical dimensions of the objects $C_x$ and $C_y$ are the same. 
	\end{enumerate}
\end{prop}
\begin{proof}
	In fact the morphisms 
	\begin{align}
		\begin{tikzpicture}
	\begin{pgfonlayer}{nodelayer}
		\node [style=square] (0) at (-2, 0.5) {\scriptsize{$x,-x$}};
		\node [style=none] (1) at (-2.75, 0.5) {};
		\node [style=none] (2) at (-1.25, 0.5) {};
		\node [style=none, label={below:$C_x$}] (5) at (-2.75, -0.5) {};
		\node [style=none, label={below:$C_{-x}$}] (6) at (-1.25, -0.5) {};
		\node [style=none] (7) at (-4.75, 0) {$\mathrm{ev}_{C_x}=$};
		\node [style=none, label={above:$C_{-x}$}] (8) at (4.5, 0.5) {};
		\node [style=none, label={above:$C_{x}$}] (9) at (6, 0.5) {};
		\node [style=square] (11) at (5.25, -0.5) {\scriptsize{$-x,x$}};
		\node [style=none] (12) at (4.5, -0.5) {};
		\node [style=none] (13) at (6, -0.5) {};
		\node [style=none] (15) at (2, 0) {$\mathrm{coev}_{C_x}=$};
		\node [style=white dot] (16) at (-2, 1.5) {};
		\node [style=white dot] (17) at (5.25, -1.5) {};
	\end{pgfonlayer}
	\begin{pgfonlayer}{edgelayer}
		\draw (1.center) to (5.center);
		\draw (2.center) to (6.center);
		\draw (8.center) to (12.center);
		\draw (9.center) to (13.center);
		\draw (16) to (0);
		\draw (11) to (17);
	\end{pgfonlayer}
\end{tikzpicture}
		\label{eq:C-duality-morphisms}
	\end{align}
	are duality morphisms for $C_x$ and $C_x^{\vee}=C_{-x}$.
	This can be seen using the Frobenius relation \eqref{eq:C-morphisms-relations-Frobenius}
	and unitality \eqref{eq:C-morphisms-relations-alg} and counitality 
	\eqref{eq:C-morphisms-relations-coalg}.

	The categorical dimension  of an object $C$ in a symmetric monoidal category
	can be computed as $\mathrm{dim}(C)=\mathrm{ev_C}\circ\sigma_{C,C}\circ\mathrm{coev_C}$, 
	so for $C_x$ we have
	\begin{align}
		\begin{tikzpicture}
	\begin{pgfonlayer}{nodelayer}
		\node [style=square] (0) at (-1.75, 1.5) {\scriptsize{$x,-x$}};
		\node [style=none] (1) at (-2.5, 1.5) {};
		\node [style=none] (2) at (-1, 1.5) {};
		\node [style=none] (7) at (-5, 0) {$\mathrm{dim}(C_x)=$};
		\node [style=square] (11) at (-1.75, -1.5) {\scriptsize{$-x,x$}};
		\node [style=none] (12) at (-2.5, -1.5) {};
		\node [style=none] (13) at (-1, -1.5) {};
		\node [style=white dot] (16) at (-1.75, 2.5) {};
		\node [style=white dot] (17) at (-1.75, -2.5) {};
		\node [style=none] (18) at (-2.5, 0.5) {};
		\node [style=none] (19) at (-1, 0.5) {};
		\node [style=none] (20) at (-2.5, -0.5) {};
		\node [style=none] (21) at (-1, -0.5) {};
		\node [style=none] (22) at (0, 0) {$=$};
		\node [style=square] (23) at (2.25, 1.5) {\scriptsize{$-x,x$}};
		\node [style=none] (24) at (1.5, 1.5) {};
		\node [style=none] (25) at (3, 1.5) {};
		\node [style=square] (26) at (2.25, -1.5) {\scriptsize{$-x,x$}};
		\node [style=none] (27) at (1.5, -1.5) {};
		\node [style=none] (28) at (3, -1.5) {};
		\node [style=white dot] (29) at (2.25, 2.5) {};
		\node [style=white dot] (30) at (2.25, -2.5) {};
		\node [style=none] (31) at (1.5, 0.5) {};
		\node [style=none] (32) at (3, 0.5) {};
		\node [style=none] (33) at (1.5, -0.5) {};
		\node [style=none] (34) at (3, -0.5) {};
		\node [style=white dot] (36) at (1.5, 0) {\scriptsize{$x-1$}};
	\end{pgfonlayer}
	\begin{pgfonlayer}{edgelayer}
		\draw (16) to (0);
		\draw (11) to (17);
		\draw (1.center) to (18.center);
		\draw (2.center) to (19.center);
		\draw (20.center) to (12.center);
		\draw (21.center) to (13.center);
		\draw [in=90, out=-90, looseness=0.75] (18.center) to (21.center);
		\draw [in=-90, out=90, looseness=0.75] (20.center) to (19.center);
		\draw (29) to (23);
		\draw (26) to (30);
		\draw (24.center) to (31.center);
		\draw (25.center) to (32.center);
		\draw (33.center) to (27.center);
		\draw (34.center) to (28.center);
		\draw (31.center) to (36);
		\draw (36) to (33.center);
		\draw (32.center) to (34.center);
	\end{pgfonlayer}
\end{tikzpicture}
		\label{eq:cat-dim}
	\end{align}
	where we have used commutativity \eqref{eq:C-morphisms-relations-commutativity}.

	Let us write 
	\begin{align}
		\begin{tikzpicture}
	\begin{pgfonlayer}{nodelayer}
		\node [style=none] (7) at (-1, 0) {$T(x,z)=$};
		\node [style=square] (23) at (2.25, 1.5) {\scriptsize{$-x,x$}};
		\node [style=none] (24) at (1.5, 1.5) {};
		\node [style=none] (25) at (3, 1.5) {};
		\node [style=square] (26) at (2.25, -1.5) {\scriptsize{$-x,x$}};
		\node [style=none] (27) at (1.5, -1.5) {};
		\node [style=none] (28) at (3, -1.5) {};
		\node [style=white dot] (29) at (2.25, 2.5) {};
		\node [style=white dot] (30) at (2.25, -2.5) {};
		\node [style=none] (31) at (1.5, 0.5) {};
		\node [style=none] (32) at (3, 0.5) {};
		\node [style=none] (33) at (1.5, -0.5) {};
		\node [style=none] (34) at (3, -0.5) {};
		\node [style=white dot] (36) at (1.5, 0) {\scriptsize{$z-1$}};
	\end{pgfonlayer}
	\begin{pgfonlayer}{edgelayer}
		\draw (29) to (23);
		\draw (26) to (30);
		\draw (24.center) to (31.center);
		\draw (25.center) to (32.center);
		\draw (33.center) to (27.center);
		\draw (34.center) to (28.center);
		\draw (31.center) to (36);
		\draw (36) to (33.center);
		\draw (32.center) to (34.center);
	\end{pgfonlayer}
\end{tikzpicture}
		\label{eq:torus-xy}
	\end{align}
	in particular we have $T(x,x)=\mathrm{dim}(C_x)$.
	Using \eqref{eq:C-morphisms-relations-twist} and \eqref{eq:C-morphisms-relations-SL2Z} we obtain $T(x,z)=T(x,x+z)=T(x+z,z)$,
	and using these relations we get via the Euclidean algorithm that
		$T(x,z)=T(\mathrm{gcd}(x,z,r),0)$ and hence
	\begin{align}
		\mathrm{dim}(C_x)=
		T(x,x)=T(\mathrm{gcd}(x,r),0)\stackrel{\text{(assumption)}}{=}
		T(\mathrm{gcd}(y,r),0)= T(y,y)=\mathrm{dim}(C_y)\ .
		\label{eq:equal-dim-Cx}
	\end{align}

\end{proof}
\begin{rem} \label{rem:cat-dim-SVect}
	In \cite[Sec.\,2.6]{Moore:2006db} the special case
	of $r=2$ and $\Cc=\SVect$ is discussed. There it is assumed that
	the $\Zb/2$-grading by the eigenvalues of $N_x$ coincides with
	the grading of vector spaces in $\SVect$.
	This implies that $C_1$ is purely even, however $C_0$ may have even and odd components.
	Furthermore $\mu_{0,0}$ 
	is commutative and the dimensions of the vector spaces $C_0$ and $C_1$ 
	(but not their categorical dimension!) agree.
\end{rem}

\begin{defn}
	\label{def:knowledgable-lambda-r-frobenius-algebra}
	 A \textsl{$\Lambda_r$-Frobenius algebra} $A\in\Sc$ is a Frobenius algebra in $\Sc$ such that its Nakayama automorphism $N_A$ satisfies $N_A^r=\id_A$ \cite{Dyckerhoff:2015csg}. We denote the structure maps of $A$ by 
	 \begin{equation}
	 \begin{tikzpicture}
	\begin{pgfonlayer}{nodelayer}
		\node [style=square] (0) at (-6.5, 0) {};
		\node [style=none] (1) at (-6.5, 0) {};
		\node [style=none] (2) at (-6.5, 0) {};
		\node [style=none, label={above:$A$}] (4) at (-6.5, 1) {};
		\node [style=none, label={below:$A$}] (5) at (-7, -1) {};
		\node [style=none, label={below:$A$}] (6) at (-6, -1) {};
		\node [style=none] (7) at (-8, 0) {$\mu=$};
		\node [style=none, label={above:$A$}] (8) at (-1, 1) {};
		\node [style=none, label={above:$A$}] (9) at (0, 1) {};
		\node [style=none, label={below:$A$}] (10) at (-0.5, -1) {};
		\node [style=square] (11) at (-0.5, 0) {};
		\node [style=none] (12) at (-0.5, 0) {};
		\node [style=none] (13) at (-0.5, 0) {};
		\node [style=none] (15) at (-2, 0) {$\Delta=$};
		\node [style=none, label={above:$A$}] (16) at (-3.5, 1) {};
		\node [style=none] (17) at (-4.5, 0) {$\eta=$};
		\node [style=white dot] (18) at (-3.5, -1) {};
		\node [style=none, label={below:$A$}] (19) at (2.5, -1) {};
		\node [style=none] (20) at (1.5, 0) {$\varepsilon=$};
		\node [style=white dot] (21) at (2.5, 1) {};
		\node [style=none, label={below:$A$}] (22) at (9.5, -1) {};
		\node [style=none, label={above:$A$}] (23) at (9.5, 1) {};
		\node [style=none] (24) at (8, 0) {$N^k=$};
		\node [style=white dot big] (26) at (9.5, 0) {\scriptsize{$k$}};
		\node [style=none, label={below:$A$}] (27) at (5, -1) {};
		\node [style=none, label={above:$A$}] (28) at (5, 1) {};
		\node [style=none] (29) at (4, 0) {$N=$};
		\node [style=square] (30) at (6, -0.5) {};
		\node [style=square] (31) at (6, 0.5) {};
		\node [style=none] (32) at (5, 0.5) {};
		\node [style=none] (33) at (5, -0.5) {};
		\node [style=white dot] (34) at (6, 1.25) {};
		\node [style=white dot] (35) at (6, -1.25) {};
	\end{pgfonlayer}
	\begin{pgfonlayer}{edgelayer}
		\draw [in=90, out=-120] (1.center) to (5.center);
		\draw [in=90, out=-60] (2.center) to (6.center);
		\draw (16.center) to (18);
		\draw [in=120, out=-90] (8.center) to (12.center);
		\draw [in=60, out=-90] (9.center) to (13.center);
		\draw (21) to (19.center);
		\draw (23.center) to (26);
		\draw (26) to (22.center);
		\draw (4.center) to (0);
		\draw (11) to (10.center);
		\draw [bend left=45, looseness=1.25] (31) to (30);
		\draw (28.center) to (32.center);
		\draw (33.center) to (27.center);
		\draw [in=-135, out=90] (33.center) to (31);
		\draw [in=135, out=-90, looseness=0.75] (32.center) to (30);
		\draw (34) to (31);
		\draw (30) to (35);
	\end{pgfonlayer}
\end{tikzpicture}
	 \label{eq:A-morphisms}
	 \end{equation}
	 Where $N$, the \textsl{Nakayama automorphism} of $A$, is defined in terms of the previous structure morphisms.
	
	A \textsl{ knowledgeable $\Lambda_r$-Frobenius algebra $(A,C)$ in $\Sc$} consists of the following data.
	\begin{itemize}
		\item A $\Lambda_r$-Frobenius algebra $A\in\Sc$ with structure maps $\mu,\eta,\Delta,\varepsilon$ (no subscript) as above,
		\item a closed $\Lambda_r$-Frobenius algebra $C$ in $\Sc$,
		\item two morphisms $\iota_x\in\Sc(C_x,A)$ and $\pi_x=\iota_x^*\in\Sc(A,C_x)$ for every $x\in\Zb/r$,
			\begin{align}
				\begin{tikzpicture}
	\begin{pgfonlayer}{nodelayer}
		\node [style=triangle] (0) at (-1, 0) {\scriptsize{$x$}};
		\node [style=up-triangle] (1) at (2.5, 0) {\scriptsize{$x$}};
		\node [style=none, label={above:$A$}] (2) at (-1, 1) {};
		\node [style=none, label={above:$C_x$}] (3) at (2.5, 1) {};
		\node [style=none, label={below:$A$}] (4) at (2.5, -1) {};
		\node [style=none, label={below:$C_x$}] (5) at (-1, -1) {};
		\node [style=none] (6) at (-2.5, 0) {$\iota_x=$};
		\node [style=none] (7) at (1, 0) {$\pi_x=$};
	\end{pgfonlayer}
	\begin{pgfonlayer}{edgelayer}
		\draw (2.center) to (0);
		\draw (0) to (5.center);
		\draw (3.center) to (1);
		\draw (1) to (4.center);
	\end{pgfonlayer}
\end{tikzpicture}
				\label{eq:AC-morphisms}
			\end{align}
	\end{itemize}
	These morphisms are required satisfy the following conditions. For every $x,y,z\in\Zb/r$
	\begin{align}
		&\begin{tikzpicture}
	\begin{pgfonlayer}{nodelayer}
		\node [style=square] (8) at (0, 1.5) {};
		\node [style=none] (9) at (0, 1.5) {};
		\node [style=none] (10) at (0, 1.5) {};
		\node [style=none, label={above:$A$}] (11) at (0, 2.5) {};
		\node [style=none] (12) at (-0.5, 0.5) {};
		\node [style=none] (13) at (0.5, 0.5) {};
		\node [style=triangle] (14) at (-0.5, -1) {\scriptsize{$x$}};
		\node [style=none, label={below:$C_x$}] (16) at (-0.5, -2) {};
		\node [style=none] (17) at (-0.5, -0.5) {};
		\node [style=none] (18) at (0.5, -0.5) {};
		\node [style=none, label={below:$A$}] (19) at (0.5, -2) {};
		\node [style=square] (20) at (4, 1.5) {};
		\node [style=none] (21) at (4, 1.5) {};
		\node [style=none] (22) at (4, 1.5) {};
		\node [style=none, label={above:$A$}] (23) at (4, 2.5) {};
		\node [style=none] (24) at (3.5, 0.5) {};
		\node [style=none] (25) at (4.5, 0.5) {};
		\node [style=triangle] (26) at (3.5, -1) {\scriptsize{$x$}};
		\node [style=none, label={below:$C_x$}] (27) at (3.5, -2) {};
		\node [style=none] (28) at (3.5, -0.5) {};
		\node [style=none, label={below:$A$}] (30) at (4.5, -2) {};
		\node [style=none] (31) at (2, 0) {$=$};
		\node [style=white dot big] (32) at (4.5, 0) {\scriptsize{$1-x$}};
	\end{pgfonlayer}
	\begin{pgfonlayer}{edgelayer}
		\draw [in=90, out=-120] (9.center) to (12.center);
		\draw [in=90, out=-60] (10.center) to (13.center);
		\draw (11.center) to (8);
		\draw (14) to (16.center);
		\draw [in=90, out=-90, looseness=0.75] (12.center) to (18.center);
		\draw [in=90, out=-90, looseness=0.75] (13.center) to (17.center);
		\draw (17.center) to (14);
		\draw (18.center) to (19.center);
		\draw [in=90, out=-120] (21.center) to (24.center);
		\draw [in=90, out=-60] (22.center) to (25.center);
		\draw (23.center) to (20);
		\draw (26) to (27.center);
		\draw (28.center) to (26);
		\draw (24.center) to (28.center);
		\draw (25.center) to (32);
		\draw (32) to (30.center);
	\end{pgfonlayer}
\end{tikzpicture}
		&&\text{(knowledge)}
                \label{eq:knowledge}\ ,\\
		&\begin{tikzpicture}
	\begin{pgfonlayer}{nodelayer}
		\node [style=square] (8) at (0, 1) {};
		\node [style=none] (9) at (0, 1) {};
		\node [style=none] (10) at (0, 1) {};
		\node [style=none] (13) at (0.5, 0) {};
		\node [style=triangle] (14) at (-0.5, -0.5) {\scriptsize{$x$}};
		\node [style=none, label={below:$C_x$}] (16) at (-0.5, -1.5) {};
		\node [style=none] (17) at (-0.5, 0) {};
		\node [style=none, label={below:$A$}] (19) at (0.5, -1.5) {};
		\node [style=none, label={below:$C_x$}] (27) at (3.5, -1.5) {};
		\node [style=none, label={below:$A$}] (30) at (4.5, -1.5) {};
		\node [style=none] (31) at (2, 0) {$=$};
		\node [style=white dot] (33) at (0, 2) {};
		\node [style=up-triangle] (35) at (4.5, -0.25) {\scriptsize{$-x$}};
		\node [style=square] (36) at (4, 1) {\scriptsize{$x,-x$}};
		\node [style=none] (37) at (3.5, 1) {};
		\node [style=none] (38) at (4.5, 1) {};
		\node [style=none] (41) at (3.5, 0) {};
		\node [style=none] (42) at (4.5, 0) {};
		\node [style=white dot] (43) at (4, 2) {};
	\end{pgfonlayer}
	\begin{pgfonlayer}{edgelayer}
		\draw [in=90, out=-60] (10.center) to (13.center);
		\draw (14) to (16.center);
		\draw (17.center) to (14);
		\draw (33) to (10.center);
		\draw (37.center) to (41.center);
		\draw (38.center) to (42.center);
		\draw (42.center) to (35);
		\draw (35) to (30.center);
		\draw (43) to (36);
		\draw (13.center) to (19.center);
		\draw [in=90, out=-120] (9.center) to (17.center);
		\draw (41.center) to (27.center);
	\end{pgfonlayer}
\end{tikzpicture}
                &&\text{(duality)}\label{eq:duality}\ ,\\
		&\begin{tikzpicture}
	\begin{pgfonlayer}{nodelayer}
		\node [style=triangle] (0) at (-2, 1.5) {\scriptsize{$x$}};
		\node [style=up-triangle] (1) at (-2, -1.5) {\scriptsize{$x$}};
		\node [style=none, label={above:$A$}] (2) at (-2, 2.5) {};
		\node [style=none, label={below:$A$}] (4) at (-2, -2.5) {};
		\node [style=none] (5) at (-2, 0) {};
		\node [style=none] (7) at (0, 0) {$=$};
		\node [style=square] (8) at (2, 1.5) {};
		\node [style=none] (9) at (2, 1.5) {};
		\node [style=none] (10) at (2, 1.5) {};
		\node [style=none, label={above:$A$}] (11) at (2, 2.5) {};
		\node [style=none] (12) at (1.5, 0.5) {};
		\node [style=none] (13) at (2.5, 0.5) {};
		\node [style=none, label={below:$A$}] (16) at (2, -2.5) {};
		\node [style=square] (17) at (2, -1.5) {};
		\node [style=none] (18) at (2, -1.5) {};
		\node [style=none] (19) at (2, -1.5) {};
		\node [style=white dot big] (28) at (2.5, -0.5) {\scriptsize{$x$}};
		\node [style=none] (29) at (1.5, -0.5) {};
	\end{pgfonlayer}
	\begin{pgfonlayer}{edgelayer}
		\draw (2.center) to (0);
		\draw (0) to (5.center);
		\draw (1) to (4.center);
		\draw (5.center) to (1);
		\draw [in=90, out=-120] (9.center) to (12.center);
		\draw [in=90, out=-60] (10.center) to (13.center);
		\draw (11.center) to (8);
		\draw (17) to (16.center);
		\draw [in=90, out=-90, looseness=0.75] (12.center) to (28);
		\draw [in=90, out=-90, looseness=0.75] (13.center) to (29.center);
		\draw [in=120, out=-90] (29.center) to (18.center);
		\draw [in=60, out=-90] (28) to (19.center);
	\end{pgfonlayer}
\end{tikzpicture}
		&&\text{(Cardy condition)}\label{eq:Cardy}\ .
	\end{align}
\end{defn}
\begin{rem}\label{rem:spin-tft-MS}
In the case when $r=2$, a $\Lambda_r$-Frobenius algebra is a $\Zb/2$-graded Frobenius algebra.
If furthermore $\Sc=\SVect$ and the grading of $C$ coincides with the grading in $\SVect$, then we recover the 
Frobenius algebras discussed in \cite[Sec.\,3.4]{Moore:2006db}, which describe open-closed (2-)spin TFTs.
This characterization first appeared in \cite{Lazaroiu:2001oc}.
\end{rem}
\begin{rem}
	A knowledgeable Frobenius algebra as defined in \cite{Lauda:2008oc}
	is the same as a knowledgeable $\Lambda_1$-Frobenius algebra.
	Note the absence of relations 
	\eqref{eq:C-morphisms-relations-twist} and \eqref{eq:C-morphisms-relations-SL2Z}
	in this case.
	\label{rem:kn-Lr-FA-comparison}
\end{rem}

\subsection{\texorpdfstring{$\Zb/r$}{Z\_r}-graded center of \texorpdfstring{$\Lambda_r$}{Lambda\_r}-Frobenius algebras}

We briefly recall the notion of the $\Zb/r$-graded center $Z^r(A)$ of a $\Lambda_r$-Frobenius algebra $A$ which satisfies that $\mu\circ\Delta$ is invertible.
Let $\Sc$ be idempotent complete,
$\zeta:=(\mu\circ\Delta)^{-1}$ and 
\begin{equation}
	\begin{tikzpicture}
	\begin{pgfonlayer}{nodelayer}
		\node [style=none] (14) at (-1, 0) {$P_x=$};
		\node [style=square] (29) at (0.5, 1.5) {};
		\node [style=none] (30) at (0.5, 1.5) {};
		\node [style=none] (31) at (0.5, 1.5) {};
		\node [style=none, label={above:$A$}] (32) at (0.5, 3.5) {};
		\node [style=none] (33) at (0, 0.5) {};
		\node [style=none] (34) at (1, 0.5) {};
		\node [style=none, label={below:$A$}] (35) at (0.5, -2.5) {};
		\node [style=square] (36) at (0.5, -1.5) {};
		\node [style=none] (37) at (0.5, -1.5) {};
		\node [style=none] (38) at (0.5, -1.5) {};
		\node [style=white dot big] (39) at (1, -0.5) {\scriptsize{$x$}};
		\node [style=none] (40) at (0, -0.5) {};
		\node [style=square] (41) at (0.5, 2.5) {$\zeta$};
	\end{pgfonlayer}
	\begin{pgfonlayer}{edgelayer}
		\draw [in=90, out=-120] (30.center) to (33.center);
		\draw [in=90, out=-60] (31.center) to (34.center);
		\draw (32.center) to (29);
		\draw (36) to (35.center);
		\draw [in=90, out=-90, looseness=0.75] (33.center) to (39);
		\draw [in=90, out=-90, looseness=0.75] (34.center) to (40.center);
		\draw [in=120, out=-90] (40.center) to (37.center);
		\draw [in=60, out=-90] (39) to (38.center);
	\end{pgfonlayer}
\end{tikzpicture}
	\label{eq:idempot-Px}
\end{equation}
for $x\in\Zb/r$, which is an idempotent and hence splits as $P_x=\tilde{\iota}_x\circ\tilde{\pi}_x$
for an object $C_x\in\Sc$ and morphisms $\tilde{\iota}_x:C_x\to A$ and $\tilde{\pi}_x:A\to C_x$.
We define the morphisms in \eqref{eq:AC-morphisms} and \eqref{eq:A-morphisms} 
to be $\iota_x:=\mu\circ\Delta\circ\tilde{\iota}_x$ and $\pi_x:=\tilde{\pi}_x$.
We furthermore introduce
\begin{equation}
	\begin{tikzpicture}
	\begin{pgfonlayer}{nodelayer}
		\node [style=none] (7) at (-6, -1) {$\mu_{x,y}=$};
		\node [style=none, label={above:$C_{x}$}] (8) at (5.5, 2.5) {};
		\node [style=none, label={above:$C_{y}$}] (9) at (7.5, 2.5) {};
		\node [style=none, label={below:$C_{w}$}] (10) at (6.5, -4) {};
		\node [style=square] (11) at (6.5, 0) {};
		\node [style=none] (12) at (5.5, 1.5) {};
		\node [style=none] (13) at (7.5, 1.5) {};
		\node [style=none] (15) at (4, -1) {$\Delta_{x,y}=$};
		\node [style=none, label={above:$C_{1}$}] (16) at (1, 2.5) {};
		\node [style=none] (17) at (-0.5, -1) {$\eta_{1}=$};
		\node [style=white dot] (18) at (1, -4) {};
		\node [style=none, label={below:$C_{-1}$}] (19) at (12, -4) {};
		\node [style=none] (20) at (10, -1) {$\varepsilon_{-1}=$};
		\node [style=white dot] (21) at (12, 2.5) {};
		\node [style=none, label={below:$C_x$}] (22) at (16, -4) {};
		\node [style=none, label={above:$C_{x}$}] (23) at (16, 2.5) {};
		\node [style=none] (24) at (14.5, -1) {$N_{x}^k=$};
		\node [style=white dot big] (26) at (16, -1) {\scriptsize{$k$}};
		\node [style=square] (27) at (-4.5, -3) {\scriptsize{$\tilde{\iota}_x$}};
		\node [style=none, label={below:$C_x$}] (29) at (-4.5, -4) {};
		\node [style=square] (30) at (-2.5, -3) {\scriptsize{$\tilde{\iota}_y$}};
		\node [style=none, label={below:$C_y$}] (32) at (-2.5, -4) {};
		\node [style=square] (33) at (-3.5, 1.5) {\scriptsize{$\tilde{\pi}_z$}};
		\node [style=none, label={above:$C_{z}$}] (34) at (-3.5, 2.5) {};
		\node [style=square] (40) at (1, 1.5) {\scriptsize{$\tilde{\pi}_1$}};
		\node [style=square] (41) at (6.5, -3) {\scriptsize{$\tilde{\iota}_w$}};
		\node [style=square] (42) at (5.5, 1.5) {\scriptsize{$\tilde{\pi}_x$}};
		\node [style=square] (43) at (7.5, 1.5) {\scriptsize{$\tilde{\pi}_y$}};
		\node [style=square] (44) at (12, -3) {\scriptsize{$\tilde{\iota}_{-1}$}};
		\node [style=square] (45) at (16, -3) {\scriptsize{$\tilde{\iota}_x$}};
		\node [style=square] (46) at (16, 1.5) {\scriptsize{$\tilde{\pi}_y$}};
		\node [style=square] (47) at (12, -1) {$\zeta$};
		\node [style=square] (48) at (6.5, -1) {};
		\node [style=square] (49) at (6.5, -2) {};
		\node [style=square] (50) at (-3.5, -1) {};
	\end{pgfonlayer}
	\begin{pgfonlayer}{edgelayer}
		\draw (16.center) to (18);
		\draw (8.center) to (12.center);
		\draw (9.center) to (13.center);
		\draw (21) to (19.center);
		\draw (23.center) to (26);
		\draw (26) to (22.center);
		\draw (27) to (29.center);
		\draw (30) to (32.center);
		\draw (34.center) to (33);
		\draw [in=135, out=-90] (42) to (11);
		\draw [in=45, out=-90] (43) to (11);
		\draw (41) to (10.center);
		\draw (49) to (41);
		\draw (11) to (48);
		\draw [bend right=60, looseness=1.25] (48) to (49);
		\draw [bend left=60, looseness=1.25] (48) to (49);
		\draw [in=90, out=-135] (50) to (27);
		\draw (33) to (50);
		\draw [in=90, out=-45] (50) to (30);
	\end{pgfonlayer}
\end{tikzpicture}
	\label{eq:graded-center-morphisms}
\end{equation}
where $z=x+y-1$ and $w=x+y+1$.
\begin{defn}
	We write $Z^r(A):=\left\{ C_x \right\}_{x\in\Zb/r}$ and call it
	together with the structure morphisms in \eqref{eq:graded-center-morphisms}, $\iota_x$ and $\pi_x$
	the $\Zb_r$-graded center algebra of $A$.
\end{defn}

\begin{rmk}
	In \cite[Rem.\,3.9]{Brunner2014:bcp} a similar structure has been observed on the state spaces of a generalized orbifold TFT, though the convention for the grading differs somewhat. 
	This is not surprising, as any Frobenius algebra $A$ carries a $\Zb$-action on itself generated by its Nakayama automorphism, 
	and if $A$ satisfies that $\mu\circ\Delta$ is invertible one can define its $\Zb$-graded center.
\end{rmk}

The following proposition follows directly from 
\cite[Sec.\,3.2]{Runkel:2018:rs}:
\begin{prop}
	The pair $(A,Z^r(A))$ together with
	structure morphisms \eqref{eq:idempot-Px}
	is a knowledgeable $\Lambda_r$-Frobenius algebra.
\end{prop}

\subsection{Examples}
We recall an example from \cite[Sec.\,4.1]{Runkel:2018:rs}.
Let us assume that $r$ is even, $k$ is a field of characteristic not 2
and let $A:=\Cl(1)=k\oplus k\theta\in\SVect$
be the Clifford algebra with $\theta$ odd and satisfying $\theta^2=1$.
With the Frobenius form $\varepsilon(1)=\frac{1}{2}$, $\varepsilon(\theta)=0$,
$A$ is a $\Lambda_r$-Frobenius algebra with $N^2=1$ and with window element 1. 
Its $\Zb/r$-graded center is given by $C_x=k\theta^{1-x}$ 
with the projection and embedding of the summand.

	\section{Topological field theories on open-closed \texorpdfstring{$r$}{r}-spin surfaces} \label{sec:TFTclass}
	We fix $\in\Rb_{\le0}$.
\subsection{A knowledgeable \texorpdfstring{$\Lambda_r$}{Lambda\_r}-Frobenius algebra in \texorpdfstring{$\Bord{r}$}{Bord\^{}r,oc}}

A set of generators of $\Bord{r}$ can be obtained by considering a set of generators of the open-closed oriented bordism category \cite{Lazaroiu:2001oc,Moore:2006db,Lauda:2008oc} and considering all $r$-spin structures on them. Every $r$-spin structure on a fixed generator of the oriented bordism category can be written as a composition of cylinders with $r$-spin structures and the generator with a single fixed $r$-spin structure. This latter observation makes the list of generators obviously shorter. In order to fix an $r$-spin structure on each generator of the oriented bordism category, we introduce a fixed marked PLCW decomposition of them, which we present now.

The generators of the closed sector (including the identity morphism) are:
\begin{align}
	&\begin{tikzpicture}
	\begin{pgfonlayer}{nodelayer}
		\node [style=none] (0) at (-1, 2) {};
		\node [style=none] (1) at (1, 2) {};
		\node [style=none] (3) at (0, 2.5) {};
		\node [style=none] (4) at (-3, -2) {};
		\node [style=none] (5) at (-1, -2) {};
		\node [style=none] (6) at (-2, -1.5) {};
		\node [style=none] (8) at (1, -2) {};
		\node [style=none] (9) at (3, -2) {};
		\node [style=none] (10) at (2, -1.5) {};
		\node [style=none] (14) at (-1.25, -0.25) {\scriptsize{$-x$}};
		\node [style=none] (15) at (0.5, -1) {\scriptsize{$-y$}};
		\node [style=none] (16) at (0.75, 0.25) {\scriptsize{$z$}};
		\node [style=none] (17) at (-0.5, -2) {\scriptsize{$0$}};
		\node [style=none] (22) at (0, 3.25) {$S^1_z$};
		\node [style=none] (23) at (-2, -3.25) {$S^1_x$};
		\node [style=none] (24) at (2, -3.25) {$S^1_y$};
		\node [style=none] (25) at (3.5, -2) {\scriptsize{$0$}};
		\node [style=none] (26) at (1.5, 2) {\scriptsize{$-1$}};
		\node [style=black dot small] (27) at (-2, -2.5) {};
		\node [style=black dot small] (28) at (2, -2.5) {};
		\node [style=black dot small] (29) at (0, 1.5) {};
		\node [style=black dot small] (30) at (0, -0.25) {};
		\node [style=none] (31) at (1, -0.75) {};
	\end{pgfonlayer}
	\begin{pgfonlayer}{edgelayer}
		\draw [in=0, out=90, looseness=0.75] (1.center) to (3.center);
		\draw [in=90, out=180, looseness=0.75] (3.center) to (0.center);
		\draw [style=black dashed, in=180, out=90, looseness=0.75] (4.center) to (6.center);
		\draw [style=black dashed, in=180, out=90, looseness=0.75] (8.center) to (10.center);
		\draw [in=-90, out=90] (4.center) to (0.center);
		\draw [in=90, out=-90, looseness=0.75] (1.center) to (9.center);
		\draw [bend left=270] (8.center) to (5.center);
		\draw [style=mid arrow, in=-90, out=0, looseness=0.75] (29) to (1.center);
		\draw [in=-90, out=180, looseness=0.75] (27) to (4.center);
		\draw [in=-90, out=180, looseness=0.75] (28) to (8.center);
		\draw [style=mid arrow, in=-90, out=0, looseness=0.75] (27) to (5.center);
		\draw [style=mid arrow, in=-90, out=0, looseness=0.75] (28) to (9.center);
		\draw [style=black dashed, in=0, out=90, looseness=0.75] (9.center) to (10.center);
		\draw [style=black dashed, in=0, out=90, looseness=0.75] (5.center) to (6.center);
		\draw [style=mid arrow, in=90, out=-165, looseness=0.75] (30) to (27);
		\draw [style=mid arrow] (30) to (29);
		\draw [in=180, out=-90, looseness=0.75] (0.center) to (29);
		\draw [style=mid arrow, in=90, out=-45, looseness=0.75] (31.center) to (28);
		\draw [style=semicircle line, in=135, out=-15] (30) to (31.center);
	\end{pgfonlayer}
\end{tikzpicture},
	&&\begin{tikzpicture}
	\begin{pgfonlayer}{nodelayer}
		\node [style=none] (0) at (-1, -2) {};
		\node [style=none] (1) at (1, -2) {};
		\node [style=none] (3) at (0, -1.5) {};
		\node [style=none] (4) at (-3, 2) {};
		\node [style=none] (5) at (-1, 2) {};
		\node [style=none] (6) at (-2, 2.5) {};
		\node [style=none] (8) at (1, 2) {};
		\node [style=none] (9) at (3, 2) {};
		\node [style=none] (10) at (2, 2.5) {};
		\node [style=none] (14) at (-1, 1) {\scriptsize{$x$}};
		\node [style=none] (15) at (1, 1) {\scriptsize{$y$}};
		\node [style=none] (16) at (-0.5, -0.5) {\scriptsize{$-w$}};
		\node [style=none] (22) at (0, -3.25) {$S^1_w$};
		\node [style=none] (23) at (-2, 3.25) {$S^1_x$};
		\node [style=none] (24) at (2, 3.25) {$S^1_y$};
		\node [style=none] (26) at (1.5, -2) {\scriptsize{$-1$}};
		\node [style=none] (27) at (3.5, 2) {\scriptsize{$0$}};
		\node [style=none] (28) at (-0.5, 2) {\scriptsize{$0$}};
		\node [style=black dot small] (29) at (-2, 1.5) {};
		\node [style=black dot small] (30) at (2, 1.5) {};
		\node [style=black dot small] (31) at (0, -2.5) {};
		\node [style=black dot small] (32) at (0, 0.25) {};
		\node [style=none] (33) at (0, -0.75) {};
	\end{pgfonlayer}
	\begin{pgfonlayer}{edgelayer}
		\draw [style=black dashed, in=0, out=90, looseness=0.75] (1.center) to (3.center);
		\draw [style=black dashed, in=90, out=-180, looseness=0.75] (3.center) to (0.center);
		\draw [in=-180, out=90, looseness=0.75] (4.center) to (6.center);
		\draw [in=-180, out=90, looseness=0.75] (8.center) to (10.center);
		\draw [in=90, out=-90] (4.center) to (0.center);
		\draw [in=-90, out=90, looseness=0.75] (1.center) to (9.center);
		\draw [bend right=270] (8.center) to (5.center);
		\draw [style=mid arrow, in=-75, out=0, looseness=0.75] (31) to (1.center);
		\draw [in=-90, out=-180, looseness=0.75] (29) to (4.center);
		\draw [in=-90, out=-180, looseness=0.75] (30) to (8.center);
		\draw [in=90, out=0, looseness=0.75] (6.center) to (5.center);
		\draw [in=90, out=0, looseness=0.75] (10.center) to (9.center);
		\draw [style=mid arrow, in=-90, out=0, looseness=0.75] (29) to (5.center);
		\draw [style=mid arrow, in=-90, out=0, looseness=0.75] (30) to (9.center);
		\draw [style=mid arrow, in=-90, out=15, looseness=0.75] (32) to (30);
		\draw [style=mid arrow, in=-90, out=165, looseness=0.75] (32) to (29);
		\draw [in=-180, out=-90, looseness=0.75] (0.center) to (31);
		\draw [style=semicircle line] (32) to (33.center);
		\draw [style=mid arrow] (33.center) to (31);
	\end{pgfonlayer}
\end{tikzpicture},
	&&\begin{tikzpicture}
	\begin{pgfonlayer}{nodelayer}
		\node [style=none] (0) at (-1, 0.75) {};
		\node [style=none] (1) at (1, 0.75) {};
		\node [style=none] (2) at (0, 0.25) {};
		\node [style=none] (22) at (0, 2) {$S^1_{1}$};
		\node [style=none] (26) at (1.5, 0.75) {\scriptsize{$0$}};
		\node [style=none] (27) at (0, -1.25) {};
		\node [style=none] (28) at (0, -2) {$\emptyset$};
		\node [style=black dot small] (29) at (0, 1.25) {};
	\end{pgfonlayer}
	\begin{pgfonlayer}{edgelayer}
		\draw [style=mid arrow, in=-90, out=0, looseness=0.75] (2.center) to (1.center);
		\draw [in=180, out=-90, looseness=0.75] (0.center) to (27.center);
		\draw [in=0, out=-90, looseness=0.75] (1.center) to (27.center);
		\draw [in=0, out=90, looseness=0.75] (1.center) to (29);
		\draw [in=90, out=180, looseness=0.75] (29) to (0.center);
		\draw [style=semicircle line, in=-90, out=180, looseness=0.75] (2.center) to (0.center);
	\end{pgfonlayer}
\end{tikzpicture},
	&&\begin{tikzpicture}
	\begin{pgfonlayer}{nodelayer}
		\node [style=none] (0) at (-1, -0.75) {};
		\node [style=none] (1) at (1, -0.75) {};
		\node [style=none] (3) at (0, -0.25) {};
		\node [style=none] (22) at (0, -2) {$S^1_{-1}$};
		\node [style=none] (26) at (1.5, -0.75) {\scriptsize{$0$}};
		\node [style=none] (27) at (0, 1.25) {};
		\node [style=none] (28) at (0, 2) {$\emptyset$};
		\node [style=black dot small] (29) at (0, -1.25) {};
	\end{pgfonlayer}
	\begin{pgfonlayer}{edgelayer}
		\draw [style=black dashed, in=0, out=90, looseness=0.75] (1.center) to (3.center);
		\draw [style=black dashed, in=90, out=-180, looseness=0.75] (3.center) to (0.center);
		\draw [in=-180, out=90, looseness=0.75] (0.center) to (27.center);
		\draw [in=0, out=90, looseness=0.75] (1.center) to (27.center);
		\draw [style=mid arrow, in=-90, out=0, looseness=0.75] (29) to (1.center);
		\draw [style=semicircle line, in=180, out=-90, looseness=0.75] (0.center) to (29);
	\end{pgfonlayer}
\end{tikzpicture},
	&&\begin{tikzpicture}
	\begin{pgfonlayer}{nodelayer}
		\node [style=none] (0) at (-1, -2) {};
		\node [style=none] (1) at (1, -2) {};
		\node [style=none] (3) at (0, -1.5) {};
		\node [style=none] (22) at (0, -3.25) {$S^1_{x}$};
		\node [style=none] (26) at (1.5, -2) {\scriptsize{$0$}};
		\node [style=none] (27) at (-1, 2) {};
		\node [style=none] (28) at (1, 2) {};
		\node [style=none] (29) at (0, 2.5) {};
		\node [style=none] (31) at (0, 3.25) {$S^1_{x}$};
		\node [style=none] (32) at (1.5, 2) {\scriptsize{$-1$}};
		\node [style=black dot small] (34) at (0, 1.5) {};
		\node [style=black dot small] (35) at (0, -2.5) {};
		\node [style=none] (36) at (0.5, -0.5) {\scriptsize{$x$}};
		\node [style=none] (37) at (2.5, 0) {$=\id_{S^1_x}$};
	\end{pgfonlayer}
	\begin{pgfonlayer}{edgelayer}
		\draw [style=black dashed, in=0, out=90, looseness=0.75] (1.center) to (3.center);
		\draw [style=black dashed, in=90, out=-180, looseness=0.75] (3.center) to (0.center);
		\draw [in=-180, out=90, looseness=0.75] (27.center) to (29.center);
		\draw (27.center) to (0.center);
		\draw (28.center) to (1.center);
		\draw [in=-90, out=-180, looseness=0.75] (34) to (27.center);
		\draw [style=mid arrow, in=-90, out=0, looseness=0.75] (35) to (1.center);
		\draw [style=semicircle line, in=-180, out=-90, looseness=0.75] (0.center) to (35);
		\draw [in=90, out=0, looseness=0.75] (29.center) to (28.center);
		\draw [style=mid arrow, in=-90, out=0, looseness=0.75] (34) to (28.center);
		\draw [style=mid arrow] (35) to (34);
	\end{pgfonlayer}
\end{tikzpicture}\ ,
	\label{eq:Bord-gen-closed}
\end{align}
where $x,y\in\Zb/r$, $z=x+y-1$ and $w=x+y+1$ are the boundary labels.
The symbols $S^1_x$, etc.\ refer to the objects in the bordism category,
e.g.\ the symbol $S^1_x~S^1_y$ refers to $X=\{1,2\}\xrightarrow{\rho}\Zb/r$ 
with $\rho_1=x$ and $\rho_2=y$.

We compute $N_x:S^1_x \to S^1_x$ using the generators:
\begin{equation}
	\input{Bord-comp-Nx.tikz}
	\label{eq:Bord-comp-Nx}
\end{equation}
In each step of the computation we have used the PLCW moves of 
Definition~\ref{defn:fixmoves} and Definition~\ref{defn:elementarymoves} 
in Section~\ref{subsec:PLCW} and the numbers refer to the numbers there.

The generators of the open sector (with the identity) are:
\begin{align}
	&\begin{tikzpicture}
	\begin{pgfonlayer}{nodelayer}
		\node [style=none] (24) at (0, 3) {$I$};
		\node [style=none] (27) at (-0.5, -1) {\scriptsize{$0$}};
		\node [style=none] (49) at (2, -1.75) {$I$};
		\node [style=none] (50) at (3.5, -1) {\scriptsize{$0$}};
		\node [style=none] (54) at (1.5, 2.25) {\scriptsize{$-1$}};
		\node [style=none] (57) at (-2, -1.75) {$I$};
		\node [style=none] (60) at (-2, -1) {};
		\node [style=black dot] (61) at (-1, 2.25) {};
		\node [style=black dot] (62) at (1, 2.25) {};
		\node [style=black dot] (63) at (-3, -1) {};
		\node [style=black dot] (64) at (-1, -1) {};
		\node [style=black dot] (65) at (1, -1) {};
		\node [style=black dot] (66) at (3, -1) {};
	\end{pgfonlayer}
	\begin{pgfonlayer}{edgelayer}
		\draw [style=thick, bend left=90, looseness=1.75] (64) to (65);
		\draw [style=mid arrow] (61) to (62);
		\draw [style=thick, in=-90, out=90, looseness=0.75] (63) to (61);
		\draw [style=thick, in=90, out=-90, looseness=0.75] (62) to (66);
		\draw [style=mid arrow] (60.center) to (64);
		\draw [style=semicircle line] (63) to (60.center);
		\draw [style=mid arrow] (65) to (66);
	\end{pgfonlayer}
\end{tikzpicture},
	&&\begin{tikzpicture}
	\begin{pgfonlayer}{nodelayer}
		\node [style=none] (24) at (0, -1.75) {$I$};
		\node [style=none] (27) at (-0.5, 2.25) {\scriptsize{$0$}};
		\node [style=none] (49) at (2, 3) {$I$};
		\node [style=none] (50) at (3.5, 2.25) {\scriptsize{$0$}};
		\node [style=none] (54) at (1.5, -1) {\scriptsize{$-1$}};
		\node [style=none] (57) at (-2, 3) {$I$};
		\node [style=none] (58) at (0, -1) {};
		\node [style=black dot] (59) at (-3, 2.25) {};
		\node [style=black dot] (60) at (-1, 2.25) {};
		\node [style=black dot] (61) at (1, 2.25) {};
		\node [style=black dot] (62) at (3, 2.25) {};
		\node [style=black dot] (63) at (-1, -1) {};
		\node [style=black dot] (64) at (1, -1) {};
	\end{pgfonlayer}
	\begin{pgfonlayer}{edgelayer}
		\draw [style=mid arrow] (59) to (60);
		\draw [style=thick, bend right=90, looseness=1.75] (60) to (61);
		\draw [style=thick, in=90, out=-90, looseness=0.75] (59) to (63);
		\draw [style=thick, in=-90, out=90, looseness=0.75] (64) to (62);
		\draw [style=mid arrow] (61) to (62);
		\draw [style=mid arrow] (58.center) to (64);
		\draw [style=semicircle line] (63) to (58.center);
	\end{pgfonlayer}
\end{tikzpicture},
	&&\begin{tikzpicture}
	\begin{pgfonlayer}{nodelayer}
		\node [style=none] (27) at (1.5, 0) {\scriptsize{$0$}};
		\node [style=none] (57) at (0, 0.75) {$I$};
		\node [style=none] (60) at (0, 0) {};
		\node [style=black dot] (63) at (-1, 0) {};
		\node [style=black dot] (64) at (1, 0) {};
	\end{pgfonlayer}
	\begin{pgfonlayer}{edgelayer}
		\draw [style=mid arrow] (60.center) to (64);
		\draw [style=thick, bend right=90, looseness=1.75] (63) to (64);
		\draw [style=semicircle line] (60.center) to (63);
	\end{pgfonlayer}
\end{tikzpicture},
	&&\begin{tikzpicture}
	\begin{pgfonlayer}{nodelayer}
		\node [style=none] (27) at (1.5, 0) {\scriptsize{$0$}};
		\node [style=none] (57) at (0, -0.75) {$I$};
		\node [style=none] (60) at (0, 0) {};
		\node [style=black dot] (63) at (-1, 0) {};
		\node [style=black dot] (64) at (1, 0) {};
	\end{pgfonlayer}
	\begin{pgfonlayer}{edgelayer}
		\draw [style=mid arrow] (60.center) to (64);
		\draw [style=semicircle line] (63) to (60.center);
		\draw [style=thick, bend left=90, looseness=1.75] (63) to (64);
	\end{pgfonlayer}
\end{tikzpicture},
	&&\begin{tikzpicture}
	\begin{pgfonlayer}{nodelayer}
		\node [style=none] (24) at (0, 2.5) {$I$};
		\node [style=none] (27) at (1.5, -1.5) {\scriptsize{$0$}};
		\node [style=none] (54) at (1.5, 1.75) {\scriptsize{$-1$}};
		\node [style=none] (57) at (0, -2.25) {$I$};
		\node [style=none] (60) at (0, -1.5) {};
		\node [style=black dot] (61) at (-1, 1.75) {};
		\node [style=black dot] (62) at (1, 1.75) {};
		\node [style=black dot] (63) at (-1, -1.5) {};
		\node [style=black dot] (64) at (1, -1.5) {};
		\node [style=none] (65) at (2, 0) {$=\id_{I}$};
	\end{pgfonlayer}
	\begin{pgfonlayer}{edgelayer}
		\draw [style=mid arrow] (61) to (62);
		\draw [style=thick, in=-90, out=90, looseness=0.75] (63) to (61);
		\draw [style=mid arrow] (60.center) to (64);
		\draw [style=semicircle line] (63) to (60.center);
		\draw [style=thick] (62) to (64);
	\end{pgfonlayer}
\end{tikzpicture}\ .
	\label{eq:Bord-gen-open}
\end{align}
Similarly to the computation in \eqref{eq:Bord-comp-Nx} one can show that
\begin{equation}
	\begin{tikzpicture}
	\begin{pgfonlayer}{nodelayer}
		\node [style=none] (24) at (0, 2.25) {$I$};
		\node [style=none] (27) at (1.5, -1.5) {\scriptsize{$0$}};
		\node [style=none] (54) at (1.5, 1.5) {\scriptsize{$-2$}};
		\node [style=none] (57) at (0, -2.25) {$I$};
		\node [style=none] (60) at (0, -1.5) {};
		\node [style=black dot] (61) at (-1, 1.5) {};
		\node [style=black dot] (62) at (1, 1.5) {};
		\node [style=black dot] (63) at (-1, -1.5) {};
		\node [style=black dot] (64) at (1, -1.5) {};
		\node [style=none] (65) at (-2, 0) {$N=$};
	\end{pgfonlayer}
	\begin{pgfonlayer}{edgelayer}
		\draw [style=mid arrow] (61) to (62);
		\draw [style=thick, in=-90, out=90, looseness=0.75] (63) to (61);
		\draw [style=mid arrow] (60.center) to (64);
		\draw [style=semicircle line] (63) to (60.center);
		\draw [style=thick] (62) to (64);
	\end{pgfonlayer}
\end{tikzpicture}\ .
	\label{eq:Bord-gen-o-N}
\end{equation}
Furthermore we need the two $r$-spin bordisms
\begin{align}
	\begin{tikzpicture}
	\begin{pgfonlayer}{nodelayer}
		\node [style=none] (16) at (-0.25, 0.75) {\scriptsize{$x$}};
		\node [style=none] (24) at (0, 3) {$I$};
		\node [style=none] (27) at (1.5, 2.25) {\scriptsize{$-1$}};
		\node [style=none] (33) at (-1, -1.75) {};
		\node [style=none] (34) at (1, -1.75) {};
		\node [style=none] (35) at (0, -1.25) {};
		\node [style=none] (36) at (0, -3) {$S^1_x$};
		\node [style=none] (37) at (1.5, -1.75) {\scriptsize{$0$}};
		\node [style=black dot small] (38) at (0, -2.25) {};
		\node [style=black dot small] (39) at (-1, 2.25) {};
		\node [style=black dot small] (40) at (1, 2.25) {};
	\end{pgfonlayer}
	\begin{pgfonlayer}{edgelayer}
		\draw [style=black dashed, in=0, out=90, looseness=0.75] (34.center) to (35.center);
		\draw [style=black dashed, in=90, out=-180, looseness=0.75] (35.center) to (33.center);
		\draw [style=mid arrow, in=-75, out=0, looseness=0.75] (38) to (34.center);
		\draw [style=semicircle line, in=-180, out=-90, looseness=0.75] (33.center) to (38);
		\draw [style=dashed black thick, bend right=60, looseness=1.50] (39) to (40);
		\draw (33.center) to (39);
		\draw (34.center) to (40);
		\draw [style=mid arrow] (39) to (40);
		\draw [style=mid arrow, in=255, out=90] (38) to (40);
	\end{pgfonlayer}
\end{tikzpicture}\quad\text{and}\quad
	\begin{tikzpicture}
	\begin{pgfonlayer}{nodelayer}
		\node [style=none] (42) at (0.25, -0.75) {\scriptsize{$x$}};
		\node [style=none] (43) at (0, -3) {$I$};
		\node [style=none] (44) at (1.5, -2.25) {\scriptsize{$0$}};
		\node [style=none] (45) at (1, 1.75) {};
		\node [style=none] (46) at (-1, 1.75) {};
		\node [style=none] (47) at (0, 2.25) {};
		\node [style=none] (48) at (0, 3) {$S^1_x$};
		\node [style=none] (49) at (1.5, 1.75) {\scriptsize{$-1$}};
		\node [style=black dot small] (50) at (0, 1.25) {};
		\node [style=black dot small] (51) at (1, -2.25) {};
		\node [style=black dot small] (52) at (-1, -2.25) {};
		\node [style=none] (54) at (0, -2.25) {};
	\end{pgfonlayer}
	\begin{pgfonlayer}{edgelayer}
		\draw [in=180, out=90, looseness=0.75] (46.center) to (47.center);
		\draw [in=90, out=0, looseness=0.75] (47.center) to (45.center);
		\draw [in=-90, out=180, looseness=0.75] (50) to (46.center);
		\draw [style=dashed black thick, bend right=60, looseness=1.50] (51) to (52);
		\draw (45.center) to (51);
		\draw (46.center) to (52);
		\draw [style=mid arrow, in=-90, out=0, looseness=0.75] (50) to (45.center);
		\draw [style=mid arrow] (54.center) to (51);
		\draw [style=semicircle line] (52) to (54.center);
		\draw [style=mid arrow, in=-90, out=75] (52) to (50);
	\end{pgfonlayer}
\end{tikzpicture}.
	\label{eq:Bord-gen-open-closed}
\end{align}

\begin{rmk}
	Note that it would be possible (by applications of PLCW moves) to alter the morphisms in \ref{eq:Bord-gen-closed}, \ref{eq:Bord-gen-open} and \ref{eq:Bord-gen-open-closed} so that the choice of marked edge is more symmetric between, for example, the pants and copants. The convention above is chosen for ease of notational conventions for objects and morphisms. Changing the marked edge would also alter, e.g., the values $x$, $y$, and $w$. 
\end{rmk}

\begin{prop}
	The objects 
	$I,S_x^1\in\Bord{r}$ $(x\in\Zb/r)$
	together with the morphisms in \eqref{eq:Bord-gen-closed}, 
	\eqref{eq:Bord-gen-open} and \eqref{eq:Bord-gen-open-closed} 
	form a knowledgeable $\Lambda_r$-Frobenius algebra.
\end{prop}
\begin{proof}
	Recall from \eqref{eq:glueing-rule-c} 
	that when glueing two boundary components, we glue the edges with
	edge labels $s_\mathrm{in}$ and $s_\mathrm{out}$,
	assign the same orientation and the edge label $s_\mathrm{in}+s_\mathrm{out}+1$.

	We first deal with the closed sector.
	The generators being $\Zb/r$-intertwiners can be seen by 
	applying deck transformations.
	Unitality and counitality
	directly follow using Corollary~\ref{cor:contract-edges-inside-disk}.
	We check associativity, proving coassociativity can be done similarly.
\begin{align}
	\input{Bord-rel-associativity.tikz}
	\label{eq:Bord-rel-associativity}
\end{align}
	We have used that we are composing with the identity and we moved the marking to the middle loops.
	After removing the middle loops we can apply 
	Corollary~\ref{cor:contract-edges-inside-disk} 
	to see that the two compositions agree.

	We check the Frobenius relation (and we omit the identity morphisms for simplicity):
\begin{align}
	\input{Bord-rel-Frobenius.tikz}
	\label{eq:Bord-rel-Frobenius}
\end{align}
Again, we can remove the middle loops apply 
	Corollary~\ref{cor:contract-edges-inside-disk} 
	to see that the three compositions agree.

	Note that 
	\begin{align}
		\begin{tikzpicture}
	\begin{pgfonlayer}{nodelayer}
		\node [style=none] (0) at (-4, -1.5) {};
		\node [style=none] (1) at (-2, -1.5) {};
		\node [style=none] (3) at (-3, -1) {};
		\node [style=none] (22) at (-3, -2.75) {$S^1_{x}$};
		\node [style=none] (26) at (-1.5, -1.5) {\scriptsize{$0$}};
		\node [style=none] (27) at (-4, 1.5) {};
		\node [style=none] (28) at (-2, 1.5) {};
		\node [style=none] (29) at (-3, 2) {};
		\node [style=none] (31) at (-3, 2.75) {$S^1_{x}$};
		\node [style=none] (32) at (-1, 1.5) {\scriptsize{$k-1$}};
		\node [style=black dot small] (34) at (-3, 1) {};
		\node [style=black dot small] (35) at (-3, -2) {};
		\node [style=none] (36) at (-2.5, 0) {\scriptsize{$x$}};
		\node [style=none] (37) at (-5.5, 0) {$N^k_x=$};
		\node [style=none] (38) at (0, 0) {and};
		\node [style=none] (42) at (5, -2.75) {$I$};
		\node [style=none] (43) at (5, -1.5) {\scriptsize{$0$}};
		\node [style=none] (47) at (5, 2.75) {$I$};
		\node [style=none] (48) at (5, 1.5) {\scriptsize{$-k-1$}};
		\node [style=none] (52) at (2.5, 0) {$N^k=$};
		\node [style=none] (54) at (5, -2) {};
		\node [style=black dot] (55) at (4, 2) {};
		\node [style=black dot] (56) at (6, 2) {};
		\node [style=black dot] (57) at (4, -2) {};
		\node [style=black dot] (58) at (6, -2) {};
	\end{pgfonlayer}
	\begin{pgfonlayer}{edgelayer}
		\draw [style=black dashed, in=0, out=90, looseness=0.75] (1.center) to (3.center);
		\draw [style=black dashed, in=90, out=-180, looseness=0.75] (3.center) to (0.center);
		\draw [in=-180, out=90, looseness=0.75] (27.center) to (29.center);
		\draw (27.center) to (0.center);
		\draw (28.center) to (1.center);
		\draw [in=-90, out=-180, looseness=0.75] (34) to (27.center);
		\draw [style=mid arrow, in=-90, out=0, looseness=0.75] (35) to (1.center);
		\draw [style=semicircle line, in=-180, out=-90, looseness=0.75] (0.center) to (35);
		\draw [in=90, out=0, looseness=0.75] (29.center) to (28.center);
		\draw [style=mid arrow, in=-90, out=0, looseness=0.75] (34) to (28.center);
		\draw [style=mid arrow] (35) to (34);
		\draw [style=mid arrow] (55) to (56);
		\draw [style=thick, in=-90, out=90, looseness=0.75] (57) to (55);
		\draw [style=mid arrow] (54.center) to (58);
		\draw [style=semicircle line] (57) to (54.center);
		\draw [style=thick] (56) to (58);
	\end{pgfonlayer}
\end{tikzpicture}\ .
		\label{eq:Bord-rel-N-power}
	\end{align}
	With this it is easy to see that
	relation \eqref{eq:C-morphisms-relations-twist} is just invariance under Dehn twists \cite[Fig.\,3.13.]{Szegedy:2018phd}.

	We check commutativity:

	\begin{align}
		\begin{aligned}
		\input{Bord-rel-commutativity.tikz}
		\end{aligned}
		\label{eq:Bord-rel-commutativity}
	\end{align}
	The labels over the equivalences correspond to the PLCW moves used.\footnote{Here, 2 and 3 are fixed PLCW moves, 4 are elementary moves, and 5 is the univalent move.}
	Finally we check relation \eqref{eq:C-morphisms-relations-SL2Z}:
	Consider the composition on the left hand side of \eqref{eq:C-morphisms-relations-SL2Z}:
	\begin{align}
		\input{Bord-rel-SL2Z.tikz}
		\label{eq:Bord-rel-SL2Z}
	\end{align}
	We have brought the marked PLCW decomposition to the standard form of
	\cite[Fig.\,3.7]{Szegedy:2018phd}, applying a Dehn twist along the loop labeled
	by $x$ gives the required result.
At this point we are done with the relations involving the closed sector only.

\medskip

The open sector has been treated in \cite{Stern:2016stft}.
One still needs to check that the morphisms in \eqref{eq:Bord-gen-open} 
represent the structure morphisms of a $\Lambda_r$-Frobenius algebra,
but we omit the computation here.

\medskip

We turn to the relations involving both the open and the closed sector.
We check ``knowledge about the center'' \eqref{eq:knowledge}:
\begin{align}
	\input{Bord-rel-center.tikz}
	\label{eq:Bord-rel-center}
\end{align}

Next we check the Cardy relations \eqref{eq:Cardy}:
\begin{align}
	\begin{aligned}
		\begin{tikzpicture}
	\begin{pgfonlayer}{nodelayer}
		\node [style=none] (42) at (-8, -2.5) {\scriptsize{$x$}};
		\node [style=none] (43) at (-8, -4.75) {$I$};
		\node [style=none] (44) at (-6.5, -4) {\scriptsize{$0$}};
		\node [style=none] (54) at (-8, -4) {};
		\node [style=none] (55) at (-8, 2.5) {\scriptsize{$x$}};
		\node [style=none] (56) at (-8, 4.75) {$I$};
		\node [style=none] (57) at (-6.5, 4) {\scriptsize{$-1$}};
		\node [style=none] (58) at (-9, 0) {};
		\node [style=none] (59) at (-7, 0) {};
		\node [style=none] (60) at (-8, 0.5) {};
		\node [style=none] (62) at (-7.5, 0) {\scriptsize{$0$}};
		\node [style=black dot small] (63) at (-8, -0.5) {};
		\node [style=none] (67) at (-4, -4.75) {$I$};
		\node [style=none] (68) at (-2.5, -4) {\scriptsize{$-x$}};
		\node [style=none] (71) at (-4, -4) {};
		\node [style=none] (72) at (-4, 1.75) {\scriptsize{$x$}};
		\node [style=none] (73) at (-4, 4.75) {$I$};
		\node [style=none] (74) at (-2.5, 4) {\scriptsize{$-1$}};
		\node [style=none] (75) at (-5, 0) {};
		\node [style=none] (76) at (-3, 0) {};
		\node [style=none] (82) at (-6, 0) {$\simeq$};
		\node [style=none] (95) at (-2, 0) {$\simeq$};
		\node [style=none] (96) at (1, -4.75) {$I$};
		\node [style=none] (97) at (2.5, -4) {\scriptsize{$0$}};
		\node [style=none] (100) at (1, -4) {};
		\node [style=none] (101) at (1, 4.75) {$I$};
		\node [style=none] (102) at (2.5, 4) {\scriptsize{$-1$}};
		\node [style=none] (107) at (3, 2) {\scriptsize{$x$}};
		\node [style=none] (108) at (-1, -2) {\scriptsize{$x$}};
		\node [style=black dot] (109) at (-9, 4) {};
		\node [style=black dot] (110) at (-7, 4) {};
		\node [style=black dot] (111) at (-5, 4) {};
		\node [style=black dot] (112) at (-3, 4) {};
		\node [style=black dot] (113) at (0, 4) {};
		\node [style=black dot] (114) at (2, 4) {};
		\node [style=black dot] (115) at (0, -4) {};
		\node [style=black dot] (116) at (2, -4) {};
		\node [style=black dot] (117) at (-1, 0) {};
		\node [style=black dot] (118) at (3, 0) {};
		\node [style=black dot] (119) at (-5, -4) {};
		\node [style=black dot] (120) at (-3, -4) {};
		\node [style=black dot] (121) at (-7, -4) {};
		\node [style=black dot] (122) at (-9, -4) {};
		\node [style=none] (123) at (4, 0) {$\simeq$};
		\node [style=none] (124) at (8, -4.75) {$I$};
		\node [style=none] (125) at (9.5, -4) {\scriptsize{$-x$}};
		\node [style=none] (126) at (8, -4) {};
		\node [style=none] (127) at (8, 4.75) {$I$};
		\node [style=none] (128) at (9.5, 4) {\scriptsize{$-1$}};
		\node [style=black dot] (129) at (7, 4) {};
		\node [style=black dot] (130) at (9, 4) {};
		\node [style=black dot] (131) at (7, -4) {};
		\node [style=black dot] (132) at (9, -4) {};
		\node [style=none] (133) at (10, 2) {\scriptsize{$x$}};
		\node [style=black dot] (134) at (9, 1.5) {};
		\node [style=black dot] (135) at (11, 1.5) {};
		\node [style=none] (136) at (7, 1.5) {};
		\node [style=none] (137) at (5, 1.5) {};
		\node [style=none] (138) at (7, -1.5) {};
		\node [style=none] (139) at (5, -1.5) {};
		\node [style=none] (140) at (11, -1.5) {};
		\node [style=none] (141) at (9, -1.5) {};
		\node [style=none] (142) at (12, 0) {$\simeq$};
		\node [style=none] (143) at (16, -4.75) {$I$};
		\node [style=none] (144) at (17.5, -4) {\scriptsize{$0$}};
		\node [style=none] (145) at (16, -4) {};
		\node [style=none] (146) at (16, 4.75) {$I$};
		\node [style=none] (147) at (17.5, 4) {\scriptsize{$-1$}};
		\node [style=black dot] (148) at (15, 4) {};
		\node [style=black dot] (149) at (17, 4) {};
		\node [style=black dot] (150) at (15, -4) {};
		\node [style=black dot] (151) at (17, -4) {};
		\node [style=none] (152) at (18, 2) {\scriptsize{$0$}};
		\node [style=black dot] (153) at (17, 1.5) {};
		\node [style=black dot] (154) at (19, 1.5) {};
		\node [style=none] (155) at (15, 1.5) {};
		\node [style=none] (156) at (13, 1.5) {};
		\node [style=none] (157) at (15, -1.5) {};
		\node [style=none] (158) at (13, -1.5) {};
		\node [style=none] (159) at (19, -1.5) {};
		\node [style=none] (160) at (17, -1.5) {};
		\node [style=none] (161) at (14, 1.5) {};
		\node [style=none] (162) at (14, -1.5) {};
		\node [style=none] (163) at (18, -1.5) {};
		\node [style=none] (164) at (14.25, -1) {\scriptsize{$0$}};
		\node [style=none] (165) at (17.75, -1) {\scriptsize{$-x$}};
		\node [style=none] (166) at (14, 2) {\scriptsize{$0$}};
	\end{pgfonlayer}
	\begin{pgfonlayer}{edgelayer}
		\draw [style=black dashed, in=0, out=90, looseness=0.75] (59.center) to (60.center);
		\draw [style=black dashed, in=90, out=-180, looseness=0.75] (60.center) to (58.center);
		\draw [style=mid arrow, in=-75, out=0, looseness=0.75] (63) to (59.center);
		\draw [style=semicircle line, in=-180, out=-90, looseness=0.75] (58.center) to (63);
		\draw [style=dashed black thick, bend right=60, looseness=1.50] (121) to (122);
		\draw [style=mid arrow] (54.center) to (121);
		\draw [style=semicircle line] (122) to (54.center);
		\draw [style=dashed black thick, bend right=60, looseness=1.50] (109) to (110);
		\draw (58.center) to (109);
		\draw (59.center) to (110);
		\draw [style=mid arrow] (109) to (110);
		\draw [style=mid arrow, in=255, out=90] (63) to (110);
		\draw (59.center) to (121);
		\draw (58.center) to (122);
		\draw [style=mid arrow, in=-90, out=75] (122) to (63);
		\draw [style=dashed black thick, bend right=60, looseness=1.50] (120) to (119);
		\draw [style=mid arrow] (71.center) to (120);
		\draw [style=semicircle line] (119) to (71.center);
		\draw [style=dashed black thick, bend right=60, looseness=1.50] (111) to (112);
		\draw (75.center) to (111);
		\draw (76.center) to (112);
		\draw [style=mid arrow] (111) to (112);
		\draw (76.center) to (120);
		\draw (75.center) to (119);
		\draw [style=mid arrow] (119) to (112);
		\draw [style=mid arrow] (100.center) to (116);
		\draw [style=semicircle line] (115) to (100.center);
		\draw [style=mid arrow] (113) to (114);
		\draw [style=mid arrow] (115) to (117);
		\draw [style=mid arrow] (118) to (114);
		\draw [style=thick] (117) to (113);
		\draw [style=thick] (116) to (118);
		\draw [style=mid arrow] (126.center) to (132);
		\draw [style=semicircle line] (131) to (126.center);
		\draw [style=mid arrow] (129) to (130);
		\draw [style=mid arrow] (134) to (135);
		\draw [style=thick, in=-90, out=90, looseness=0.75] (137.center) to (129);
		\draw [style=thick, in=90, out=-90, looseness=0.75] (130) to (135);
		\draw [style=thick, bend left=90, looseness=1.75] (136.center) to (134);
		\draw [style=thick, bend right=90, looseness=1.75] (138.center) to (141.center);
		\draw [style=thick, in=-90, out=90] (132) to (140.center);
		\draw [style=thick, in=-90, out=90] (131) to (139.center);
		\draw [style=thick, in=90, out=-90, looseness=0.50] (136.center) to (140.center);
		\draw [style=thick, in=90, out=-90, looseness=0.50] (137.center) to (141.center);
		\draw [style=thick, in=-90, out=90, looseness=0.50] (139.center) to (134);
		\draw [style=thick, in=-90, out=90, looseness=0.50] (138.center) to (135);
		\draw [style=mid arrow] (145.center) to (151);
		\draw [style=semicircle line] (150) to (145.center);
		\draw [style=mid arrow] (148) to (149);
		\draw [style=mid arrow] (153) to (154);
		\draw [style=thick, in=-90, out=90, looseness=0.75] (156.center) to (148);
		\draw [style=thick, in=90, out=-90, looseness=0.75] (149) to (154);
		\draw [style=thick, bend left=90, looseness=1.75] (155.center) to (153);
		\draw [style=thick, bend right=90, looseness=1.75] (157.center) to (160.center);
		\draw [style=thick, in=-90, out=90] (151) to (159.center);
		\draw [style=thick, in=-90, out=90] (150) to (158.center);
		\draw [style=thick, in=90, out=-90, looseness=0.50] (155.center) to (159.center);
		\draw [style=thick, in=90, out=-90, looseness=0.50] (156.center) to (160.center);
		\draw [style=thick, in=-90, out=90, looseness=0.50] (158.center) to (153);
		\draw [style=thick, in=-90, out=90, looseness=0.50] (157.center) to (154);
		\draw [style=mid arrow] (161.center) to (155.center);
		\draw [style=mid arrow] (162.center) to (157.center);
		\draw [style=mid arrow] (163.center) to (159.center);
		\draw [style=semicircle line] (156.center) to (161.center);
		\draw [style=semicircle line] (158.center) to (162.center);
		\draw [style=semicircle line] (160.center) to (163.center);
	\end{pgfonlayer}
\end{tikzpicture}
	\end{aligned}
	\label{eq:Bord-rel-Cardy}
\end{align}

Finally we check the duality condition \eqref{eq:duality}. On one hand we have:
\begin{align}
	\begin{aligned}
		\begin{tikzpicture}
	\begin{pgfonlayer}{nodelayer}
		\node [style=none] (42) at (4.25, -0.5) {\scriptsize{$x$}};
		\node [style=none] (43) at (4, -2.75) {$I$};
		\node [style=none] (44) at (5.5, -2) {\scriptsize{$0$}};
		\node [style=none] (49) at (5.5, 2) {\scriptsize{$0$}};
		\node [style=black dot small] (50) at (4, 1.5) {};
		\node [style=black dot small] (51) at (5, -2) {};
		\node [style=black dot small] (52) at (3, -2) {};
		\node [style=none] (55) at (-8.25, 0.5) {\scriptsize{$x$}};
		\node [style=none] (57) at (-6.5, 2) {\scriptsize{$0$}};
		\node [style=none] (58) at (-9, -2) {};
		\node [style=none] (59) at (-7, -2) {};
		\node [style=none] (60) at (-8, -1.5) {};
		\node [style=none] (61) at (-8, -3.25) {$S^1_x$};
		\node [style=none] (62) at (-6.5, -2) {\scriptsize{$0$}};
		\node [style=black dot small] (63) at (-8, -2.5) {};
		\node [style=black dot small] (64) at (-9, 2) {};
		\node [style=black dot small] (65) at (-7, 2) {};
		\node [style=none] (66) at (-4, -2.75) {$I$};
		\node [style=none] (67) at (-2.5, -2) {\scriptsize{$0$}};
		\node [style=black dot small] (68) at (-3, -2) {};
		\node [style=black dot small] (69) at (-5, -2) {};
		\node [style=none] (70) at (-4, -2) {};
		\node [style=black dot] (71) at (-5, 2) {};
		\node [style=black dot] (72) at (-3, 2) {};
		\node [style=black dot] (73) at (-7, 4.5) {};
		\node [style=black dot] (74) at (-5, 4.5) {};
		\node [style=none] (75) at (-4.5, 4.5) {\scriptsize{$0$}};
		\node [style=none] (76) at (-2.5, 2) {\scriptsize{$0$}};
		\node [style=none] (77) at (-6, 4.5) {};
		\node [style=none] (78) at (-8, 2) {};
		\node [style=none] (79) at (12.25, -0.5) {\scriptsize{$-x$}};
		\node [style=none] (80) at (12, -2.75) {$I$};
		\node [style=none] (81) at (13.5, -2) {\scriptsize{$0$}};
		\node [style=none] (82) at (13, 2) {};
		\node [style=none] (83) at (11, 2) {};
		\node [style=none] (86) at (13.5, 2) {\scriptsize{$0$}};
		\node [style=black dot small] (87) at (12, 1.5) {};
		\node [style=black dot small] (88) at (13, -2) {};
		\node [style=black dot small] (89) at (11, -2) {};
		\node [style=none] (90) at (12, -2) {};
		\node [style=none] (91) at (-1, -2) {};
		\node [style=none] (92) at (1, -2) {};
		\node [style=none] (93) at (0, -1.5) {};
		\node [style=none] (94) at (0, -3.25) {$S^1_x$};
		\node [style=none] (95) at (1.5, -2) {\scriptsize{$0$}};
		\node [style=black dot small] (96) at (0, -2.5) {};
		\node [style=none] (97) at (-1, 2) {};
		\node [style=none] (98) at (1, 2) {};
		\node [style=none] (99) at (0, 2.5) {};
		\node [style=none] (100) at (1.5, 2) {\scriptsize{$0$}};
		\node [style=black dot small] (101) at (0, 1.5) {};
		\node [style=none] (102) at (-0.5, 0) {\scriptsize{$x$}};
		\node [style=none] (103) at (3, 2) {};
		\node [style=none] (104) at (5, 2) {};
		\node [style=none] (105) at (4, 2.5) {};
		\node [style=none] (106) at (2, 3.5) {\scriptsize{$x$}};
		\node [style=none] (107) at (-2, 1) {$\simeq$};
		\node [style=none] (108) at (6, 1) {$\simeq$};
		\node [style=none] (109) at (7, -2) {};
		\node [style=none] (110) at (9, -2) {};
		\node [style=none] (111) at (8, -1.5) {};
		\node [style=none] (112) at (8, -3.25) {$S^1_x$};
		\node [style=none] (113) at (9.5, -2) {\scriptsize{$0$}};
		\node [style=black dot small] (114) at (8, -2.5) {};
		\node [style=none] (115) at (7, 2) {};
		\node [style=none] (116) at (9, 2) {};
		\node [style=none] (117) at (8, 2.5) {};
		\node [style=none] (118) at (9.5, 2) {\scriptsize{$0$}};
		\node [style=black dot small] (119) at (8, 1.5) {};
		\node [style=none] (120) at (7.5, 0) {\scriptsize{$x$}};
		\node [style=none] (121) at (11, 2) {};
		\node [style=none] (122) at (13, 2) {};
		\node [style=none] (123) at (12, 2.5) {};
		\node [style=none] (124) at (10, 3.5) {\scriptsize{$x$}};
	\end{pgfonlayer}
	\begin{pgfonlayer}{edgelayer}
		\draw [style=dashed black thick, bend right=60, looseness=1.50] (51) to (52);
		\draw [style=black dashed, in=0, out=90, looseness=0.75] (59.center) to (60.center);
		\draw [style=black dashed, in=90, out=-180, looseness=0.75] (60.center) to (58.center);
		\draw [style=mid arrow, in=-75, out=0, looseness=0.75] (63) to (59.center);
		\draw [style=semicircle line, in=-180, out=-90, looseness=0.75] (58.center) to (63);
		\draw [style=dashed black thick, bend right=60, looseness=1.50] (64) to (65);
		\draw (58.center) to (64);
		\draw (59.center) to (65);
		\draw [style=mid arrow, in=255, out=90] (63) to (65);
		\draw [style=mid arrow] (70.center) to (68);
		\draw [style=semicircle line] (69) to (70.center);
		\draw [style=mid arrow] (71) to (72);
		\draw [style=thick] (71) to (69);
		\draw [style=thick] (72) to (68);
		\draw [style=thick, bend left=90, looseness=1.50] (65) to (71);
		\draw [style=thick, in=-90, out=90, looseness=0.75] (64) to (73);
		\draw [style=thick, in=90, out=-90, looseness=0.75] (74) to (72);
		\draw [style=thick, bend left=90, looseness=1.75] (73) to (74);
		\draw [style=mid arrow] (77.center) to (74);
		\draw [style=semicircle line] (73) to (77.center);
		\draw [style=semicircle line] (64) to (78.center);
		\draw [style=mid arrow] (78.center) to (65);
		\draw [in=-90, out=180, looseness=0.75] (87) to (83.center);
		\draw [style=dashed black thick, bend right=60, looseness=1.50] (88) to (89);
		\draw (82.center) to (88);
		\draw (83.center) to (89);
		\draw [style=mid arrow, in=-90, out=0, looseness=0.75] (87) to (82.center);
		\draw [style=mid arrow] (90.center) to (88);
		\draw [style=semicircle line] (89) to (90.center);
		\draw [style=mid arrow, in=-90, out=75] (89) to (87);
		\draw [style=mid arrow, in=75, out=-90] (50) to (52);
		\draw [style=mid arrow] (52) to (51);
		\draw [style=black dashed, in=0, out=90, looseness=0.75] (92.center) to (93.center);
		\draw [style=black dashed, in=90, out=-180, looseness=0.75] (93.center) to (91.center);
		\draw [style=mid arrow, in=-75, out=0, looseness=0.75] (96) to (92.center);
		\draw [style=semicircle line, in=-180, out=-90, looseness=0.75] (91.center) to (96);
		\draw [style=black dashed, in=0, out=90, looseness=0.75] (98.center) to (99.center);
		\draw [style=black dashed, in=90, out=-180, looseness=0.75] (99.center) to (97.center);
		\draw [style=mid arrow, in=-75, out=0, looseness=0.75] (101) to (98.center);
		\draw [style=semicircle line, in=-180, out=-90, looseness=0.75] (97.center) to (101);
		\draw [style=mid arrow] (96) to (101);
		\draw (91.center) to (97.center);
		\draw (92.center) to (98.center);
		\draw [style=black dashed, in=0, out=90, looseness=0.75] (104.center) to (105.center);
		\draw [style=black dashed, in=90, out=-180, looseness=0.75] (105.center) to (103.center);
		\draw [style=semicircle line, in=-90, out=180, looseness=0.75] (50) to (103.center);
		\draw (104.center) to (51);
		\draw (103.center) to (52);
		\draw [bend left=90, looseness=1.75] (98.center) to (103.center);
		\draw [bend left=90, looseness=1.50] (97.center) to (104.center);
		\draw [style=mid arrow, bend left=90, looseness=2.00] (101) to (50);
		\draw [style=mid arrow, in=0, out=-75, looseness=0.75] (104.center) to (50);
		\draw [style=black dashed, in=0, out=90, looseness=0.75] (110.center) to (111.center);
		\draw [style=black dashed, in=90, out=-180, looseness=0.75] (111.center) to (109.center);
		\draw [style=mid arrow, in=-75, out=0, looseness=0.75] (114) to (110.center);
		\draw [style=semicircle line, in=-180, out=-90, looseness=0.75] (109.center) to (114);
		\draw [style=black dashed, in=0, out=90, looseness=0.75] (116.center) to (117.center);
		\draw [style=black dashed, in=90, out=-180, looseness=0.75] (117.center) to (115.center);
		\draw [style=mid arrow, in=-75, out=0, looseness=0.75] (119) to (116.center);
		\draw [style=semicircle line, in=-180, out=-90, looseness=0.75] (115.center) to (119);
		\draw [style=mid arrow] (114) to (119);
		\draw (109.center) to (115.center);
		\draw (110.center) to (116.center);
		\draw [style=black dashed, in=0, out=90, looseness=0.75] (122.center) to (123.center);
		\draw [style=black dashed, in=90, out=-180, looseness=0.75] (123.center) to (121.center);
		\draw [bend left=90, looseness=1.75] (116.center) to (121.center);
		\draw [bend left=90, looseness=1.50] (115.center) to (122.center);
		\draw [style=mid arrow, bend left=90, looseness=2.00] (119) to (87);
	\end{pgfonlayer}
\end{tikzpicture}
	\end{aligned}
	\label{eq:Bord-rel-duality}
\end{align}
Now computing $\varepsilon_1\circ\mu_{x,-x}$ we get:
\begin{align}
	\begin{aligned}
		\begin{tikzpicture}
	\begin{pgfonlayer}{nodelayer}
		\node [style=none] (82) at (12, 0) {};
		\node [style=none] (83) at (10, 0) {};
		\node [style=none] (86) at (12.5, 0) {\scriptsize{$0$}};
		\node [style=black dot small] (87) at (11, -0.5) {};
		\node [style=none] (112) at (7, -1.25) {$S^1_x$};
		\node [style=none] (115) at (6, 0) {};
		\node [style=none] (116) at (8, 0) {};
		\node [style=none] (117) at (7, 0.5) {};
		\node [style=none] (118) at (8.5, 0) {\scriptsize{$0$}};
		\node [style=black dot small] (119) at (7, -0.5) {};
		\node [style=none] (121) at (10, 0) {};
		\node [style=none] (122) at (12, 0) {};
		\node [style=none] (123) at (11, 0.5) {};
		\node [style=none] (124) at (9, 1.5) {\scriptsize{$x$}};
		\node [style=none] (125) at (11, -1.25) {$S^1_{-x}$};
		\node [style=none] (128) at (-3.5, 0) {\scriptsize{$0$}};
		\node [style=black dot small] (129) at (-5, -0.5) {};
		\node [style=none] (130) at (-9, -1.25) {$S^1_x$};
		\node [style=none] (132) at (-8, 0) {};
		\node [style=none] (133) at (-9, 0.5) {};
		\node [style=none] (134) at (-7.5, 0) {\scriptsize{$0$}};
		\node [style=none] (136) at (-6, 0) {};
		\node [style=none] (137) at (-4, 0) {};
		\node [style=none] (138) at (-5, 0.5) {};
		\node [style=none] (139) at (-5, -1.25) {$S^1_{-x}$};
		\node [style=none] (142) at (-5.5, 3) {\scriptsize{$0$}};
		\node [style=black dot small] (143) at (-7, 2.5) {};
		\node [style=none] (144) at (-8, 3) {};
		\node [style=none] (145) at (-6, 3) {};
		\node [style=none] (146) at (-7, 3.5) {};
		\node [style=black dot small] (148) at (-9, -0.5) {};
		\node [style=none] (149) at (-10, 0) {};
		\node [style=black dot] (150) at (-7, 1.5) {};
		\node [style=none] (151) at (-5.5, 0.75) {};
		\node [style=none] (152) at (-5.5, 1.25) {\scriptsize{$x$}};
		\node [style=none] (153) at (-8.5, 1.25) {\scriptsize{$-x$}};
		\node [style=none] (154) at (-6.5, 2) {\scriptsize{$1$}};
		\node [style=none] (155) at (4.5, 0) {\scriptsize{$0$}};
		\node [style=black dot small] (156) at (3, -0.5) {};
		\node [style=none] (157) at (-1, -1.25) {$S^1_x$};
		\node [style=none] (158) at (0, 0) {};
		\node [style=none] (159) at (-1, 0.5) {};
		\node [style=none] (160) at (0.5, 0) {\scriptsize{$1$}};
		\node [style=none] (161) at (2, 0) {};
		\node [style=none] (162) at (4, 0) {};
		\node [style=none] (163) at (3, 0.5) {};
		\node [style=none] (164) at (3, -1.25) {$S^1_{-x}$};
		\node [style=black dot small] (170) at (-1, -0.5) {};
		\node [style=none] (171) at (-2, 0) {};
		\node [style=black dot] (172) at (1, 1.5) {};
		\node [style=none] (173) at (2.5, 0.75) {};
		\node [style=none] (174) at (2.5, 1.25) {\scriptsize{$x$}};
		\node [style=none] (175) at (-0.5, 1.25) {\scriptsize{$-x$}};
		\node [style=none] (177) at (-3, 1.5) {$\simeq$};
		\node [style=none] (178) at (5, 1.5) {$\simeq$};
	\end{pgfonlayer}
	\begin{pgfonlayer}{edgelayer}
		\draw [in=-90, out=180, looseness=0.75] (87) to (83.center);
		\draw [style=mid arrow, in=-90, out=0, looseness=0.75] (87) to (82.center);
		\draw [style=black dashed, in=0, out=90, looseness=0.75] (116.center) to (117.center);
		\draw [style=black dashed, in=90, out=-180, looseness=0.75] (117.center) to (115.center);
		\draw [style=mid arrow, in=-75, out=0, looseness=0.75] (119) to (116.center);
		\draw [style=semicircle line, in=-180, out=-90, looseness=0.75] (115.center) to (119);
		\draw [style=black dashed, in=0, out=90, looseness=0.75] (122.center) to (123.center);
		\draw [style=black dashed, in=90, out=-180, looseness=0.75] (123.center) to (121.center);
		\draw [bend left=90, looseness=1.75] (116.center) to (121.center);
		\draw [bend left=90, looseness=1.50] (115.center) to (122.center);
		\draw [style=mid arrow, bend left=90, looseness=2.00] (119) to (87);
		\draw [style=black dashed, in=0, out=90, looseness=0.75] (132.center) to (133.center);
		\draw [style=black dashed, in=0, out=90, looseness=0.75] (137.center) to (138.center);
		\draw [style=black dashed, in=90, out=-180, looseness=0.75] (138.center) to (136.center);
		\draw [bend left=90, looseness=1.75] (132.center) to (136.center);
		\draw [style=black dashed, in=0, out=90, looseness=0.75] (145.center) to (146.center);
		\draw [style=black dashed, in=90, out=-180, looseness=0.75] (146.center) to (144.center);
		\draw [in=-90, out=180, looseness=0.75] (129) to (136.center);
		\draw [style=mid arrow, in=-90, out=0, looseness=0.75] (129) to (137.center);
		\draw [style=black dashed, in=90, out=-180, looseness=0.75] (133.center) to (149.center);
		\draw [style=mid arrow, in=-75, out=0, looseness=0.75] (148) to (132.center);
		\draw [style=mid arrow, in=-90, out=0, looseness=0.75] (143) to (145.center);
		\draw [in=-90, out=180, looseness=0.75] (148) to (149.center);
		\draw [style=mid arrow] (150) to (143);
		\draw [style=mid arrow, in=90, out=-165] (150) to (148);
		\draw [in=90, out=-90, looseness=0.75] (144.center) to (149.center);
		\draw [in=90, out=-90, looseness=0.75] (145.center) to (137.center);
		\draw [bend left=90, looseness=2.00] (144.center) to (145.center);
		\draw [style=mid arrow, in=90, out=-45] (151.center) to (129);
		\draw [style=semicircle line, in=135, out=-15, looseness=0.75] (150) to (151.center);
		\draw [style=black dashed, in=0, out=90, looseness=0.75] (158.center) to (159.center);
		\draw [style=black dashed, in=0, out=90, looseness=0.75] (162.center) to (163.center);
		\draw [style=black dashed, in=90, out=-180, looseness=0.75] (163.center) to (161.center);
		\draw [bend left=90, looseness=1.75] (158.center) to (161.center);
		\draw [in=-90, out=180, looseness=0.75] (156) to (161.center);
		\draw [style=mid arrow, in=-90, out=0, looseness=0.75] (156) to (162.center);
		\draw [style=black dashed, in=90, out=-180, looseness=0.75] (159.center) to (171.center);
		\draw [style=mid arrow, in=-75, out=0, looseness=0.75] (170) to (158.center);
		\draw [in=-90, out=180, looseness=0.75] (170) to (171.center);
		\draw [style=mid arrow, in=90, out=-165] (172) to (170);
		\draw [style=mid arrow, in=90, out=-45] (173.center) to (156);
		\draw [style=semicircle line, in=135, out=-15, looseness=0.75] (172) to (173.center);
		\draw [bend left=90, looseness=1.50] (171.center) to (162.center);
		\draw [style=semicircle line, in=180, out=-90, looseness=0.75] (144.center) to (143);
	\end{pgfonlayer}
\end{tikzpicture}
	\end{aligned}
	\label{eq:Bord-rel-pairing}
\end{align}

\end{proof}

\begin{rem}
	For the knowledgeable $\Lambda_r$-Frobenius algebra
        $(I,(S_x^1)_{x\in\Zb/r})$ 
	the relations \eqref{eq:C-morphisms-relations-twist} and
	\eqref{eq:C-morphisms-relations-SL2Z}
	are $r$-spin lifts of Dehn twists.
	The other defining relations are $r$-spin lifts of the defining relations
	of a knowledgeable Frobenius algebra, c.f.\ Remark~\ref{rem:kn-Lr-FA-comparison}.
	\label{rem:kn-Lr-FA-comparison-Bord}
\end{rem}

\subsection{Generators and relations description of \texorpdfstring{$\Bord{r}$}{Bord\^{}r,oc}}

\begin{theorem}
	The symmetric monoidal category $\Bord{r}$ is generated by the
	knowledgeable $\Lambda_r$-Frobenius algebra $(I,\left\{ S_x^1 \right\}_{x\in\Zb/r})$.
	\label{thm:Bord-r-gen-by-kn-Lr-FA}
\end{theorem}
\begin{proof}
	First we show that every $r$-spin bordism can be written as a composition of
	disjoint unions of generators in \eqref{eq:Bord-gen-closed}, \eqref{eq:Bord-gen-open} and \eqref{eq:Bord-gen-open-closed}. Let $\Sigma:\rho\to\sigma$ be an $r$-spin bordism and consider a decomposition $D$ of the underlying bordism as in \cite[Section 3.6]{Lauda:2008oc}. This decomposition 
	determines a set of circles and intervals along which we cut $\Sigma$. Looking at the isomorphism class of the $r$-spin structure near the circles (and intervals) determines
	objects in $\Bord{r}$ along which we compose disjoint unions of generators from
	\eqref{eq:Bord-gen-closed}, \eqref{eq:Bord-gen-open} and \eqref{eq:Bord-gen-open-closed} and a certain number of cylinders over circles and intervals:
	the last morphisms in \eqref{eq:Bord-gen-closed} and \eqref{eq:Bord-gen-open}.
	Restricting the $r$-spin structure on $\Sigma$ to each generator determines 
	the exact number of cylinders in the above composition via
	Proposition~\ref{prop:PLCWopenclosedisRSpin}.

	\medskip

	Next we show that for any two $r$-spin bordism representing the same morphism
	in $\Bord{r}$ there is an $r$-spin diffeomorphism built via the relations
	\eqref{eq:Bord-rel-commutativity}--\eqref{eq:Bord-rel-pairing}.
	
	Let $(\Sigma,q),(\Sigma',q'):\rho\to\sigma$ be two $r$-spin bordisms and
	let $\phi:\Sigma\to\Sigma'$ be a diffeomorphism such that $q'\simeq\phi^* q$.
	Let $\Psi:\Sigma\to\Sigma_\mathrm{std}$ be a diffeomorphism that brings
	$\Sigma$ to the standard form $\Sigma_\mathrm{std}$ described in
	\cite[Def.\,3.20]{Lauda:2008oc} and
	similarly define $\Psi':\Sigma'\to\Sigma_\mathrm{std}'=\Sigma_\mathrm{std}$.
	These define an $r$-spin diffeomorphism $\tilde{\phi}$ via the following commutative diagram:
	\begin{equation}
		\begin{tikzcd}
			(\Sigma,q)\ar{r}{\Psi}\ar{d}[swap]{\phi}&
			(\Sigma_\mathrm{std},\Psi^* q)\ar{d}{\tilde{\phi}}\\
			(\Sigma',q')\ar{r}{\Psi'}&
			(\Sigma_\mathrm{std},(\Psi')^* \phi^*q)\\
		\end{tikzcd}
		\label{eq:std-form-pres-r-spin-diffeo}
	\end{equation}
	and $\tilde{\phi}$ preserves the decomposition $\Sigma_\mathrm{std}$.
	The $r$-spin diffeomorphisms covering $\Psi$ and $\Psi'$
	are obtained from the relations 
	defining a knowledgeable $\Lambda_r$-Frobenius algebra
	as in \cite[Thm.\,3.22]{Lauda:2008oc}, since 
	those relations are $r$-spin lifts of the relations
	defining a knowledgeable Frobenius algebra,
	see Remark~\ref{rem:kn-Lr-FA-comparison-Bord}.
	So the only thing left to show is 
	that the $\tilde{\phi}$ is also obtained this way.

	It is enough to consider bordisms which have as source only circles
	and target only intervals, as any bordism can be transformed to this form
	by composing with cylinders and permutations \cite[Sec.\,3.6.2]{Lauda:2008oc}.
	Since $\tilde{\phi}$ does not change the underlying surface,
	it is a composition of ($r$-spin lifts of) Dehn twists along the loops
	shown in Figure~\ref{fig:loops}
	\begin{figure}[tb]
		\centering
		\includegraphics{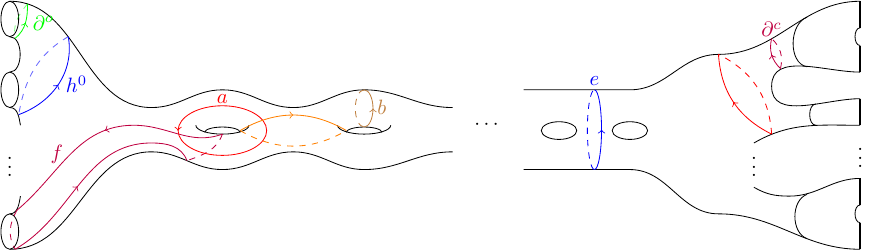}
		\caption{Loops on a bordism, along which Dehn twists generate its mapping class group.}
		\label{fig:loops}
	\end{figure}

	\begin{figure}[tb]
		\centering
		\includegraphics{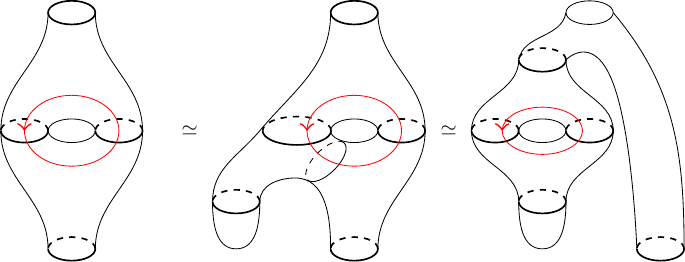}
		\caption{A diffeomorphism isotopic to the identity.}
		\label{fig:pull-out-handle}
	\end{figure}
	The Dehn twists along loops marked with 
	$b,e,\partial^\mathrm{O},\partial^\mathrm{C}$ can clearly be obtained
	from \eqref{eq:C-morphisms-relations-twist}.
	For the Dehn twists along the $d,h^\mathrm{O},h^\mathrm{C}$
	apply (co)associativity or the Frobenius relation and then 
	\eqref{eq:C-morphisms-relations-twist}.
	For the Dehn twists along the $d$ loops first consider the diffeomorphism
	in Figure~\ref{fig:pull-out-handle} which is isotopic to the identity. 
	Then apply relation \eqref{eq:C-morphisms-relations-SL2Z}.

\end{proof}
\begin{corollary}
	The symmetric monoidal categories $\Fun^{\otimes,\mathrm{symm}}\left( \Bord{r},\Sc \right)$ and $\Kn{r}{\Sc}$ are equivalent. The equivalence sends an open-closed $r$-spin topological field theory to the knowledgeable Frobenius algebra obtained by evaluation on the generators of (\ref{eq:Bord-gen-closed}) and (\ref{eq:Bord-gen-open}). 
\end{corollary}

\begin{rmk}
  The invariant assigned to an $r$-spin torus by a TFT $\Zc$ corresponding to the 
  closed $\Lambda_r$-Frobenius algebra $\{C_x\}_{x\in\Zb/r}$ is $T(d,0)=\dim(C_d)$
  for some $d\in\Zb/r$, which can be chosen to be a divisor of $r$,
  c.f.\ Proposition~\ref{prop:duals-cat-dimension}.
  Diffeomorphism classes of $r$-spin tori are in bijection with 
  the set of divisors of $r$ (see e.g.\ \cite[Sec.\,3.5]{Szegedy:2018phd}),
  and precisely the above invariant is assigned by $\Zc$ 
  to the $r$-spin torus corresponding to $d$.
  \label{rem:torus-invariant}
\end{rmk}

	\subsection{State-sum \texorpdfstring{$r$}{r}-spin topological field theories}
The state sum construction of oriented open-closed TFTs \cite{Bachas:1993lat,Fukuma:1994sts,Lauda:2007oc,Davydov:2011dt} and the state sum construction of $r$-spin TFTs of \cite{Novak:2015phd,Runkel:2018:rs} can be easily united to a construction of open-closed $r$-spin TFTs. 
The construction takes as an input a $\Lambda_r$-Frobenius algebra $A\in\Sc$
with invertible window element and produces an open-closed $r$-spin TFT
$\Zc_A:\Bord{r}\to\Sc$.
The steps of the construction are essentially the same
as in \cite{Runkel:2018:rs} with the following differences:
\begin{itemize}
	\item The object assigned to the interval is the algebra itself.
	\item At free boundary edges we compose with the counit.
	\item At open boundary vertices we do not multiply with the inverse of the window element.
\end{itemize}
The following proposition directly follows from this construction.

\begin{proposition}
	The open-closed $r$-spin TFT $\Zc_A$ of the state sum construction 
	corresponds to the knowledgeable $\Lambda_r$-Frobenius algebra
	$(A,Z^r(A))$ given by $A$ and its $\Zb/r$-graded center.
	\label{prop:state-sum-graded-center}
\end{proposition}

	\appendix
\section{Cyclic categories and computations}\label{app:Lambda-r}

\subsection{The \texorpdfstring{$r$}{r}-cyclic category}
For convenience, we here lay out the notations, definitions, and conventions we use when discussing the $r$-cyclic category. 

\begin{defn}
	For any positive integer $r$, the \emph{$r$-cyclic category} $\Lambda_r$ has objects $[n]_r$ for $n\geq 0$. The morphisms are generated by 
	\begin{align*}
	\delta^n_i:[n-1]_r & \to [n]_r\\
	\sigma^n_i:[n+1]_r & \to [n]_r\\
	\tau_n :[n]_r & \to [n]_r
	\end{align*}
	subject to the usual simplicial relations 
	\begin{eqnarray*}
		\delta_j^{n+1}\circ\delta_i^n &= &\delta_i^{n+1}\circ\delta_{j-1}^n \;\;\;\; 0\leq i<j\leq n \\
		\sigma_j^n\circ \sigma_i^{n+1} & = & \sigma_i^n\circ \sigma^{n+1}_{j+1} \;\;\;\;\;\; 0\leq i \leq j < n
	\end{eqnarray*}
	and
	\[
	\sigma_j^n\circ\delta_i^{n+1}=\begin{cases}
	\delta_i^n\circ \sigma_{j-1}^{n-1} &0\leq i<j<n\\
	\id_{[n]} & 0\leq j< n\text{ and } i=j \text{ or } i=j+1\\
	\delta_{i-1}^n\circ\sigma_j^{n-1} & 0\leq j \text{ and } j+1<i\leq n\\
	\end{cases}
	\]
	as well as the relations
	\begin{eqnarray*}
		\tau_n\circ \delta_i^n & = & \delta_{i-1}^n\circ \tau_{n-1} \;\;\;\;\; 1\leq i\leq n \\
		\tau_n\circ \delta_0^n & = & \delta_n^n \\
		\tau_n\circ \sigma_i^n & = & \sigma_{i-1}^n\circ \tau_{n+1} \;\;\;\;\; 1\leq i \leq n \\
		\tau_n\circ \sigma_0^n &=& \sigma_n^n \circ \tau_{n+1}^2 \\
		(\tau_n)^{r(n+1)} & = & 1.
	\end{eqnarray*} 
	Note that $\Lambda_1$ is precisely Connes' cyclic category, that there is a canonical full functor $\Lambda_r\to \Lambda_1$, and that there is a faithful inclusion $\Delta\to \Lambda_r$ for any $r$.  
\end{defn}

\begin{rmk}
	There is an equivalent characterization of the morphisms in $\Lambda_r$ paralleling the description of $\Lambda$ using homeomorphisms of the circle. Consider the circle 
	\[
	S^1:= \{z\in \mathbb{C} \mid |z|=1\}
	\]
	and fix an $r$-fold cover 
	\[
	p:S^1\twoheadrightarrow S^1.
	\]
	Let $[n]$ be the set 
	\[
	  \left\lbrace\exp(\frac{2\pi \mathrm{i}}{n+1}\cdot j) \mid 0\leq j\leq n\right\rbrace
	\]
	of roots of unity. The morphisms in $\Lambda_1$ from $[n]_1\to [m]_1$ are homotopy classes of orientation-preserving degree 1 maps $f:S^1\to S^1$ such that $f([n])\subseteq [m]$. Fix a set $\{f_i\}$ of homeomorphisms $S^1$ to $S^1$ representing the morphisms from $[n]_1\to [m]_1$. Then morphisms $[n]_r\to [m]_r$ can be seen as lifts 
	\[
	\begin{tikzcd}
	S^1\arrow[r,"\tilde{f}_i"]\arrow[d, "p"'] & S^1\arrow[d,"p"] \\
	S^1\arrow[r, "f_i"'] & S^1
	\end{tikzcd}
	\]
	to the $r$-fold cover. 
	
	From this picture, it is clear that there is a forgetful functor 
	\begin{align*}
	\lambda_r: \Lambda_r & \to \Set \\
	[n]_r & \mapsto \{0,1\ldots, n\}
	\end{align*}
	which sends each morphism to the underlying map on roots of unity. 
\end{rmk}

\begin{ntt}
	We fix some notation for special morphisms in $\Lambda_r$. Using the fully-faithful inclusion $\iota:\Delta \to \Lambda_r$, we consider the following morphisms in $\Delta$ as morphisms in $\Lambda_r$:  
	\begin{eqnarray*}
		\phi_n^n : [n]_r & \to & [1]_r\\
		k & \mapsto & \begin{cases}
			1 & k=n \\
			0 & \text{otherwise}
		\end{cases}\\
		\phi^0_n : [n]_r & \to & [1]_r \\
		k & \mapsto & \begin{cases}
			0 & k=0 \\
			1 & \text{otherwise} 
		\end{cases}\\
		\psi_{n}^2: [n]_r &\to &[2]_r\\
		k & \mapsto & \begin{cases}
			2 & k=n\\
			1 & k=n-1\\
			0 & \text{otherwise}.
		\end{cases}
	\end{eqnarray*}
	The specification of these morphisms using using maps of underlying sets is unambiguous because we are considering the images of morphisms in the simplex category. Note that, in terms of the generators of $\Lambda_r$, we could equivalently define.  
	\[ 
	\phi_n^n:=\sigma_0^1\circ \cdots \circ \sigma_0^{n-2}\circ\sigma_0^{n-1}
	\]
	\[
	\phi_n^0 := \sigma_1^1 \circ \cdots \circ \sigma_{n-2}^{n-2}\circ \sigma_{n-1}^{n-1} 
	\]
	\[
	\psi_n^2 := \sigma_0^2\circ \cdots \circ \sigma_0^{n-2}\circ \sigma_0^{n-1}.
	\] 
	
	We further define 
	\[
	\psi_n^{k,n} :=\phi_n^n\circ \tau_n^{k-n}
	\]
	and 
	\[
	\psi_n^{k,0}:=\phi_n^0\circ \tau^k.
	\]
	We call these morphisms \emph{augmentation morphisms} in $\Lambda_r$. 
	
	We further define \emph{pullback morphisms} ($0\leq k\leq n$) 
	\[
	\theta_n^{n,m,k}: [m+n-1]_r \to [n]_r
	\]
	in terms of generators of $\Lambda_r$ via the formula
	\[
	\theta^{n,m,k}_n:= \sigma_k^n\circ \sigma_k^{n+1}\circ \cdots \circ \sigma_k^{n+m-2}.
	\]
	These satisfy the following commutativity relations
	\begin{eqnarray*}
		\tau_n \circ \theta^{n,m,k}_n &=& \begin{cases}
			\theta_n^{n,m,k-1}\circ \tau_{n+m-1} & k\neq 0\\
			\theta_n^{n,m,n}\circ \tau_{n+m-1}^m & k=0
		\end{cases} \\
		\tau_n^{-1} \circ \theta^{n,m,k}_n &=&\begin{cases}
			\theta_n^{n,m,k+1}\circ \tau_{n+m-1}^{-1} & k\neq n\\
			\theta_n^{n,m,0}\circ \tau_{n+m-1}^{-m} & k=n
		\end{cases} 
	\end{eqnarray*}
\end{ntt}

\begin{rmk}
	The pullback morphisms are so called because the diagram 
	\[
	\begin{tikzcd}
	{[n+m-1]}_r\arrow[r, "\theta_m^{m,n,j}"]\arrow[d,"\theta_n^{n,m,k}\circ \gamma"'] & {[m]}_r\arrow[d,"\psi_m^{j,m}\circ \tau^{-(m+1)s}_m"]\\
	{[n]}_r\arrow[r,"\psi_n^{k,0}"'] & {[1]}_r
	\end{tikzcd}
	\]
	is pullback, where $\gamma=\tau_{n+m-1}^{(j-k-m)-(n+m)s}$. These are the pullbacks we use to contract edges in structured graphs. 
\end{rmk}

\begin{rmk}
	The augmentation morphisms satisfy the following commutativity relations with respect to $\tau_n$: 
	\begin{eqnarray*}
		\tau_1\circ \phi_n^n & = & \phi_n^0\circ \tau_n^n \\
		\tau_1\circ \phi_n^0 & = & \phi_n^n\circ \tau_n \\
		\tau_1^{-1}\circ\phi^n_n &=& \phi_n^0 \circ \tau_n^{-1}\\
		\tau_1^{-1}\circ \phi_n^0 &=& \phi_n^n\circ\tau^{-n}_n.\\
	\end{eqnarray*}
	These relations will form the backbone of many of our computations in $\Lambda_r$. 
\end{rmk}

\subsection{Computations in the \texorpdfstring{$r$}{r}-cyclic category}\label{app:appendixproofs}

We here provide exemplar computations verifying lemmas about $\Lambda_r$-structured graphs. 

\begin{proof}[Proof (of Proposition~\ref{prop:fixedmovesandIsos})]
	We proceed by cases.
	\begin{itemize}
		\item[Move 1.] Let $e$ be an edge of $K_\bullet$, and let $\mathscr{M}^\prime$ denote the marking obtained by reversing the edge orientation per Move 1 of Definition~\ref{defn:fixmoves}. The only difference between $(\tilde{A}^{\mathscr{M}}_\Gamma,\mu^{\mathscr{M}})$ and $(\tilde{A}^{\mathscr{M}^\prime}_\Gamma,\mu^{\mathscr{M}^\prime})$ occurs in those morphisms which have target $e$. Locally around $e$, the functor $\tilde{A}^{\mathscr{M}}_\Gamma$
		yields the diagram  
		\[
		\begin{tikzcd}
		\phantom{A} &  && && &  \phantom{A}\\
		\vdots	& {[n]}_r\arrow[rr,"\psi^{k,n}_n\circ \tau^{-(n+1)s_e}"] \arrow[ul]\arrow[dl]&& {[1]}_r && {[m]}_r\arrow[ll,"\psi^{j,0}_m"']\arrow[ur]\arrow[dr] &\vdots  \\
		\phantom{A} &  && && &  \phantom{A}
		\end{tikzcd}
		\]
		and the functor $\tilde{A}^{\mathscr{M}^\prime}_\Gamma$ yields the diagram 
		\[
		\begin{tikzcd}
		\phantom{A} &  && &&& &  \phantom{A}\\
		\vdots	& {[n]}_r\arrow[rr,"\psi^{k,0}_n"]\arrow[ul]\arrow[dl] && {[1]}_r &&& {[m]}_r\arrow[lll,"\psi^{j,0}_m\circ \tau^{-(m-1)(-s_e-1)}"']\arrow[ur]\arrow[dr] & \vdots \\
		\phantom{A} &  && &&& &  \phantom{A}
		\end{tikzcd}
		\]
		We claim that $\eta(e,\tau^{2s_e+1})$ defines an isomorphism $(\tilde{A}^{\mathscr{M}}_\Gamma,\mu^{\mathscr{M}})\to (\tilde{A}^{\mathscr{M}^\prime}_\Gamma,\mu^{\mathscr{M}^\prime})$. To see that this is, indeed the case, we simply compute 
		\begin{eqnarray*}
			\tau^{2s_e+1}\circ\psi_n^{k,n}\circ\tau^{-(n+1)s_e}&=&\tau^{2s_e+1}\circ\phi_n^{n}\circ\tau^{-(n+1)s_e+(k-n)}\\
			&=& \tau\circ\phi_n^n\circ\tau^{(k-n)}\\
			&=& \phi_n^0 \circ \tau^n\circ \tau^{k-n}\\
			&=&\phi_n^0\circ\tau^k=\psi_n^{k,0}
		\end{eqnarray*} 
		and
		\begin{eqnarray*}
			\tau^{2s_e+1}\circ \psi_m^{j,0} &=& \tau^{2s_e+1}\circ \phi_m^0\circ \tau^j\\
			&=& \tau\circ \phi_m^0 \circ \tau^{(m+1)s_e+j} \\
			&=& \phi_m^m \circ \tau \circ \tau^{(m+1)s_e+j}\\
			&=& \phi_m^m \circ \tau^{m+1-m}\circ \tau^{(m+1)s_e+j}\\
			&=& \phi_m^m \circ \tau^{(m+1)(s_e+1)+j-m}\\
			&=& \phi_m^m \circ \tau^{j-m}\circ \tau^{-(m+1)(-s_e-1)}\\
			&=&\psi_m^{j,m}\circ \tau^{-(m+1)(-s_e-1)}
		\end{eqnarray*} 
		which shows that the target of $\eta(e,\tau^{2s_e+1})$ is indeed $(\tilde{A}^{\mathscr{M}^\prime}_\Gamma,\mu^{\mathscr{M}^\prime})$, as desired. 
	\item[Move 2.] In light of our proof for Move 1, it will suffice to confirm Move 2b, as in the image
		\begin{center}
			\begin{tikzpicture}
	\begin{pgfonlayer}{nodelayer}
		\node [style=none] (45) at (-5.5, -1.5) {};
		\node [style=none] (46) at (-2.5, -1.5) {};
		\node [style=none] (47) at (-6.5, 1.5) {};
		\node [style=none] (48) at (-1.5, 1.5) {};
		\node [style=none] (49) at (-4, 3.5) {};
		\node [style=none] (50) at (2.5, -1.5) {};
		\node [style=none] (51) at (5.5, -1.5) {};
		\node [style=none] (52) at (1.5, 1.5) {};
		\node [style=none] (53) at (6.5, 1.5) {};
		\node [style=none] (54) at (4, 3.5) {};
		\node [style=none] (55) at (-4, -1) {$e$};
		\node [style=none] (56) at (-4, -2) {$s_e$};
		\node [style=none] (57) at (4, -1) {$e$};
		\node [style=none] (58) at (4, -2) {$s_e+1$};
		\node [style=none] (59) at (0, 0) {$\mapsto$};
		\node [style=none] (65) at (-4, -1.5) {};
	\end{pgfonlayer}
	\begin{pgfonlayer}{edgelayer}
		\draw (47.center) to (45.center);
		\draw (47.center) to (49.center);
		\draw (49.center) to (48.center);
		\draw (48.center) to (46.center);
		\draw (52.center) to (50.center);
		\draw (52.center) to (54.center);
		\draw (54.center) to (53.center);
		\draw [style=mid arrow] (51.center) to (50.center);
		\draw [style=mid arrow] (65.center) to (45.center);
		\draw [style=semicircle line] (65.center) to (46.center);
		\draw [style=semicircle line] (51.center) to (53.center);
	\end{pgfonlayer}
\end{tikzpicture}
		\end{center}
		We will denote by $\mathscr{M}$ the marking on the left, and by $\mathscr{M}^\prime$ that on the right. Here, the PLCW decomposition only changes around the vertex $v$ corresponding the the pictured polygon. Locally around $v$, the functors $\tilde{A}_\Gamma^{\mathscr{M}}$ and $\tilde{A}_\Gamma^{\mathscr{M}}$ look like
		\[
		\begin{tikzcd}
		\phantom{A} &\cdots & \phantom{A} \\
		\phantom{A} & {[n]}_r\arrow[d,"\psi^{0,n}_n\circ \tau^{-(n+1)s_e}"]\arrow[ur]\arrow[ul]\arrow[r]\arrow[l] & \phantom{A}\\
		&{[1]}_r &
		\end{tikzcd}\quad \text{and}\quad \begin{tikzcd}
		\phantom{A} &\cdots & \phantom{A} \\
		\phantom{A} & {[n]}_r\arrow[d,"\psi^{n,n}_n\circ \tau^{-(n+1)(s_e+1)}"]\arrow[ur]\arrow[ul]\arrow[r]\arrow[l] & \phantom{A}\\
		&{[1]}_r &
		\end{tikzcd}
		\]
		respectively.	
	
		We claim that $\eta(v,\tau^{-1})$ provides an isomorphism $(\tilde{A}^{\mathscr{M}}_\Gamma,\mu^{\mathscr{M}})\to (\tilde{A}^{\mathscr{M}^\prime}_\Gamma,\mu^{\mathscr{M}^\prime})$. To see this, we first compute 
		\begin{eqnarray*}
			\psi_n^{0,n}\circ \tau^{-(n+1)s_e}\circ \tau^{-1}
			&=&\phi_n^n\circ \tau^{-n-1}\circ \tau^{-(n+1)s_e}\\
			&=&\phi_n^n\circ \tau^{-(n+1)(s_e+1)}\\
			&=&\psi^{n,n}_n\circ \tau^{-(n+1)(s_e+1)}
		\end{eqnarray*} 
		We then compute that, for $k\neq 0$ 
		\begin{eqnarray*}
			\psi^{k,n}_n \circ \tau^{-(n+1)s}\circ \tau^{-1} &= & \phi_n^n\circ \tau^{k-n-1}\circ \tau^{-(n+1)s}\\
			&=&\psi^{k-1,n}_n \circ \tau^{-(n+1)s}
		\end{eqnarray*}
		showing that the target of $\eta(v,\tau^{-1})$ is indeed $(\tilde{A}^{\mathscr{M}^\prime}_\Gamma,\mu^{\mathscr{M}^\prime})$.
		\item[Move 3.] Deck transformations can be obtained as iterated applications of Move 2, so we see by the previous argument that deck transformations correspond to $\eta(v,\tau^{-(n+1)k})$. 
	\end{itemize} 
	We therefore get a well-defined map 
	\[
	\xi: \mathcal{M}^K(S,M)_{/\sim_{\on{fix}}} \to \Lambda_r(\Gamma_K)_{\sim_{\on{iso}}}.
	\]
	
	However, it is clear  that every $\Lambda_r$-structure on $\Gamma$ is isomorphic to one in the image of $\xi$, so it remains only to show that an isomorphism between $\Lambda_r$ structures lying in the image of $\xi$ implies that the corresponding markings on $K_\bullet$ are related by a sequence of PLCW moves. However, by Lemma~\ref{lem:factoringgraphisos} and the fact that $\tau$ generates the automorphism group of $[n]_n$, it will therefore suffice to show that this is true for $\eta(e,\tau)$, $\eta(e,\tau^{-1})$, $\eta(v,\tau)$, and $\eta(v,\tau^{-1})$ for all edges $e$ and vertices $v$ of $\Gamma$. This, however, is immediate from our proof for Move 1 (since $\tau^{2s_e+1}=\tau$) and our proof of Move 2.  
\end{proof}

	\phantomsection

\addcontentsline{toc}{section}{References}

\newpage 
\begin{thebibliography}{DKR11}

\bibitem[Abr96]{Abrams:1996ty}
 L.S. Abrams, {\em Two-dimensional topological quantum field theories and
  Frobenius algebras}.
\newblock \href{http://dx.doi.org/10.1142/S0218216596000333}{J. Knot Theor.
  Ramif. {\bfseries 5} (1996) 569--587}.

\bibitem[BP93]{Bachas:1993lat}
 C.~Bachas and P.M.S. Petropoulos, {\em {Topological models on the lattice and
  a remark on string theory cloning}}.
\newblock \href{http://dx.doi.org/10.1007/BF02097063}{Comm. Math. Phys.
  {\bfseries 152} (1993) 191--202},
  \href{http://arxiv.org/abs/hep-th/9205031}{{\ttfamily [hep-th/9205031]}}.
  
\bibitem[BCP14]{Brunner2014:bcp} I. Brunner, N. Carqueville, and D. Plencner  {\em {Orbifolds and Topological Defects.}} \newblock \href{https://doi.org/10.1007/s00220-014-2056-3}{Comm. Math. Phys. {\bfseries 332}, (2014) 669--712},
	\href{http://arxiv.org/abs/1307.3141}{{\ttfamily [1307.3141 [hep-th]]}}.

\bibitem[Dij89]{Dijkgraaf:1989phd}
 R.H. Dijkgraaf. {\em {A geometrical approach to two-dimensional Conformal
  Field Theory}}. PhD thesis, Utrecht University, 1989,
  \url{https://dspace.library.uu.nl/handle/1874/210872}.

\bibitem[DK15]{Dyckerhoff:2015csg}
 T.~Dyckerhoff and M.~Kapranov. {\em {Crossed simplicial groups and structured
  surfaces}}. \href{http://dx.doi.org/10.1090/conm/643}{In T.~Pantev,
  C.~Simpson, B.~To{\"e}n, M.~Vaqui{\'e}, and G.~Vezzosi, editors, {\em {Stacks
  and Categories in Geometry, Topology, and Algebra}}},
  \href{http://dx.doi.org/10.1090/conm/643}{volume 643}.
  \href{http://dx.doi.org/10.1090/conm/643}{AMS},
  \href{http://dx.doi.org/10.1090/conm/643}{2015},
  \href{http://arxiv.org/abs/1403.5799}{{\ttfamily [1403.5799 [math.AT]]}}.

\bibitem[DKR11]{Davydov:2011dt}
 A.~Davydov, L.~Kong, and I.~Runkel. {\em Field theories with defects and the
  centre functor}. In H.~Sati and U.~Schreiber, editors, {\em Mathematical
  Foundations of Quantum Field and Perturbative String Theory}. AMS, 2011,
  \href{http://arxiv.org/abs/1107.0495}{{\ttfamily [1107.0495 [math.QA]]}}.

\bibitem[FHK94]{Fukuma:1994sts}
 M.~Fukuma, S.~Hosono, and H.~Kawai, {\em {Lattice topological field theory in
  two dimensions}}.
\newblock \href{http://dx.doi.org/10.1007/BF02099416}{Comm. Math. Phys.
  {\bfseries 161} (1994) 157--175},
  \href{http://arxiv.org/abs/hep-th/9212154}{{\ttfamily [hep-th/9212154]}}.

\bibitem[FS08]{Fuchs:2008fa}
 J.~{Fuchs} and C.~{Stigner}, {\em {On Frobenius algebras in rigid monoidal
  categories}}.
\newblock Arab. J. Sci. Eng. {\bfseries 33-2C} (2008) 175--191,
  \href{http://arxiv.org/abs/0901.4886}{{\ttfamily [0901.4886 [math.CT]]}}.

\bibitem[Kir12]{Kirillov:2012pl}
 A.~Kirillov, Jr., {\em On piecewise linear cell decompositions}.
\newblock \href{http://dx.doi.org/10.2140/agt.2012.12.95}{Algebr. Geom. Topol.
  {\bfseries 12} (2012) 95--108},
  \href{http://arxiv.org/abs/1009.4227}{{\ttfamily [1009.4227 [math.GT]]}}.

\bibitem[Koc04]{Kock:2004fa}
 J.~Kock. {\em Frobenius Algebras and 2D Topological Quantum Field Theories}.
  \href{http://dx.doi.org/10.1017/CBO9780511615443}{Cambridge University
  Press}, \href{http://dx.doi.org/10.1017/CBO9780511615443}{(2004)}.

\bibitem[LP07]{Lauda:2007oc}
 A.D. Lauda and H.~Pfeiffer, {\em {State sum construction of two-dimensional
  open-closed Topological Quantum Field Theories}}.
\newblock \href{http://dx.doi.org/10.1142/S0218216507005725}{J. Knot Theor.
  Ramif. {\bfseries 16} (2007) 1121--1163},
  \href{http://arxiv.org/abs/math/0602047}{{\ttfamily [math/0602047
  [math.QA]]}}.

\bibitem[LP08]{Lauda:2008oc}
 A.D. Lauda and H.~Pfeiffer, {\em {Open-closed strings: Two-dimensional
  extended TQFTs and Frobenius algebras}}.
\newblock \href{http://dx.doi.org/10.1016/j.topol.2007.11.005}{Topology Appl.
  {\bfseries 155} (2008) 623--666},
  \href{http://arxiv.org/abs/math/0510664}{{\ttfamily [math/0510664
  [math.AT]]}}.

\bibitem[Laz01]{Lazaroiu:2001oc}
C.I. Lazaroiu, {\em On the structure of open-closed topological field theory
	in two dimensions}.
\newblock \href{http://dx.doi.org/10.1016/S0550-3213(01)00135-3}{Nucl.Phys.
	{\bfseries B603} (2001) 497--530},
\href{http://arxiv.org/abs/hep-th/0010269}{{\ttfamily [hep-th/0010269
		[hep-th]]}}.

\bibitem[MS09]{Moore:2006db}
 G.W. Moore and G.~Segal. {\em {D-branes and K-theory in 2D topological field
  theory}}. In P.S. Aspinwall, T.~Bridgeland, A.~Craw, M.R. Douglas, M.~Gross,
  A.~Kapustin, G.W. Moore, G.~Segal, B.~Szendr\H{o}i, and P.M.H. Wilson,
  editors, {\em {Dirichlet Branes and Mirror Symmetry}}. AMS, (2009).

\bibitem[NR15]{Novak:2015NR}
	S.~Novak and I. Runkel. {\em {State sum construction of two-dimensional topological quantum field theories on spin surfaces}}. \href{https://doi.org/10.1142/S0218216515500285}{J. Knot Th. and its Ramifications, Vol. 24, No. 5 (2015)}

\bibitem[Nov15]{Novak:2015phd}
 S.~Novak. {\em {Lattice topological field theories in two dimensions}}. PhD
  thesis, Universit\"at Hamburg, 2015,
  \url{http://ediss.sub.uni-hamburg.de/volltexte/2015/7527}.

\bibitem[RS21]{Runkel:2018:rs}
 I.~Runkel and L.~Szegedy, {\em {Topological field theory on r-spin surfaces
  and the Arf invariant}}.
  \newblock \href{http://dx.doi.org/10.1063/5.0037826}{J. Math. Phys.
	{\bfseries 62} (2021) 102302},
\newblock \href{http://arxiv.org/abs/1802.09978}{{\ttfamily [1802.09978
[math.QA]]}}.

\bibitem[Ste16]{Stern:2016stft}
 W.H. Stern, {\em {Structured Topological Field Theories via Crossed Simplicial
  Groups}}.
\newblock \href{http://arxiv.org/abs/1603.02614}{{\ttfamily 1603.02614
  [math.CT]}}.

\bibitem[Ste19]{Stern:2019phd}
 W.H. Stern. {\em {Open Topological Field Theories and 2-Segal Objects}}. PhD
  thesis, Universit\"at Bonn, 2019,
  \url{http://hss.ulb.uni-bonn.de/2019/5591/5591.htm}.

\bibitem[Sze18]{Szegedy:2018phd}
 L.~Szegedy. {\em {}State-sum construction of two-dimensional functorial field theories}. PhD
  thesis, Universit\"at Hamburg, 2018,
  \url{https://ediss.sub.uni-hamburg.de/handle/ediss/7848}.


\end{thebibliography}
	\end{document}